\documentclass[smallextended,final]{svjour3}
\usepackage[letterpaper,top=2in, bottom=1.5in, left=1in, right=1in]{geometry}

\usepackage{epsfig,epsf,fancybox}
\usepackage{amsmath}
\usepackage{mathrsfs}
\usepackage{amssymb}
\usepackage{amsfonts}
\usepackage{graphicx}
\usepackage{color}
\usepackage[linktocpage,pagebackref,colorlinks,linkcolor=blue,anchorcolor=blue,citecolor=blue,urlcolor=blue,hypertexnames=false]{hyperref}
\usepackage{boxedminipage}
\usepackage{stmaryrd}
\usepackage{multirow}
\usepackage{booktabs}
\usepackage{algorithm}
\usepackage{algpseudocode}
\usepackage{cite}
\usepackage{float}

\usepackage{xcolor}
\usepackage{tcolorbox}
\usepackage{braket,mathdots}
\usepackage{mathtools}
\usepackage[thinlines]{easytable}
\usepackage{array}

\allowdisplaybreaks[4]
\usepackage{mathtools}

\newcommand{\vnu}{{\boldsymbol{\nu}}}
\newcommand{\vgamma}{{\boldsymbol{\gamma}}}
\newcommand{\vxi}{{\boldsymbol{\xi}}}
\newcommand{\vzeta}{{\boldsymbol{\zeta}}}

\usepackage{mathtools}

\usepackage[normalem]{ulem} 

\newcommand{\vb}{{\mathbf{b}}}

\newcommand{\vd}{{\mathbf{d}}}

\newcommand{\vu}{{\mathbf{u}}}
\newcommand{\vv}{{\mathbf{v}}}

\newcommand{\vx}{{\mathbf{x}}}
\newcommand{\vy}{{\mathbf{y}}}
\newcommand{\vz}{{\mathbf{z}}}

\newcommand{\vA}{{\mathbf{A}}}
\newcommand{\vB}{{\mathbf{B}}}
\newcommand{\vC}{{\mathbf{C}}}

\newcommand{\vH}{{\mathbf{H}}}
\newcommand{\vI}{{\mathbf{I}}}
\newcommand{\vJ}{{\mathbf{J}}}

\newcommand{\cL}{{\mathcal{L}}}

\newcommand{\cP}{{\mathcal{P}}}

\newcommand{\cS}{{\mathcal{S}}}

\newcommand{\cX}{{\mathcal{X}}}

\newcommand{\RR}{\mathbb{R}} 
\newcommand{\zz}{^{\top}} 
\newcommand{\vzero}{\mathbf{0}} 

\newcommand{\dist}{\mathrm{dist}}    
\newcommand{\prox}{{\mathbf{prox}}} 
\newcommand{\dom}{{\mathrm{dom}}} 

\newcommand{\st}{\mbox{ s.t. }}

\DeclareMathOperator*{\argmin}{arg\,min} 
\DeclareMathOperator*{\Argmin}{Arg\,min} 

\newcommand{\bc}{\begin{center}}
\newcommand{\ec}{\end{center}}

\newcommand{\bdm}{\begin{displaymath}}
\newcommand{\edm}{\end{displaymath}}

\newcommand{\beq}{\begin{equation}}
\newcommand{\eeq}{\end{equation}}

\newcommand{\bfl}{\begin{flushleft}}
\newcommand{\efl}{\end{flushleft}}

\newcommand{\bt}{\begin{tabbing}}
\newcommand{\et}{\end{tabbing}}

\newcommand{\beqn}{\begin{eqnarray}}
\newcommand{\eeqn}{\end{eqnarray}}

\newcommand{\beqs}{\begin{align*}} 
\newcommand{\eeqs}{\end{align*}}  

\newtheorem{assumption}{Assumption}

\numberwithin{equation}{section}
\numberwithin{corollary}{section}
\numberwithin{theorem}{section}
\numberwithin{lemma}{section}
\numberwithin{proposition}{section}
\numberwithin{remark}{section}

\begin{document}
	
	\title{First-order Methods for Affinely Constrained Composite Non-convex Non-smooth Problems: Lower Complexity Bound and Near-optimal Methods}
	
	\author{Wei Liu\and Qihang Lin \and Yangyang Xu}
	
	\institute{W. Liu, Y. Xu \at Department of Mathematical Sciences, Rensselaer Polytechnic Institute, Troy, NY\\
		\email{\{liuw16, xuy21\}@rpi.edu}\\
		Q. Lin \at Department of Business Analytics, University of Iowa, Iowa City\\ \email{qihanglin@uiowa.edu} 
		}
	
	\date{\today}
		\maketitle
	
	
		
	\begin{abstract}
Many recent studies on first-order methods (FOMs) focus on \emph{composite non-convex non-smooth} optimization with linear and/or nonlinear function constraints. Upper (or worst-case) complexity bounds have been established for these methods. However, little can be claimed about their optimality as no lower bound is known, except for a few special \emph{smooth non-convex} cases. In this paper, we make the first attempt to establish lower complexity bounds of FOMs for solving a class of composite non-convex non-smooth optimization with linear constraints. Assuming two different first-order oracles, we establish lower complexity bounds of FOMs to produce a (near) $\epsilon$-stationary point of a problem (and its reformulation) in the considered problem class, for any given tolerance $\epsilon>0$.  
In addition, we present an inexact proximal gradient (IPG) method  
by using the more relaxed one of the two assumed first-order oracles. The oracle complexity of the proposed IPG, to find a (near) $\epsilon$-stationary point of the considered problem and its reformulation, matches our established lower bounds up to a logarithmic factor.  
	Therefore, our lower complexity bounds and the proposed IPG method are almost non-improvable.
	\end{abstract}
	
	\vspace{0.3cm}
	


	\section{Introduction}

	
	First-order methods (FOMs) have attracted increasing attention because of their efficiency in solving large-scale problems arising from machine learning and other areas. 
	The recent studies on FOMs have focused on non-convex problems, and one of the actively studied topics is the oracle complexity for finding a near-stationary point under various assumptions. In this paper, we explore this topic for problems with a composite non-convex non-smooth objective function and linear equality constraints, formulated as 
	\begin{equation}\label{eq:model}
		\begin{aligned}
			\min_{ \vx\in\RR^d} \,\,& F_0( \vx):= f_0( \vx) + g( \vx), \ \st  \vA \vx +  \vb=\mathbf{0}.
		\end{aligned}
	\tag{P}
	\end{equation}
	Here,  $\vA\in\RR^{n\times d}$, $\vb\in\RR^{n}$, $f_0:\mathbb{R}^d \rightarrow \mathbb{R}$ is smooth, and $g: \mathbb{R}^d \rightarrow \mathbb{R} \cup\{+\infty\}$ is a proper lower semicontinuous convex function but potentially non-smooth. 
	We assume the following  properties throughout this paper.
	
	\begin{assumption}The following statements hold.
		\label{assume:problemsetup}
		\begin{itemize}
			\item[\textnormal{(a)}]  $\nabla f_0$ is $L_f$-Lipschitz continuous,  i.e., 
			$
			\left\|\nabla f_0\left(\vx\right)- \nabla f_0\left(\vx'\right)\right\| \leq L_f\left\|\vx-\vx'\right\|,\quad \forall\, \vx,\,\vx'\in \mathbb{R}^d.
			$ 
			\item[\textnormal{(b)}] $\inf_{\vx} F_0(\vx) > -\infty$.
			\item[\textnormal{(c)}]	$
			g(\vx)= \bar{g}(\bar{\vA} \vx +\bar{\vb}),
			$
			where $\bar{\vA}\in \RR^{\bar{n}\times d}$, $\bar{\vb}\in\RR^{\bar{n}}$ and $\bar{g}: \mathbb{R}^{\bar n} \rightarrow \mathbb{R} \cup\{+\infty\}$ is a proper lower semicontinuous convex function but potentially non-smooth.
		\end{itemize}
			\end{assumption}

	Due to non-convexity, computing or even approximating a global optimal solution for problem~\eqref{eq:model} is intractable in general~\cite{murty1985some}. Hence, we focus on using an FOM to find a (near) $\epsilon$-stationary point of problem~\eqref{eq:model} for a given tolerance $\epsilon > 0$; see Definition~\ref{def:eps-pt-P}.  An FOM finds a (near) $\epsilon$-stationary point by querying information from some oracles, which typically dominates the runtime of the method, so its efficiency can be measured by the number of oracles it queries, 
	which is defined as the method's \emph{oracle complexity}. The goal of this paper is to establish a lower bound for the oracle complexity of a class of FOMs to find a (near) $\epsilon$-stationary point of problem~\eqref{eq:model} and, additionally, to present an FOM that can nearly achieve this lower complexity bound. 
 
Undoubtedly, an algorithm's oracle complexity depends on what oracle information the algorithm can utilize and what operations it can perform by using the oracle. 
Since we focus on FOMs, we assume that a first-order oracle is accessible and each generated iterate is a certain combination of the oracle information. More specifically, we make the following assumption that is standard for an FOM to solve problem~\eqref{eq:model}. 
	\begin{assumption}
		\label{ass:linearspan}
		There is an oracle such that given any $\vx$ and $\vxi$ in $\RR^d$, it can return $(\vA\zz\vb, \nabla f_0(\vx), \vA\zz \vA\vx )$	and $\prox_{\eta g}(\vxi)$ for any $\eta>0$.
		In addition, the sequence $\left\{\vx^{(t)}\right\}_{t=0}^{\infty}$ generated by the underlying algorithm satisfies that,  for any $t\geq 1$, $\vx^{(t)} \in  
		\mathbf{span}\left( \big\{\vxi^{(t)},\vzeta^{(t)} \big\} \right)$, where
		\begin{equation*}
			\begin{aligned}		
				&\vxi^{(t)}\in \mathbf{span}\left(\left\{ \vA\zz\vb\right\}\bigcup\cup_{s=0}^{t-1}\left\{\vx^{(s)}, \nabla f_0(\vx^{(s)}), \vA\zz \vA\vx^{(s)}\right\}\right),\ 
				\vzeta^{(t)}\in \left\{\prox_{\eta g}(\vxi^{(t)})~|~\eta>0\right\}.
			\end{aligned}		
		\end{equation*}
	\end{assumption}
	Here, $\mathbf{span}(\mathcal{S})$ represents  the set of all linear combinations of finitely many vectors in  a set $\mathcal{S}$. Under the linear-span assumption above, an algorithm can access $\nabla f_0$ at any historical solutions and can compute matrix-vector multiplications with $\vA$ and $\vA^{\top}$ as well as the proximal mapping of $g$, i.e.,  
\begin{eqnarray}
	\label{eq:proxg}
	\prox_{\eta g}(\vx):=\arg\min_{ \vx'} \left\{ \textstyle g(\vx')+\frac{1}{2\eta}\|\vx'-\vx\|^2 \right\}
\end{eqnarray}
for any $\vx\in\mathbb{R}^d$ and $\eta>0$. We say  $\vx^{(t)}$ is generated by the \emph{$t$-th iteration}\footnote{Without further specification, one iteration of an FOM will call the oracle once. We will specify the type of iteration if this does not hold, e.g., by outer- or inner-iteration.} of the algorithm. 


It should be noted that the above linear-span assumption allows the algorithm to compute the proximal mapping of $g$ 
but does not permit the projection onto the affine set $\{\vx\in\mathbb{R}^d~|~\vA\vx +  \vb=\mathbf{0}\}$. 
It is also worth noting here that computing the proximal mapping of $g$ is more difficult than computing a subgradient of $g$. In fact, if we disallow computing the proximal mapping of $g$ but instead give the algorithm access to its subgradients, the algorithm is limited to the class of subgradient methods for non-smooth optimization, in which case finding an $\epsilon$-stationary point is  impossible with finite oracle complexity~\cite{zhang2020complexity}. 

FOMs under Assumption~\ref{ass:linearspan} have been developed with theoretically proved oracle complexity for finding an $\epsilon$-stationary point of problem~\eqref{eq:model} under different settings. These results are stated as  \emph{upper bounds} for the maximum number of oracles those algorithms require to reach an $\epsilon$-stationary point. In two special cases of problem~\eqref{eq:model}, some existing algorithms' oracle complexity is known to be non-improvable (also called optimal) because the complexity matches, up to constant factors, a theoretical \emph{lower bound} of oracle complexity~\cite{nemirovskij1983problem}, which is the minimum number of oracles an algorithm in a class 
needs to find an 
$\epsilon$-stationary point. These two cases are as follows.
\begin{enumerate} 
	\item When the linear constraint $\vA \vx +  \vb=\mathbf{0}$ does not exist,  problem~\eqref{eq:model} becomes a composite non-smooth non-convex optimization problem 
	\begin{equation}
		\label{eq:composite}
	\min_{\vx}f_0(\vx)+g(\vx), 
	\end{equation}
	for which the proximal gradient method finds an $\epsilon$-stationary point by $O(L_f\epsilon^{-2})$ first-order oracles~\cite{nesterov2012make}. This oracle complexity cannot be improved because it matches the lower bound provided in~\cite{carmon2020lower,carmon2021lower}. 
	\item When $g\equiv0$, problem~\eqref{eq:model} becomes a linear equality-constrained smooth non-convex optimization problem 
	\begin{align}
		\label{eq:affinenonconvex}
			\min_\vx&~f_0(\vx), \ \text{ s.t. } \vA\vx+\vb=\mathbf{0}.
	\end{align}
	It is known that, if exact projection onto the feasible set $\{\vx~|~\vA\vx +  \vb=\mathbf{0}\}$ is allowed, the projected gradient method can find an $\epsilon$-stationary point in $O(L_f\epsilon^{-2})$ iterations. When exact projection is prohibited, 
	one can perform inexact projection through the matrix-vector multiplications with $\vA$ and $\vA^{\top}$, which are allowed under Assumption~\ref{ass:linearspan}. 
	This way, with $O(\kappa(\vA)\log(\epsilon^{-1}))$ multiplications, one can project any point to the feasible set with a $\mathrm{poly}(\epsilon)$ error\footnote{$\mathrm{poly}(\epsilon)$ denotes a polynomial of $\epsilon$.}. 
	Here, $\kappa(\vA)$ is the condition number of $\vA$ defined as
	$$
	\kappa(\vA):= \sqrt{\frac{\lambda_{\max}(\vA\vA\zz)}{\lambda_{\min}(\vA\vA\zz)}},
	$$
	where $\lambda_{\min}(\vA\vA\zz)$  and $\lambda_{\max}(\vA\vA\zz)$ are the smallest positive and the largest eigenvalues of $\vA\vA\zz$, respectively. Using this inexact projection as a subroutine, one can easily develop an inexact projected gradient method to this special case of \eqref{eq:model} with oracle complexity of  $O(\kappa(\vA)\log(\epsilon^{-1})L_f\epsilon^{-2})$. This complexity matches the lower bound $O(\kappa(\vA)L_f\epsilon^{-2})$ in~\cite{sun2019distributed} up to a logarithmic factor and thus is nearly optimal. 
\end{enumerate}

These two examples suggest that a lower bound of oracle complexity is valuable as it informs algorithm designers which algorithms can be potentially improved for higher efficiency and which are non-improvable without additional assumptions. In literature, there exist multiple algorithms with the theoretically proved oracle complexity for problem~\eqref{eq:model} or a more general problem.  We refer readers to~\cite{bian2015linearly, haeser2019optimality, kong2019complexity, jiang2019structured,melo2020iteration1,zhang2020proximal,o2021log,zhang2022global} for problems with affine constraints and~\cite{cartis2014complexity, cartis2017corrigendum,sahin2019inexact,xie2021complexity,xiao2023dissolving, grapiglia2021complexity, lin2022complexity,kong2022iteration} for problems with nonlinear constraints. However, to the best of our knowledge, there is no lower bound of the oracle complexity for finding a (near) $\epsilon$-stationary point of problem~\eqref{eq:model} under Assumption~\ref{ass:linearspan}. Although there are lower bounds established for decentralized smooth optimization~\cite{sun2019distributed}, linearly constrained smooth optimization~\cite{sun2019distributed}, and non-convex strongly-concave min-max problems~\cite{li2021complexity,zhang2021complexity}, these lower bounds either cannot be applied to problem~\eqref{eq:model} or can only be applied to a special case of problem~\eqref{eq:model}, so they will not be tight enough as we will show later in this paper. 
To fill this gap in literature, we pose the following question:
	\begin{center}
		\emph{
			What is the minimum oracle complexity for a first-order method satisfying a linear-span assumption, e.g., Assumption~\ref{ass:linearspan}, 
			to find a (near) $\epsilon$-stationary point of problem~\eqref{eq:model} that satisfies Assumption~\ref{assume:problemsetup}?}
	\end{center}



	\subsection{Contributions}
		
		Our first major contribution is to provide an answer to the question above by establishing a lower bound of the oracle complexity of FOMs for finding a (near) $\epsilon$-stationary point of problem~\eqref{eq:model}. This is achieved by adapting the worst-case instance in~\cite{sun2019distributed} for affinely constrained smooth optimization to the affinely constrained structured non-smooth problem~\eqref{eq:model}. In particular, we use a subset of the affine constraints of the instance in~\cite{sun2019distributed} to design the regularized term $g(\vx)$ in problem~\eqref{eq:model} and leave the remaining constraints in that instance as affine constraints in problem~\eqref{eq:model}. We show that under Assumption~\ref{ass:linearspan}, any algorithm needs at least $O({\kappa([\bar{\vA};\vA]) L_f \Delta_{F_0}} \epsilon^{-2})$ iterations/oracles to find a (near) $\epsilon$-stationary point of the instance we design; see Theorem~\ref{thm:lower}. Here 
		\begin{align}
			\label{eq:joint_condition_number}
		\Delta_{F_0}:= F_0(\vx^{(0)})-\inf_\vx F_0(\vx), \ [\bar{\vA};\vA]:= \left[
		\begin{array}{c}
			\bar{\vA}\\
			\vA
		\end{array}
		\right]\text{ and }	\kappa([\bar{\vA};\vA]):= \sqrt{\frac{\lambda_{\max}([\bar{\vA};\vA] [\bar{\vA};\vA]^\top )}{\lambda_{\min}([\bar{\vA};\vA] [\bar{\vA};\vA]^\top )}}.
		\end{align}
		Our lower complexity bound for~\eqref{eq:model} can be viewed as a generalization of the lower bound $O(\kappa(\vA)L_f\epsilon^{-2})$ for problem~\eqref{eq:affinenonconvex} in~\cite{sun2019distributed}. Our result provides a new insight that the difficulty of finding a (near) $\epsilon$-stationary point of problem~\eqref{eq:model} depends on the interaction between the affine constraints and the non-smooth regularization term characterized by $\kappa([\bar{\vA};\vA])$. 
		
Under Assumption~\ref{ass:linearspan}, an algorithm is allowed to call $\prox_{\eta g}(\cdot)$, which may not be easy to compute, especially when $\bar\vA$ in  Assumption~\ref{assume:problemsetup} is a generic matrix without a special structure. On the contrary, when $\bar g$ in  Assumption~\ref{assume:problemsetup} is simple enough so that  $\prox_{\eta \bar g}(\cdot)$ can be computed easily, one can apply a variable splitting technique such as the one used in the alternating direction method of multipliers (ADMM) to find a (near) $\epsilon$-stationary point of problem~\eqref{eq:model}. 
		To do so, we introduce a new variable $\vy\in\mathbb{R}^{\bar n}$ and reformulate problem~\eqref{eq:model} to its Splitting Problem
		\begin{equation}
				\label{eq:model-spli}
		\begin{aligned}		
			\min_{ \vx\in\RR^d, \vy\in\RR^{\bar n}} &~F(\vx,\vy):=f_0(\vx) + \bar{g}( \vy),\ 
			 \st~\vA \vx + \textbf{b} =0,\,\, \vy = \bar{\vA} \vx + \bar{\textbf{b}}.
		\end{aligned}		\tag{SP}
		\end{equation}
For $t=0,1,\dots$,	let $(\vx^{(t)},\vy^{(t)})$ be the solution generated in the $t$-th iteration by an algorithm for solving~\eqref{eq:model-spli}. We assume they satisfy the following linear-span assumption (see \cite{ouyang-xu2021lower,carmon2020lower}), instead of Assumption~\ref{ass:linearspan}. 
 
		\begin{assumption}
		\label{ass:linearspan3}
There is an oracle such that given any $\vx\in\RR^d$ and $\vy, \vxi$ in $\RR^{\bar n}$, it can return $(\bar\vb, \vA\zz\vb,\bar\vA\zz\bar\vb)$, $(\nabla f_0(\vx), \bar\vA\vx, \vA\zz \vA\vx,\bar\vA\zz \bar\vA\vx,\bar\vA\zz\vy)$, and $\prox_{\eta \bar{g}}(\vxi)$ for any $\eta>0$. In addition, the sequence $\left\{(\vx^{(t)},\vy^{(t)})\right\}_{t=0}^{\infty}$ generated by the underlying algorithm satisfies that,  for all $t\geq 1$, 
		\begin{equation*}
			\begin{aligned}
				&\vx^{(t)}\in \mathbf{span}\left(\left\{ \vA\zz\vb,\bar\vA\zz\bar\vb\right\}\bigcup\cup_{s=0}^{t-1}\left\{\vx^{(s)}, \nabla f_0(\vx^{(s)}), \vA\zz \vA\vx^{(s)},\bar\vA\zz \bar\vA\vx^{(s)},\bar\vA\zz\vy^{(s)}\right\}\right),\\
				&\vy^{(t)} \in 
				\mathbf{span}\left(\big\{\vxi^{(t)},\vzeta^{(t)} \big\}\right), \text{ where }\\ &
				\vxi^{(t)}\in \mathbf{span}\left(\left\{ \bar\vb\right\}\bigcup\cup_{s=0}^{t-1}\left\{\vy^{(s)}, \bar\vA\bar\vA\zz\vy^{(s)}, \bar\vA\vx^{(s)}\right\}\right),\ 
				\vzeta^{(t)}\in \left\{\prox_{\eta \bar{g}}(\vxi^{(t)})~|~\eta>0\right\}.
			\end{aligned}		
		\end{equation*}
	\end{assumption}
We say $(\vx^{(t)},\vy^{(t)})$ is generated by the $t$-th iteration of the algorithm. Our second major contribution is to show that under Assumption~\ref{ass:linearspan3}, the minimum oracle complexity is  $O({\kappa([\bar{\vA};\vA]) L_f \Delta_F} \epsilon^{-2})$ 
for an algorithm to find an $\epsilon$-stationary point of problem~\eqref{eq:model-spli} that satisfies Assumption~\ref{assume:problemsetup}(a,c) and the condition $$\Delta_F :=F(\vx^{(0)},\vy^{(0)})-\inf_{\vx,\vy} F(\vx,\vy)<\infty.$$ The lower bound  can be proved by using the same worst-case instance as that we will use to prove the lower bound for problem~\eqref{eq:model}. Moreover, under Assumption~\ref{ass:linearspan3}, the minimum oracle complexity 
for an algorithm to find a near $\epsilon$-stationary point of problem~\eqref{eq:model} is $O({\kappa([\bar{\vA};\vA]) L_f \Delta_{F_0}} \epsilon^{-2})$, {the same as that we establish under Assumption~\ref{ass:linearspan} which allows the more expensive operator $\prox_{\eta g}$.} 
	
Given a lower bound of oracle complexity, a critical question is whether the bound is tight, or equivalently, whether it can be achieved by an algorithm that meets the underlying assumptions upon which the lower bound is established. To address this question and shed light on the tightness of our lower bounds, 
we make a third significant contribution by introducing a novel method called the \emph{inexact proximal gradient} (IPG) method. This method is specifically designed to solve problem~\eqref{eq:model-spli}, while satisfying the conditions outlined in Assumption~\ref{ass:linearspan3}. 
Remarkably, IPG is able to achieve an oracle complexity that matches our established lower bound $O({\kappa([\bar{\vA};\vA]) L_f \Delta_F} \epsilon^{-2})$ up to logarithmic factors for solving problem~\eqref{eq:model-spli}. That is, the lower complexity bound is tight for problem \eqref{eq:model-spli}. Meanwhile, IPG is able to find a near $\epsilon$-stationary point of problem \eqref{eq:model} by  $O({\kappa([\bar{\vA};\vA]) L_f \Delta_{F_0}} \epsilon^{-2})$ iterations/oracles, up to logarithmic factors. This means that the lower complexity bound is also tight for problem \eqref{eq:model} under Assumption \ref{ass:linearspan3}.


\subsection{Related Work}\label{sec:relatedwork}

The \emph{proximal gradient} (PG) method can find an $\epsilon$-stationary point of the composite non-smooth non-convex problem~\eqref{eq:composite} 
within $O(L_f\epsilon^{-2})$~\cite{nesterov2012make} iterations, which matches the lower bound in~\cite{carmon2020lower,carmon2021lower}. 
When there is no affine constraints and $g\equiv0$ in problem~\eqref{eq:model}, our lower-bound complexity result is reduced to the lower bound in~\cite{carmon2020lower}.

		
For the affinely constrained non-convex smooth problem~\eqref{eq:affinenonconvex}, Sun and Hong~\cite{sun2019distributed} show a lower-bound complexity of  $O(\kappa(\vA)L_f\epsilon^{-2})$ for finding an $\epsilon$-stationary point. 
In the same paper, they give an FOM that achieves this lower bound. When $g\equiv 0$, our complexity lower bound for problem~\eqref{eq:model} is reduced to their lower bound.
	
	Before our work, there only exist upper bounds of oracle complexity for finding a (near) $\epsilon$-stationary point of \eqref{eq:model} and~\eqref{eq:model-spli}. For instance, Kong et al.~\cite{kong2019complexity} develop a quadratic-penalty accelerated inexact proximal point method that finds an $\epsilon$-stationary point of problem~\eqref{eq:model} with oracle complexity $O(\epsilon^{-3})$. Lin et al.~\cite{lin2022complexity} study a method similar to~\cite{kong2019complexity} and show that oracle complexity of $O(\epsilon^{-5/2})$ is sufficient. The  \emph{augmented Lagrangian method} (ALM) is another effective approach for problem~\eqref{eq:model}. The oracle complexity of ALM for problem~\eqref{eq:model} has been studied by~\cite{hong2016decomposing,hajinezhad2019perturbed,melo2020iteration1,zhang2022global}. For example, the inexact proximal accelerated ALM by~\cite{melo2020iteration1} achieves oracle complexity of $O(\epsilon^{-5/2}$) and, in a special case where $g(\vx)$ is the indicator function of a polyhedron, the smoothed proximal ALM by~\cite{zhang2022global} improves the complexity to $O(\epsilon^{-2}$). ADMM is an effective algorithm for optimization with a separable structure  like that in problem~\eqref{eq:model-spli}. ADMM and its variants have been studied by~\cite{melo2017iteration,goncalves2017convergence,jiang2019structured, zhang2020proximal, yashtini2022convergence,zhang2022global,hong2016convergence} for constrained non-convex  optimization problems including problem~\eqref{eq:model-spli}. For example, it is shown by~\cite{goncalves2017convergence,yashtini2022convergence,jiang2019structured} that ADMM finds an $\epsilon$-stationary point of  problem~\eqref{eq:model-spli} with oracle complexity of $O(\epsilon^{-2})$. 
	
	The aforementioned methods for problem~\eqref{eq:model} all satisfy Assumption~\ref{ass:linearspan}, and the aforementioned methods for problem~\eqref{eq:model-spli} all satisfy Assumption~\ref{ass:linearspan3}.  However, the oracle complexity of those methods for finding an  $\epsilon$-stationary point either does not match or is not comparable with our lower-bound complexity for the corresponding problems. 
	Specifically, the oracle complexity of ADMM in~\cite{yashtini2022convergence} for solving problem~\eqref{eq:model-spli} depends on the Kurdyka-Łojasiewicz (KŁ)
	coefficient, which is not directly comparable with our lower bound. The oracle complexity $O(\kappa^{2}([\vA;\bar{\vA}])L_f^2\Delta_F \epsilon^{-2})$ of ADMM for solving problem~\eqref{eq:model-spli} is presented in~\cite{goncalves2017convergence}, under the assumption that $[\bar{\vA}; \vA]$ has a full-row rank. 
	The results of~\cite{jiang2019structured} of ADMM are only applicable to problems~\eqref{eq:model} and~\eqref{eq:model-spli} with a separable structure and $\bar{\vA}=\vI_d$ or $g\equiv 0$, for which case their oracle complexity is $O(\kappa^{2}(\vA)L_f^2\Delta \epsilon^{-2})$ with $\Delta = \Delta_{F_0}$ or $\Delta_F$. Zhang and Luo in~\cite{zhang2022global} study the complexity of ALM for problem~\eqref{eq:model} when $g$ is the indicator function of a polyhedral set and the exact projection onto the polyhedral set can be computed. They obtained complexity of $O(\widehat\kappa^{2}L_f^3\Delta_{F_0} \epsilon^{-2})$, where $\widehat\kappa$ is a joint condition number of the equality $\vA \vx + \textbf{b} =\vzero$  and the inequality defining the polyhedral set. When $g\equiv 0$, their complexity is reduced to $O(\kappa^{2}(\vA)L_f^3\Delta_{F_0} \epsilon^{-2})$. Zhang et al.\cite{zhang2022iteration} further extend~\cite{zhang2020proximal, zhang2022global} to problems with nonlinear convex inequality constraints. However, their algorithm requires exact projection to the set defined by the inequality constraints, which is impractical for many applications and does not satisfy Assumption~\ref{ass:linearspan}. 
	
	\subsection{Notations and Definitions}
	

	For any $a\in\RR$, we use $\lceil a\rceil$ to denote the smallest integer that is no less than $a$. 
	$\mathbf{Null}(\vH)$ represents the null space of a matrix $\vH$. For any vector $\vz$, $[\vz]_j$ denotes its $j$-th coordinate. 	
	We denote $\mathbf{1}_{p}$ for an all-one vector in $\mathbb{R}^p$, and $\mathbf{0}$ to represent an all-zero vector when its dimension is clear from the context. $\vI_p$ denotes a $p\times p$ identity matrix and  
	\begin{equation}
		\label{eq:matrixJ}
		\vJ_{p}:=\left[\begin{array}{ccccc}
			-1 & 1 &&&\\
			& -1 &1 &&\\
			&& \ddots & \ddots &  \\
			&&  &-1 & 1 \\
		\end{array}\right] \in \RR^{(p-1)\times p}.
	\end{equation}
A vector $\vx$ is said $\omega$-close to another vector $\widehat\vx$ if $\|\vx -\widehat\vx\| \le \omega$. For any set $\cX$, we denote $\iota_\cX$ as its indicator function, i.e., $\iota_\cX(\vx) = 0$ if $\vx\in\cX$ and $+\infty$ otherwise.	
	We use $\otimes$ for the Kronecker product, $\mathbf{co}(\mathcal{S})$ for the convex hull of a set $\mathcal{S}$ and $\overline{\mathbf{co}}(\cS)$ for the closure of $\mathbf{co}(\mathcal{S})$. 
	For a function $f:\mathbb{R}^d\rightarrow\mathbb{R}\cup \{+\infty\}$, its directional derivative at $\vx$ along a direction $\vv\in\mathbb{R}^d$  is defined as 
	\begin{equation}\label{eq:def-dir-der}
	f^{\prime}(\vx ; \vv)=\lim _{s \downarrow 0} \frac{f(\vx+s \vv)-f(\vx)}{s},
	\end{equation}
where $s\downarrow 0$ means $s\rightarrow 0$ and $s>0$.
$\partial g$ denotes the subdifferental of a closed convex function $g$.	
	
\begin{definition}\label{def:eps-pt-P}
Given $\epsilon \ge0$, a point $\vx^*$ is called an $\epsilon$-stationary point of~\eqref{eq:model} if for some $\vgamma\in\RR^n$,
\begin{equation}
		\label{eq:epsta}
		\max\left\{\dist\left(\mathbf{0},\nabla f_0(\vx^*)+\vA\zz \vgamma 
		+\partial g(\vx^*) \right), \left\| \vA\vx^*+\vb\right\| 
		\right\}\leq \epsilon,
	\end{equation}
and a point $(\vx^*, \vy^*)$ is called an $\epsilon$-stationary point of~\eqref{eq:model-spli} if for some $\vz_1\in \RR^{\bar{n}}$ and $\vz_2\in\RR^n$,
\begin{equation}
	\label{eq:kktviofgesub}
	\begin{aligned}
		&	\max\left\{ \dist (\vzero, \partial \bar{g}(\vy^*) - \vz_1), \|\nabla f_0(\vx^*) + \bar{\vA}\zz \vz_1 + \vA\zz \vz_2\|,  \|\vy^*-\bar{\vA}\vx^*-\bar{\vb}\|,  \|\vA\vx^*+\vb\| \right\}
		 \leq\epsilon.
	\end{aligned}
\end{equation}	
When $\epsilon = 0$, we simply call $\vx^*$ and $(\vx^*, \vy^*)$ stationary points of problems~\eqref{eq:model} and~\eqref{eq:model-spli}, respectively.	We say that $\bar{\vx}$ is a near $\epsilon$-stationary point of~\eqref{eq:model} if it is $\omega$-close to an $\epsilon$-stationary point $\vx^*$ of~\eqref{eq:model}
	with $\omega= O(\epsilon)$.
\end{definition}	
	
	\subsection{Organization}
	
	The rest of this paper is organized as follows. In Section~\ref{sec:lb1}, we present lower-bound complexity results for solving problem~\eqref{eq:model} of FOMs  under Assumption \ref{ass:linearspan}. Then in Section~\ref{sec:extension}, we show the lower-bound complexity results for solving problems~\eqref{eq:model} and~\eqref{eq:model-spli} of FOMs under Assumption \ref{ass:linearspan3}.
	In Section~\ref{sec:ub}, we propose an inexact proximal gradient algorithm for solving these two problems and present its worst-case oracle complexity, which shows the tightness of our lower complexity bounds. Concluding remarks are given in Section~\ref{sec:conclusion}.
	
	\section{Lower Bound of Oracle Complexity for Problem~\eqref{eq:model} under Assumption \ref{ass:linearspan}}\label{sec:lb1}
	
	In this section, under Assumption~\ref{ass:linearspan}, we derive a lower bound of the oracle complexity of an FOM  to find a (near) $\epsilon$-stationary point of problem~\eqref{eq:model} for a given $\epsilon>0$. Motivated by~\cite{carmon2020lower,sun2019distributed}, our approach is to design a worst-case instance of \eqref{eq:model} such that 
	a large number of oracles satisfying Assumption~\ref{ass:linearspan} 
	will be needed to find a desired solution. 
	

	\subsection{A Challenging Instance  $\cP$ of Problem~\eqref{eq:model}}
		
	Let $m_1$ and $m_2$ be positive integers such that $m_1m_2$ is even and $m_1\geq2$. Let $m=3m_1m_2$ and $\bar{d}$ be an odd positive integer such that $\bar{d}\geq 5$. Also, set $d=m{\bar{d}}$, and  let  
\begin{eqnarray}
	\label{eq:xblock}
	\vx=\left(\vx_1\zz, \ldots,\vx_m\zz\right)\zz\in\mathbb{R}^{d}, \text{ with }\vx_i\in\mathbb{R}^{\bar d}, \, i=1,\dots,m.
\end{eqnarray}
Moreover, we define a matrix  $\vH\in\mathbb{R}^{(m-1)\bar{d}\times m\bar{d}}$ by
\begin{equation}
	\label{eq:matrixAstar}
	\newcommand{\zm}{
			\left[\begin{array}{ccccc}
			-\vI_{\bar{d}} & \vI_{\bar{d}} &&&\\
			& -\vI_{\bar{d}} &\vI_{\bar{d}} &&\\
			&& \ddots & \ddots &  \\
			&&  &-\vI_{\bar{d}} & \vI_{\bar{d}} \\
		\end{array}\right]
	}		
		\vH:=mL_f\cdot \vJ_{m}\otimes \vI_{\bar{d}}=
		mL_f\cdot
		\left.
		\,\smash[b]{\underbrace{\!\zm\!}_{\textstyle\text{$m$ blocks}}}\,
		\right\}\text{$m-1$ blocks.}
		\vphantom{\underbrace{\zm}_{\textstyle\text{$m$ blocks}}}		
\end{equation}
%
Define 
\begin{equation}
\label{eq:indexsetM}
\begin{aligned}
\mathcal{M}:=\{im_1| i=1,2,\dots,3m_2-1\},& \quad \mathcal{M}^C:=\{1,2,\dots,m-1\}\backslash  \mathcal{M}, \\ n=(m-3m_2){\bar{d}},& \quad \bar n=(3m_2-1){\bar{d}},
\end{aligned}
\end{equation}
and let  
\begin{eqnarray}
	\label{eq:AandAbar}
\bar{\vA}:=mL_f\cdot \vJ_{\mathcal{M}}\otimes \vI_{\bar{d}},\quad 
\vA:=mL_f\cdot \vJ_{\mathcal{M}^C}\otimes \vI_{\bar{d}}, \quad \bar\vb = \vzero \in \RR^{\bar n}, \quad \vb = \vzero \in \RR^n,
\end{eqnarray}
where $\vJ_{\mathcal{M}}$ and $\vJ_{\mathcal{M}^C}$ are the rows of $\vJ_{m}$ indexed by $\mathcal{M}$ and $\mathcal{M}^C$, respectively.


Furthermore, we define $\bar g:\mathbb{R}^{\bar{n}}\rightarrow \mathbb{R}$ and $g:\RR^d \rightarrow \RR$ as
\begin{equation}
	\label{eq:gbar}
\bar g(\vy):=\frac{\beta}{mL_f}\|\vy\|_1=\max\left\{\vu^\top\vy~\Big|~\|\vu\|_\infty\leq \frac{\beta}{mL_f}\right\},
\end{equation}
and \begin{equation}
	\label{eq:g}
	g(\vx):=\bar{g}(\bar\vA\vx)=\beta\sum_{i\in\mathcal{M}}  \|\vx_{i }-\vx_{i+1}\|_1,
\end{equation}
where $\vx$ has the block structure in~\eqref{eq:xblock}, $\mathcal{M}$ is defined in \eqref{eq:indexsetM}, and $\beta$ is a constant satisfying 
\begin{equation}
	\label{eq:betachoice}
	\beta>(50\pi+1+\|\vA\|)\sqrt{m}\epsilon.
\end{equation} 


Finally, we design the smooth function $f_0$ to complete the instance of~\eqref{eq:model}. Let $\Psi:\RR\mapsto\RR$ and $\Phi:\RR\mapsto\RR$ be defined as
	\begin{equation}
		\label{eq:PsiPhi}
	\Psi(u):=\left\{\begin{array}{ll}
		0, & \text{ if }u \leq 0, \\
		1-e^{-u^2}, & \text{ if } u>0,
	\end{array}\right.  \quad \text{and}\quad\Phi(v):=4 \arctan v+2 \pi.
	\end{equation}
In addition, for each $j=1,\ldots,\bar d$, we define function $\varphi\left(\cdot, j\right):\RR^{{\bar{d}}} \rightarrow \RR$ as
   \begin{equation}
   	\label{eq:Theta}
   	\varphi\left(\vz, j\right):=\left\{
   	\begin{array}{ll}
	-\Psi(1) \Phi\left([\vz]_1\right), &\text{ if } j=1,  \\[0.1cm]
		\Psi\left(-[\vz]_{j-1}\right) \Phi\left(-[\vz]_j\right)-\Psi\left([\vz]_{j-1}\right) \Phi\left([\vz]_j\right), &\text{ if } j = 2, \dots,{\bar{d}},
		\end{array}\right.
	\end{equation}	
and	for $i=1,\dots,m$, we define $h_i:\RR^{{\bar{d}}} \rightarrow \RR$ as
	\begin{equation}
		\label{eq:hi}
h_i\left(\vz\right):= \left\{\begin{array}{ll}
	\varphi\left(\vz, 1\right)+3 \sum_{j=1}^{\lfloor {\bar{d}} / 2\rfloor} \varphi\left(\vz, 2 j\right), & \text{ if }i \in\left[1, \frac{m}{3}\right], \\ [0.1cm]
\varphi\left(\vz, 1\right), &\text{ if } i \in\left[\frac{m}{3}+1, \frac{2 m}{3}\right], \\ [0.1cm]
\varphi\left(\vz, 1\right)+3 \sum_{j=1}^{\lfloor {\bar{d}} / 2\rfloor} \varphi\left(\vz, 2 j+1\right), &\text{ if } i \in\left[\frac{2 m}{3}+1, m\right].
\end{array}\right.
	\end{equation}	
Now for a given $\epsilon\in(0,1)$ and $L_f > 0$, for $i=1,\dots,m$, we define $f_i:\RR^{{\bar{d}}} \rightarrow \RR$ as
\begin{equation}
	\label{eq:fi}
	f_i\left(\vz\right):=\frac{300  \pi \epsilon^2}{mL_f} h_i\left(\frac{\sqrt{m}L_f \vz }{150\pi\epsilon}\right), \forall\, \vz\in \RR^{\bar d},
\end{equation}	
and let $f_0: \RR^d \rightarrow \RR$ be
\begin{equation}
	\label{eq:f0}
f_0(\vx):= \sum^{m}_{i=1}  f_i(\vx_i), \forall\, \vx\in \RR^d \text{ with the structure in }\eqref{eq:xblock}.
\end{equation}

Putting all the components given above, we obtain a specific instance of~\eqref{eq:model}. We formalize it in the following definition.
\begin{definition}[instance $\cP$]
	\label{def:hardinstance}
Given $\epsilon\in(0,1)$ and $L_f > 0$, let $m_1, m_2$ and $\bar d$ be integers such that $m_1\ge2$ is even and $\bar d\ge 5$ is odd. We refer to as \emph{instance $\cP$} the instance of problem~\eqref{eq:model} where $f_0$ is given in~\eqref{eq:f0} with each $f_i$ defined in~\eqref{eq:fi}, $g$ is given in~\eqref{eq:g} with $\beta$ satisfying~\eqref{eq:betachoice}, and $(\vA, \bar\vA, \vb, \bar\vb)$ is given in~\eqref{eq:AandAbar}. 
\end{definition}

\subsection{Properties of Instance $\cP$}

In order to show the challenge of instance $\cP$ for an algorithm under Assumption~\ref{ass:linearspan}, we give a few facts and properties about $\cP$. First, notice that $\bar{\vA}$ and $\vA$ in~\eqref{eq:AandAbar} are two block submatrices of $\vH$ in rows. 
 It is easy to obtain the following proposition.
 \begin{proposition}\label{prop:consensus}
 By the definitions in~\eqref{eq:xblock} through~\eqref{eq:g}, it holds
 \begin{enumerate}
 \item[\textnormal{(a)}] $\vx_1=\vx_2=\cdots=\vx_m$ if and only if $\vH\vx=\mathbf{0}$; 
 \item[\textnormal{(b)}] $\vx_i=\vx_{i+1}$ for $i\in \mathcal{M}$ if and only if $\bar\vA\vx=\mathbf{0}$ or equivalently $g(\vx)=0$; 
 \item[\textnormal{(c)}] $\vx_i=\vx_{i+1}$ for $i\in \mathcal{M}^C$ if and only if $\vA\vx=\mathbf{0}$; 
 \item[\textnormal{(d)}] $\vx_1=\vx_2=\cdots=\vx_m$ if and only if $\vA\vx=\mathbf{0}$ and $g(\vx)=0$.
 \end{enumerate} 
 \end{proposition}	
		
Second, it is straightforward to have	the partial derivatives $\left\{\frac{\partial h_i(\vz)}{\partial [\vz]_j}\right\}$ based on three cases of $j$. 
	
	\noindent Case i) When $j=1$,
	\begin{equation}
		\label{eq:deriveh3}
		\frac{\partial h_i(\vz)}{\partial [\vz]_j} =\left\{\begin{array}{ll}
			-\Psi(1) \Phi^{\prime}([\vz]_{1})+3\left[-\Psi^{\prime}(-[\vz]_{1}) \Phi(-[\vz]_{2})-\Psi^{\prime}([\vz]_{1}) \Phi([\vz]_{2})\right],& \text{ if }i \in\left[1, \frac{m}{3}\right], \\[0.1cm]
			-\Psi(1) \Phi^{\prime}([\vz]_{1}), &\text{ if }i \in\left[\frac{m}{3}+1, m\right].
		\end{array}\right.
	\end{equation}

\noindent	Case ii) When $j$ is even, 
	\begin{equation}
		\label{eq:deriveh1}
		\frac{\partial h_i(\vz)}{\partial [\vz]_j} =\left\{\begin{array}{ll}
			3\left[-\Psi(-[\vz]_{j-1}) \Phi^{\prime}(-[\vz]_j)-\Psi([\vz]_{j-1}) \Phi^{\prime}([\vz]_j)\right], & \text{ if }i \in\left[1, \frac{m}{3}\right], \\[0.1cm]
			0,			& \text{ if }i \in\left[\frac{m}{3}+1, \frac{2m}{3}\right], \\[0.1cm]
			3\left[-\Psi^{\prime}(-[\vz]_j) \Phi(-[\vz]_{j+1})-\Psi^{\prime}([\vz]_j) \Phi([\vz]_{j+1})\right], & \text{ if }i \in\left[\frac{2m}{3}+1, m\right].
		\end{array}\right.
	\end{equation}
	
\noindent	Case iii) When $j$ is odd and $j\neq 1$, 
	\begin{equation}
		\label{eq:deriveh2}
	\frac{\partial h_i(\vz)}{\partial [\vz]_j} = \left\{\begin{array}{ll}
		3\left[-\Psi^{\prime}(-[\vz]_j) \Phi(-[\vz]_{j+1})-\Psi^{\prime}([\vz]_j) \Phi([\vz]_{j+1})\right], & \text{ if }i \in\left[1, \frac{m}{3}\right], \\[0.1cm] 
		0,	& \text{ if }i \in\left[\frac{m}{3}+1, \frac{2m}{3}\right], \\[0.1cm] 
		3\left[-\Psi(-[\vz]_{j-1}) \Phi^{\prime}(-[\vz]_j)-\Psi([\vz]_{j-1}) \Phi^{\prime}([\vz]_j)\right], & \text{ if }i \in\left[\frac{2m}{3}+1, m\right].
		\end{array}\right.
	\end{equation}




Thirdly, instance~$\mathcal{P}$ satisfies Assumption~\ref{assume:problemsetup}. By~\eqref{eq:g}, Assumption~\ref{assume:problemsetup}(c) clearly holds. 
To verify Assumptions~\ref{assume:problemsetup}(a) and~\ref{assume:problemsetup}(b), we need the next lemma.
\begin{lemma}
	\label{lem:functions}
Let $\Psi$, $\Phi$ and $h_i$ be given in~\eqref{eq:PsiPhi} and~\eqref{eq:hi}.	The following statements hold:
\begin{enumerate}	
	\item[\textnormal{(a)}]  
	$\Psi(u)=0$ and $\Psi'(u)=0$ for all $u\leq 0$, and $0 \leq \Psi(u)<1$, $0 \leq \Psi^{\prime}(u) \leq \sqrt{2/e}$,
	$0<\Phi(v)<4 \pi$, and $0<\Phi^{\prime}(v) \leq 4$ for all $v\in\RR$.
\item	[\textnormal{(b)}] 
	$
	\Psi(u) \Phi^{\prime}(v)>1
	$
	for all $u$ and  $v$ satisfying $u \geq 1$ and $|v|<1$.
	
\item[\textnormal{(c)}] $h_i(\mathbf{0})-\inf_{\vz} h_i\left(\vz\right)\leq 10\pi {\bar{d}}$ for $i=1,2,\ldots,m$.
	
\item[\textnormal{(d)}] $\nabla {h}_i$ is $75 \pi$-Lipschitz continuous  for $i=1,2,\ldots,m$.
	
\item[\textnormal{(e)}]  ${h}_i$ is $25\pi\sqrt{{\bar{d}}}$-Lipschitz continuous for $i=1,2,\ldots,m$.
\end{enumerate}
\end{lemma}
\begin{proof}
	(a)-(d) are directly from~\cite[Lemma 3.1]{sun2019distributed}. By derivatives~\eqref{eq:deriveh3}--\eqref{eq:deriveh2}, we obtain that, for any $\vz\in\RR^{\bar{d}}$, 
	\begin{equation}
		\begin{aligned}
			\label{eq:partialhmax}
			\left|\frac{\partial h_i(\vz)}{\partial [\vz]_j}\right| \leq \max \left\{\sup _u\left|\Psi(1) \Phi^{\prime}(u)\right|+6 \sup _u\left|\Psi^{\prime}(u)\right| \sup _v|\Phi(v)|, 6 \sup _u\left|\Psi(u)\right| \sup _v\left|\Phi^{\prime}(v)\right|\right\}< 25\pi, \forall\, i, j,
		\end{aligned}
	\end{equation}
	where the second inequality is from Lemma~\ref{lem:functions}(a) and the fact that $\sup _v\left|\Psi(1) \Phi^{\prime}(v)\right|\leq 4(1-e^{-1})<\pi$ by the definition of $\Phi$. Hence, $\|\nabla h_i(\vz) \|\leq 25\pi\sqrt{{\bar{d}}}$. Thus (e) holds, and we complete the proof.
\end{proof}	
From the definition of $f_i$ in \eqref{eq:fi} and the chain rule, we have 
\begin{equation}\label{eq:grad-f-i-chain}
 	[\nabla f_i(\vz)]_j  =  \frac{2\epsilon}{\sqrt{m}} \left[\nabla {h}_i\left(\frac{\sqrt{m} L_f \vz}{150\pi\epsilon}\right)\right]_j,\quad \forall j=1,\dots,m,
 \end{equation}
 which together with Lemma~\ref{lem:functions} clearly indicates the properties below about $\{f_i\}_{i=0}^m$.  
	
	\begin{lemma}
	\label{cor:boundf}
Let $\{f_i\}_{i=0}^m$	be defined in~\eqref{eq:fi} and~\eqref{eq:f0}. Then	
\begin{enumerate}	
	\item[\textnormal{(a)}] $ f_i(\mathbf{0})-\inf _{\vz} f_i\left(\vz\right)\leq 3000\pi^2 {\bar{d}}\epsilon^2/mL_f$ for $i=1,\dots,m$, and $ f_0(\mathbf{0})-\inf _{\vx} f_0\left(\vx\right)\leq 3000\pi^2 {\bar{d}}\epsilon^2/L_f$.
	
	\item[\textnormal{(b)}] $\nabla f_i$ is $L_f$-Lipschitz continuous for $i=0,1,\dots,m$.
	
	\item[\textnormal{(c)}] $f_i$ is $\frac{50\pi\epsilon\sqrt{{\bar{d}}}}{\sqrt{m}}$-Lipschitz continuous for $i=1,\dots,m$, and $f_0$ is ${50\pi\epsilon\sqrt{m{\bar{d}}}}$-Lipschitz continuous.
\end{enumerate}	
	\end{lemma}

By Lemma~\ref{cor:boundf} and $\bar g \ge0$, we immediately have the   following proposition. 
\begin{proposition}
Instance~$\mathcal{P}$ given in Definition~\ref{def:hardinstance} satisfies Assumption~\ref{assume:problemsetup}.
\end{proposition}

The following lemma characterizes the joint condition number of $\bar{\vA}$ and $\vA$ defined in~\eqref{eq:joint_condition_number}, which will appear in our lower bound of oracle complexity.  

\begin{lemma}
\label{lem:condH}
Let $\vA$ and $\bar\vA$ be given in~\eqref{eq:AandAbar}. Then $\frac{m}{4} \le \kappa([\bar{\vA}; \vA])=\kappa(\vH)={\frac{\sin(\frac{(3m_1m_2-1)\pi}{6m_1m_2})}{\sin(\frac{\pi}{6m_1m_2})}}<m$.
\end{lemma}
\begin{proof}
The first equality holds because $\vH$ are split into $\bar{\vA}$ and $\vA$ in rows. 
	Let 
	\begin{equation*}
		\begin{aligned}
			\bar{\vH}=mL_f\left[\begin{array}{cccc}
				-1 & 1 &\\
				& \ddots & \ddots &  \\
				&  &-1 & 1 \\
			\end{array}\right]\in\RR^{(m-1)\times m}.
		\end{aligned}
	\end{equation*}
	We then have $\vH \vH\zz=(\bar{\vH}\otimes \vI_{\bar{d}})(\bar{\vH}\otimes \vI_{\bar{d}})\zz=\left(\bar{\vH}\bar{\vH}\zz\right) \otimes\left( \vI_{\bar{d}} \vI_{\bar{d}}^{\top}\right)=\left(\bar{\vH}\bar{\vH}\zz\right) \otimes\vI_{\bar{d}} $. Let $\lambda_i(\bar{\vH}\bar{\vH}\zz)$ be the $i$-th largest eigenvalue of $\bar{\vH}\bar{\vH}\zz$. 
	Since $\bar{\vH}\bar{\vH}\zz$ is tridiagonal and Toeplitz, its eigenvalues have closed forms~\cite{gray2006toeplitz}:  
	$$\textstyle \lambda_i(\bar{\vH}\bar{\vH}\zz)= 4m^2L_f^2\sin^2\left(\frac{i\pi}{6m_1m_2}\right), \forall\, i=1,2,\ldots, m-1.$$ It then yields that $\kappa([\bar{\vA}; \vA])=\kappa(\vH )=\kappa(\bar{\vH})= {\frac{\sin(\frac{(3m_1m_2-1)\pi}{6m_1m_2})}{\sin(\frac{\pi}{6m_1m_2})}}$. Because  $\sin(z) \le 1, \forall\, z$, $\sin(z)\geq \frac{2z}{3}$ for $z\in[0,\pi/12]$, and $m_1m_2\geq 2$, we have 
		\begin{equation*}
	{\frac{\sin(\frac{(3m_1m_2-1)\pi}{6m_1m_2})}{\sin(\frac{\pi}{6m_1m_2})}}\leq\frac{1}{\frac{\pi}{9m_1m_2}}=\frac{3m}{\pi}<m.
	\end{equation*}
Also,	because  $z\geq\sin(z)\geq \frac{z}{2}$ for $z\in[0,\pi/2]$ and $m_1m_2\geq 2$, we have 
	\begin{equation*}
{\frac{\sin(\frac{(3m_1m_2-1)\pi}{6m_1m_2})}{\sin(\frac{\pi}{6m_1m_2})}}\geq	\frac{\frac{(m-1)\pi}{4m}}{\frac{\pi}{2m}}=\frac{m-1}{2}\geq \frac{m}{4}.
	\end{equation*}
Hence, we have obtained all desired results and complete the proof.	
\end{proof}

		\subsection{An Auxiliary Problem and Its Properties}

	To establish a lower bound of the oracle complexity for solving \eqref{eq:model} under Assumptions~\ref{assume:problemsetup} and \ref{ass:linearspan}, we consider an auxiliary problem of instance~$\mathcal{P}$ in this subsection and analyze its properties. The Auxiliary Problem is given as follows:
	\begin{equation}
		\label{eq:model3}
		\begin{aligned}
			\min_{ \vx\in\RR^d} \,\,&   f_0(\vx), 
			\,\,\st  \vH \vx =  \mathbf{0},
		\end{aligned}
		\tag{AP}
	\end{equation}
	where $\vH$ and $f_0$  are defined in~\eqref{eq:f0} and~\eqref{eq:matrixAstar}, respectively. An $\epsilon$-stationary point $\vx^*$ of problem~\eqref{eq:model3} satisfies
	\begin{equation}
		\label{eq:KKTviomodel2}
		\max\left\{\left\|\vH\vx^*\right\|, \min _{\vgamma'\in\RR^{(m-1){\bar{d}}}}\left\|\nabla  f_0(\vx^*)+\vH^{\top} \vgamma'\right\|\right\}\leq \epsilon.
	\end{equation}
	The next lemma characterizes the relationship between the (near-)stationary points of  instance $\cP$ and the auxiliary problem~\eqref{eq:model3}. The proof techniques  have been utilized in~\cite{liu2022linearly,liu2022inexact}.
		\begin{lemma}\label{cor:kktvio2}
Let $\epsilon >0$ be given in Definition~\ref{def:hardinstance} for instance $\cP$. Then for any $\widehat\epsilon \in [0, \epsilon]$, an $\widehat\epsilon$-stationary point of instance~$\mathcal{P}$ is also an $\widehat\epsilon$-stationary point of  the auxiliary problem~\eqref{eq:model3}. 
\end{lemma}

\begin{proof}
			Suppose $\vx^*$ is an $\widehat\epsilon$-stationary point of instance~$\mathcal{P}$, i.e., 
			 for some $\vgamma\in\RR^n$ and $\vxi\in\partial g(\vx^*)$ such that 
			\begin{equation}
				\label{eq:KKTe}
				\|\nabla f_0(\vx^*)+ \vxi+\vA\zz \vgamma\|\leq  \widehat\epsilon \quad\text{ and }\quad\|\vA\vx^*\|\leq \widehat\epsilon.
			\end{equation}
			By the definitions of $g$ in \eqref{eq:gbar} and~\eqref{eq:g}, there exists $\vu\in\mathbb{R}^{\bar{n}}$ such that $\vxi=\bar\vA^{\top}\vu$ and thus the first condition in~\eqref{eq:KKTe} becomes
			\begin{equation}
				\label{eq:KKTHe}
				\|\nabla f_0(\vx^*)+ \bar\vA^{\top}\vu+\vA\zz \vgamma \|\leq  \widehat\epsilon.
			\end{equation}

Hence, in order to show that $\vx^*$ is an $\widehat\epsilon$-stationary point of  problem~\eqref{eq:model3}, we only need to prove $\|\vH\vx^*\|\leq\widehat\epsilon$. To show this, we first notice from the definition in~\eqref{eq:def-dir-der} that, for any $\vv\in\mathbf{Null}(\vA)$, 
\begin{equation}
	\label{eq:dderivegeq0epsilon}
	\begin{aligned}
		F_0^{\prime}(\vx^*;\vv)= \vv\zz\nabla  f_0(\vx^*)+ g^{\prime}(\vx^*;\vv) \geq \vv\zz\nabla  f_0(\vx^*) +\vv\zz \vxi =\vv\zz\left(\nabla f_0(\vx^*)+ \vxi+\vA\zz \vgamma\right)\geq -\widehat\epsilon\|\vv\|,
	\end{aligned}
\end{equation}
where the first inequality follows from~\cite{clarke1990optimization} 
\begin{equation}\label{eq:result-dir-der}
g^{\prime}(\vx;\vv)=\sup_{\vxi'\in\partial g(\vx)}\vv\zz \vxi', \forall\, \vx\in \dom(g), \forall\, \vv,
\end{equation} 
and the second inequality is by~\eqref{eq:KKTe} and the Cauchy-Schwarz inequality.


Second, we claim $\bar\vA\vx^*=\mathbf{0}$, namely,  $\vx^*_{i}=\vx^*_{i+1}$ for all $i\in\mathcal{M}$, where $\mathcal{M}$ is defined in~\eqref{eq:indexsetM}. Suppose this claim is not true. Then for some $\bar{i}\in \mathcal{M}$, it holds $\vx^*_{\bar i}\neq\vx^*_{\bar i+1}$. 
We let 
\begin{equation}\label{eq:set-delta}
\bar{\vxi}:=\vx^*_{\bar{i}+1}-\vx^*_{\bar{i}}\neq \mathbf{0},
\end{equation}
and $\vv^*=(\vv_1^{*\top},\vv_2^{*\top},\dots,\vv_m^{*\top})^\top$ where each $\vv_i^*\in\mathbb{R}^{\bar{d}}$ is defined as
$$
\vv_i^*= \bar{\vxi}, \text{ if }i\leq \bar{i}, \text{ and } \vv_i^*= \mathbf{0},	\text{ if }i>\bar{i}. 
$$
It is easy to see that $\vv^*_i=\vv^*_{i+1}$ for any $i\neq \bar{i}$. Thus by Proposition~\ref{prop:consensus}(b), $\vA\vv^*=\mathbf{0}$ and, for any $s\in(0,1)$,
\begin{align}\label{eq:gtvchange}
	\nonumber
	 g\left(\vx^*+s \vv^*\right) 
	= & \beta  \sum_{i<\bar{i}}\left\|\vx^*_i+s\bar{\vxi}-\vx^*_{i+1}-s\bar{\vxi}\right\|_1+\beta \|\vx^*_{\bar{i}}+s\bar{\vxi}-\vx^*_{\bar{i}+1}\|_1
	+\beta \sum_{i>\bar{i}}\left\|\vx^*_i-\vx^*_{i+1}\right\|_1 \\	
	= & \beta  \sum_{i<\bar{i}}\left\|\vx^*_i-\vx^*_{i+1}\right\|_1+\beta (1-s)\|\vx^*_{\bar{i}}-\vx^*_{\bar{i}+1}\|_1
	+\beta \sum_{i>\bar{i}}\left\|\vx^*_i-\vx^*_{i+1}\right\|_1 
	= g\left(\vx^*\right)-s\beta\|\bar{\vxi}\|_1, 
\end{align}
where the first and last equalities follow from \eqref{eq:g} and the definition of $\vv^*$, and the second one is by~\eqref{eq:set-delta} and $s\in(0,1)$.
In addition, from \eqref{eq:fi} and \eqref{eq:partialhmax}, we have that, for any $\vz\in\mathbb{R}^{\bar d}$,  
\begin{equation}
	\label{eq:fi_infinitynorm}
\|\nabla f_i(\vz)\|_\infty=\frac{2 \epsilon}{\sqrt{m}} \left\|\nabla h_i\left(\frac{\sqrt{m}L_f \vz }{150\pi\epsilon}\right)\right\|_\infty
\leq \frac{50\pi\epsilon}{\sqrt{m}}.
\end{equation}
Moreover, by the definition of $\vv^*_i$ for $i=1,\dots,m$, we have 
$$
f_i(\vx^*_i+s \vv^*_i)- f_i(\vx^*_i)=s\nabla f_i(\vx^*_i+s' \vv^*_i)^\top\vv^*_i\leq s\|\nabla f_i(\vx^*_i+s' \vv^*_i)\|_\infty\|\vv^*_i\|_1 \overset{\eqref{eq:fi_infinitynorm}}\leq \frac{50s\pi\epsilon}{\sqrt{m}}\|\bar{\vxi}\|_1,
$$
where the equality holds from the mean value theorem for some $s'\in(0,s)$. The inequality above, together with \eqref{eq:f0}  and \eqref{eq:gtvchange}, implies
\begin{equation*}
	\begin{aligned}
		&\frac{1}{s}\bigg(F_0(\vx^*+s \vv^*)-F_0(\vx^*)\bigg)=\frac{1}{s}\bigg( f_0(\vx^*+s \vv^*)- f_0(\vx^*)+g(\vx^*+s \vv^*)-g(\vx^*)\bigg)\\
		=&\frac{1}{s}\bigg( \sum_{i=1}^m\big(f_i(\vx_i^*+s \vv_i^*)- f_i(\vx_i^*)\big)+g(\vx^*+s \vv^*)-g(\vx^*)\bigg)\\
		\leq& \frac{1}{s}\left(50\pi s\epsilon\sqrt{m}\|\bar{\vxi}\|_1-\beta s\|\bar{\vxi}\|_1\right)= \left(50\pi \epsilon\sqrt{m}-\beta \right)\|\bar{\vxi}\|_1.
	\end{aligned}
\end{equation*}
Taking the limit of the left-hand side of the inequality above as $s$ approaching zero, we have $$\left(50\pi \epsilon\sqrt{m}-\beta \right)\|\bar{\vxi}\|_1\geq F_0^{\prime}(\vx^*;\vv).$$ Thus by \eqref{eq:dderivegeq0epsilon} and the choice of $\vv^*$, we have 
$$
\left(50\pi \epsilon\sqrt{m}-\beta \right)\|\bar{\vxi}\|_1\geq -\widehat\epsilon\|\vv^*\|\geq -\widehat\epsilon\sqrt{\bar{i}}\|\bar{\vxi}\|
\geq -\widehat\epsilon\sqrt{m}\|\bar{\vxi}\|_1.
$$
This leads to a contradiction as $\beta > (50\pi+1)\epsilon\sqrt{m}$ from \eqref{eq:betachoice} and $\widehat\epsilon \le \epsilon$.  Therefore, the claim $\bar\vA\vx^*=\mathbf{0}$ is true. Thus the second condition in~\eqref{eq:KKTe} indicates $\|\vH\vx^*\|\leq\widehat\epsilon$, and we complete the proof. 
\end{proof}

By Lemma~\ref{cor:kktvio2}, if $\vx^*$ is not an $\epsilon$-stationary point of the auxiliary problem~\eqref{eq:model3}, it cannot be an $\epsilon$-stationary point of instance~$\mathcal{P}$. In other words, the number of oracles needed to find an $\epsilon$-stationary point of~$\mathcal{P}$ is at least the number of oracles needed to find an $\epsilon$-stationary point of problem~\eqref{eq:model3}. 
Note that the auxiliary problem~\eqref{eq:model3} of instance~$\mathcal{P}$ is the worst-case instance used in~\cite{sun2019distributed} to establish the lower-bound complexity for affinely constrained smooth optimization. In fact, according to~\cite{sun2019distributed}, any algorithm that can access $\nabla f_0$ and matrix-vector multiplication with $\vH$ and $\vH^\top$ at any historical solutions needs at least  $\Theta({\kappa(\vH) L_f \Delta_{f_0}} \epsilon^{-2})$ oracles to find an $\epsilon$-stationary point of \eqref{eq:model3}, where $\Delta_{f_0}:= f_0(\vx^{(0)}) - \inf_\vx f_0(\vx)$. However, we cannot directly apply the lower bound here because the problem~\eqref{eq:model} that we consider has both affine constraints and a non-smooth term, and our algorithms cannot apply $\vH$ and $\vH^\top$ for matrix-vector multiplications but instead can use $\vA$ and $\vA^{\top}$ as well as the proximal mapping of $g$ (see Assumption~\ref{ass:linearspan}). Next, we show that any algorithm under Assumption~\ref{ass:linearspan} also needs at least $\Theta({\kappa(\vH) L_f \Delta_{f_0}} \epsilon^{-2})$ oracles to find an $\epsilon$-stationary point of problem~\eqref{eq:model3}.

To do so, we need the following lemma, which is a direct consequence of~\cite[Lemma 3.3]{sun2019distributed} and the proof is provided for the sake of completeness. 

\begin{lemma}
	\label{lem:nablaf}
	Let $\{f_i\}_{i=1}^m$ be defined in~\eqref{eq:fi} with $\epsilon>0$. For any $\vz\in\RR^{\bar{d}}$, if $|[\vz]_{\bar{j}}|<\frac{150\pi\epsilon}{\sqrt{m} L_f}$ for some $\bar{j} \in \{1,2,\ldots,{\bar{d}}\}$, then 
	$
	\left\|\frac{1}{m}\sum_{i=1}^m\nabla f_i(\vz)\right\| > \frac{2 \epsilon}{\sqrt{m}}.
	$
\end{lemma}

\begin{proof} 	Let $\bar{\vz}=\frac{\sqrt{m} L_f \vz}{150\pi\epsilon}$. We first consider the  case where  $|[\vz]_j|<\frac{150\pi\epsilon}{\sqrt{m} L_f}$ for all $j=1,2,\ldots,\bar{j}$. In this case,  $| [\bar \vz]_1|=\frac{\sqrt{m} L_f |[\vz]_1|}{150\pi\epsilon} <1$. By \eqref{eq:fi}, we have 
	\begin{equation}
		\label{eq:nablaf1}
		\begin{aligned}
			\left\|\frac{1}{m}\sum_{i=1}^m\nabla f_i(\vz)\right\|  \geq \left| \frac{1}{m}\sum_{i=1}^m\left[\nabla f_i(\vz)\right]_1\right| = \left| \frac{2\epsilon}{m\sqrt{m}}\sum_{i=1}^m \left[\nabla {h}_i\left(\bar{\vz}\right)\right]_1\right|.
		\end{aligned}
	\end{equation}
	In addition, according to \eqref{eq:deriveh3}, we have
	\begin{equation}
		\label{eq:nablaf2}
		\begin{aligned}
			\frac{1}{m}\sum_{i=1}^m \left[\nabla {h}_i\left(\bar{\vz}\right)\right]_1 &=-\Psi(1) \Phi^{\prime}([\bar{\vz}]_1)+ \left[-\Psi^{\prime}(-[\bar \vz]_1) \Phi(-[\bar \vz]_2)-\Psi^{\prime}([\bar \vz]_1) \Phi([\bar \vz]_2)\right] \leq-\Psi(1) \Phi'([\bar{\vz}]_1)< -1,
		\end{aligned}
	\end{equation}
	where the first inequality comes from the non-negativity of $\Psi^{\prime}$ and $\Phi$ by Lemma~\ref{lem:functions}(a), and the second inequality is  by  Lemma~\ref{lem:functions}(b) and $|\bar \vz_1|<1$. Combing~\eqref{eq:nablaf1} and~\eqref{eq:nablaf2}
	yields the desired inequality.
	
	Second, we consider the case where there exists $j\in\{2,\ldots,\bar{j}\}$ such that $|[\vz]_j| <\frac{150\pi\epsilon}{\sqrt{m} L_f} \leq |[\vz]_{j-1}|$. In this case,  $|[\bar \vz]_j| <1 \leq |[\bar  \vz]_{j-1}|$. By~\eqref{eq:fi} again, we have 
	\begin{equation}
		\label{eq:nablaf1new}
		\left\|\frac{1}{m}\sum_{i=1}^m\nabla f_i(\vz)\right\|  \geq \left| \frac{1}{m}\sum_{i=1}^m\left[\nabla f_i(\vz)\right]_j\right| = \left| \frac{2\epsilon}{m\sqrt{m}}\sum_{i=1}^m \left[\nabla {h}_i\left(\bar{\vz}\right)\right]_j\right|.
	\end{equation}		
	According to \eqref{eq:deriveh1} and~\eqref{eq:deriveh2}, we have
	\begin{equation}
		\label{eq:nablaf3}
		\begin{aligned}
			\frac{1}{m}\sum_{i=1}^m \left[\nabla {h}_i\left(\bar{\vz}\right)\right]_j =& -\Psi(-[\bar{\vz}]_{j-1}) \Phi^{\prime}(-[\bar{\vz}]_{j})-\Psi([\bar{\vz}]_{j-1}) \Phi^{\prime}([\bar{\vz}]_{j})-\Psi^{\prime}(-[\bar{\vz}]_{j}) \Phi(-[\bar{\vz}]_{j+1})-\Psi^{\prime}([\bar{\vz}]_{j}) \Phi([\bar{\vz}]_{j+1})\\
			\leq &-\Psi(-[\bar{\vz}]_{j-1}) \Phi^{\prime}(-[\bar{\vz}]_{j})-\Psi([\bar{\vz}]_{j-1}) \Phi^{\prime}([\bar{\vz}]_{j}) = -\Psi(|[\bar{\vz}]_{j-1}|) \Phi^{\prime}([\bar{\vz}]_{j}) < -1,
		\end{aligned}
	\end{equation}
	where the first inequality comes from the nonnegativity of $\Psi^{\prime}$ and $\Phi$ by  Lemma~\ref{lem:functions}(a), the second equality holds by the fact that $\Phi^{\prime}(u)=\Phi^{\prime}(-u)$ and $\Psi(u)=0$ for all $u\leq 0$ from \eqref{eq:PsiPhi} and Lemma~\ref{lem:functions}(a), and the second inequality is  by  Lemma~\ref{lem:functions}(b) and the fact that $|[\bar{\vz}]_{j-1}|\geq 1$ and $|[\bar{\vz}]_{j}|<1$. Combining~\eqref{eq:nablaf1new} and~\eqref{eq:nablaf3} 	yields the desired inequality and completes the proof.		
\end{proof}

The following lemma provides a lower bound to the stationarity measure of a point $\vx$ as a solution to problem~\eqref{eq:model3}. 

\begin{lemma}
	\label{lem:kktvio}
	Let $\vx\in\mathbb{R}^d$ be given in~\eqref{eq:xblock}, $\vH$ in~\eqref{eq:matrixAstar}, and $\{f_i\}_{i=0}^m$ in~\eqref{eq:fi} and~\eqref{eq:f0}. 
	Then
	$$
	\max\left\{\left\|\vH\vx\right\|, \min _{\vgamma}\left\|\nabla  f_0(\vx)+\vH^{\top} \vgamma\right\|\right\} \geq \frac{\sqrt{m}}{2}\left\|\frac{1}{m}\sum_{i=1}^m\nabla f_i(\bar{\vx})\right\|, \text{ with } \bar{\vx}:= \frac{1}{m} \sum_{i=1}^m \vx_i.
	$$
\end{lemma}

\begin{proof}
	By simple calculation and the fact $\mathbf{Null}(\vH)=\left\{ \mathbf{1}_{m} \otimes \vu:\vu\in\RR^{\bar{d}}\right\}$, we obtain  		
		\begin{align}
			 \min_{\vgamma}\left\|\nabla  f_0(\vx)+\vH^{\top} \vgamma\right\|^2 
			= &\left\|\mathbf{Proj}_{\mathbf{Null}(\vH)}(\nabla  f_0(\vx))\right\|^2 =\frac{1}{m}\left\|\sum_{i=1}^m \nabla f_i\left(\vx_i\right)\right\|^2, \label{eq:kktvio1} \\
			\left\|\vH\vx\right\|^2=&m^2L_f^2\sum_{i=1}^{m-1} \left\|\vx_{i}-\vx_{i+1}\right\|^2. \label{eq:kktvio1-Hx}
		\end{align}
By the $L_f$-Lipschitz continuity of $\nabla  f_0$, 
we  have 
	\begin{align*}
		&\frac{1}{2}\left\|\frac{1}{m}\sum_{i=1}^m\nabla f_i(\bar{\vx})\right\|^2 - \left\|\frac{1}{m} \sum_{i=1}^m \nabla f_i\left(\vx_i\right)\right\|^2\leq \left\|\frac{1}{m} \sum_{i=1}^m\left(\nabla f_i(\bar{\vx})-\nabla f_i\left(\vx_i\right)\right)\right\|^2 \\
		\leq &\frac{1}{m} \sum_{i=1}^m\left\|\nabla f_i\left(\bar{\vx}\right)-\nabla f_i\left(\vx_i\right)\right\|^2
		\leq \frac{1}{m} \sum_{i=1}^m L_f^2\left\|\bar{\vx}-\vx_i\right\|^2 
		\stackrel{(a)}{\leq} \frac{L_f^2}{m^2} \sum_{i=1}^{m} \sum_{j=1}^m\left\|\vx_j-\vx_i\right\|^2\\
		\stackrel{(b)}{\leq} &\frac{L_f^2}{m^2} \sum_{i=1}^{m} \sum_{j=1}^m \left[|{j-i}|\sum_{k=\min\{i,j\}}^{\max\{i,j\}-1}\left\|\vx_k-\vx_{k+1}\right\|^2\right]\stackrel{(c)}{\leq} \frac{L_f^2}{m^2} \sum_{i=1}^{m} \sum_{j=1}^m m\sum_{k=1}^{m-1} \left\|\vx_{k}-\vx_{k+1}\right\|^2=   mL_f^2 \sum_{i=1}^{m-1} \left\|\vx_{i}-\vx_{i+1}\right\|^2, 
	\end{align*}
where (a) comes from $\|\bar{\vx}-\vx_i\|^2 = \frac{1}{m^2}\|\sum_{j=1}^m \vx_j- \vx_i\|^2\leq \frac{1}{m}  \sum_{j=1}^m \|\vx_j- \vx_i\|^2 $, (b) results from the fact that $\|\vx_j- \vx_i\|^2 \leq |j-i| \sum_{k=\min\{i,j\}}^{\max\{i,j\}-1}\left\|\vx_k-\vx_{k+1}\right\|^2$, (c) holds by $|j-i|<m$, $\max\{i,j\}\leq m$ and $\min\{i,j\}\ge 1$.
 
Hence, by~\eqref{eq:kktvio1} and~\eqref{eq:kktvio1-Hx} and the fact $a+b \le 2\max\{a,  b\}$ for any $a, b\in \RR$, we obtain the desired result from the inequality above and  complete the proof. 
\end{proof}

The previous two lemmas imply that if there exists $\bar{j}\in\{1,2,\ldots,{\bar{d}}\}$ such that $[\bar{\vx}]_{\bar{j}}=0$, where $\bar{\vx}~=~\frac{1}{m} \sum_{i=1}^m \vx_i$. Then $\vx$ cannot be an $\epsilon$-stationary point of the auxiliary problem~\eqref{eq:model3} of instance~$\mathcal{P}$. 

\subsection{Lower Bound of Oracle Complexity}

In this subsection, we provide a lower bound of the oracle complexity of FOMs for solving problem~\eqref{eq:model} under Assumptions~\ref{assume:problemsetup} and \ref{ass:linearspan}, by showing that 
a large number of oracles will be needed to find a (near) $\epsilon$-stationary point of instance~$\mathcal{P}$. 	

For any integer $t\ge0$, let 
\begin{equation}\label{eq:t-th-iter}
	\vx^{(t)}=(\vx_1^{(t)\top}, \ldots,\vx_m^{(t)\top})^\top \text{ with each }\vx_i^{(t)}\in\RR^{\bar d}, \text{ and } \bar{\vx}^{(t)}= \frac{1}{m} \sum_{i=1}^m \vx_i^{(t)}
\end{equation}
be the $t$-th iterate of an algorithm and the block average. 
Without loss of generality, we assume $\vx^{(0)}=\mathbf{0}$. Otherwise, we can change variable\footnote{This involves changing functions $f(\vx)$  to $f(\vx-\vx^{(0)})$, $g(\vx)$ to  $g(\vx-\vx^{(0)})$, and $\vA\vx=\vb$ to  $\vA(\vx-\vx^{(0)})=\vb$ in instance $\cP$. Then, the resulted instance becomes $\min_{ \vx\in\RR^d}  F_0( \vx):= f_0( \vx-\vx^{(0)}) + g( \vx-\vx^{(0)}), \ \st  \vA (\vx-\vx^{(0)}) +  \vb=\mathbf{0}$.} $\vx$ in instance~$\cP$ to $\vx-\vx^{(0)}$, and process the rest of the proof with this new resulting instance. 
Our lower bound will be established based on the fact that, if $t$ is not large enough, $\textnormal{supp}(\bar{\vx}^{(t)})\subset\{1,\dots,\bar{j}-1\}$ for some $\bar{j}\in\{1,2,\ldots,{\bar{d}}\}$. Hence, $[\bar{\vx}^{(t)}]_{\bar{j}}=0$ and $\vx^{(t)}$ cannot be a (near) $\epsilon$-stationary point due to Lemmas \ref{lem:nablaf} and \ref{lem:kktvio}.
According to Assumption~\ref{ass:linearspan},  $\textnormal{supp}(\bar{\vx}^{(t)})$ is influenced by $\textnormal{supp}(\nabla f_i(\vx_i^{(t)}))$, which is characterized by the following lemma.
					
\begin{lemma}
	\label{lem:iterateguess}
Let $\{f_i\}_{i=1}^m$ be defined in~\eqref{eq:fi}. Given any $\bar{j}\in\{1,\dots,\bar{d}\}$ and $\vz\in\RR^{\bar{d}}$ with $\textnormal{supp}(\vz)\subset\{1,\dots,\bar{j}-1\}$\footnote{When $\bar{j}=1$, this means $\textnormal{supp}(\vz)=\emptyset$ and $\vz=\mathbf{0}$.}, it holds that 
\begin{enumerate}
	\item   When $\bar{j}=1$, 
	$
	\textnormal{supp}(\nabla f_i(\vz))\subset\{1\}$, for any $i\in[1,m];
	$ 
	\item  When $\bar{j}$ is even, 
	$$
	\textnormal{supp}(\nabla f_i(\vz))\subset\left\{
	\begin{array}{ll}
	\{1,\dots,\bar{j}\}, & \text{ if }\ i\in\left[1, \frac{m}{3}\right],\\[0.1cm]
	\{1,\dots,\bar{j}-1\}, & \text{ if }\ i\in\left[\frac{m}{3}+1, m\right];
	\end{array}
	\right.
	$$
	\item  When $\bar{j}$ is odd and $\bar{j}\neq 1$, 
	$$
	\textnormal{supp}(\nabla f_i(\vz))\subset\left\{
	\begin{array}{ll}
		\{1,\dots,\bar{j}-1\}, & \text{ if }\ i\in\left[1, \frac{2m}{3}\right]\\[0.1cm]
		\{1,\dots,\bar{j}\}, & \text{ if }\ i\in\left[\frac{2m}{3}+1, m\right].
	\end{array}
	\right.
	$$
\end{enumerate}
\end{lemma}

\begin{proof}
	Denote $\bar \vz = \frac{\sqrt{m} L_f \vz}{150\pi\epsilon}$ and recall definition~\eqref{eq:grad-f-i-chain}.
 %
 	By Lemma~\ref{lem:functions}(a), we have 
 	\begin{eqnarray}
 		\label{eq:zeroPsi}
 		\Psi\left(-[\bar \vz]_j\right)=\Psi\left([\bar \vz]_j\right)=
 		\Psi'\left(-[\bar \vz]_j\right)=\Psi'\left([\bar \vz]_j\right)=0, \forall\, j\geq \bar{j}.
 	\end{eqnarray}
Therefore, by definitions~\eqref{eq:deriveh3}--\eqref{eq:deriveh2}, the support of $\vz$ leads to the following structure of 
	$[\nabla f_i(\vz)]_j$ for $j\geq \bar{j}$:
	
	\begin{itemize}
			\item 
	If $j=1$,
	$
	[\nabla f_i(\vz)]_j=	-\Psi(1) \Phi^{\prime}\left([\bar \vz]_j\right)$, for any $i \in\left[1, m\right];
	$
		\item 
	If $j$ is even, 
	$$
[\nabla f_i(\vz)]_j=\left\{
\begin{array}{ll}
-\frac{6\epsilon}{\sqrt{m}}\left[\Psi\left(-[\bar \vz]_{j-1}\right) \Phi^{\prime}\left(-[\bar \vz]_{j}\right)+\Psi\left([\bar \vz]_{j-1}\right) \Phi^{\prime}\left([\bar \vz]_{j}\right)\right],& \text{ for }  i \in\left[1, \frac{m}{3}\right],\\[0.1cm]
0, & \text{ for } i \in\left[\frac{m}{3}+1, m\right];
\end{array}
\right.
	$$
	\normalsize
	\item 
	If $j$ is odd and $j\neq 1$, 
	$$
	[\nabla f_i(\vz)]_j=\left\{
	\begin{array}{ll}
		0, & \text{ for } i   \in\left[1, \frac{2m}{3}\right],\\[0.1cm]
-\frac{6\epsilon}{\sqrt{m}}\left[\Psi\left(-[\bar \vz]_{j-1}\right) \Phi^{\prime}\left(-[\bar \vz]_{j}\right)+\Psi\left([\bar \vz]_{j-1}\right) \Phi^{\prime}\left([\bar \vz]_{j}\right)\right],& \text{ for } i  \in\left[\frac{2m}{3}+1, m\right].
	\end{array}
	\right.
	$$
	\normalsize
\end{itemize} 	
Since $\Psi\left(-[\bar \vz]_{j-1}\right)=\Psi\left([\bar \vz]_{j-1}\right)=0$ for any $j > \bar{j}$, the structures above imply $[\nabla f_i(\vz)]_j = 0,\forall\, j > \bar j$ and thus give the desired claims.
\end{proof}

According to the structure of $\vA$ given in~\eqref{eq:AandAbar},  $\textnormal{supp}((\vA^{\top}\vA\vx)_i)$ is determined by $\textnormal{supp}(\vx_{i-1})$, $\textnormal{supp}(\vx_{i})$ and $\textnormal{supp}(\vx_{i+1})$. Also, $\textnormal{supp}(\prox_{\eta g}(\vx))$ has a similar property according to the definition of $g$  in~\eqref{eq:g}. These properties are formally stated in the following lemma.

\begin{lemma}
	\label{lem:supp}
Let $\vx$ be the structured vector given in~\eqref{eq:xblock}, $\vA$ in~\eqref{eq:AandAbar}, and $g$ be given in~\eqref{eq:g}. Define $\vx_0=\vx_{m+1}=\mathbf{0}\in\RR^{{\bar{d}}}$. The following statements hold:
	\begin{enumerate}
		\item[\textnormal{(a)}] Let  $\widehat\vx=\vA^{\top}\vA\vx=(\widehat\vx_1\zz, \ldots,\widehat\vx_m\zz)\zz$ with each $\widehat\vx_i\in\RR^{{\bar{d}}}$.  Then 
		\begin{equation}\label{eq:supp-relation}
		\textnormal{supp}(\widehat\vx_i)\subset \textnormal{supp}(\vx_{i-1})\cup\textnormal{supp}(\vx_{i})\cup\textnormal{supp}(\vx_{i+1}),\, \forall\, i\in [1, m].
		\end{equation}		
		\item[\textnormal{(b)}] For any $\eta>0$, let $\widetilde\vx=\prox_{\eta g}(\vx)=(\widetilde\vx_1\zz, \ldots,\widetilde\vx_m\zz)\zz$ with each $\widetilde\vx_i\in\RR^{{\bar{d}}}$. Then 
		$$\textnormal{supp}(\widetilde\vx_i)\subset \textnormal{supp}(\vx_{i-1})\cup\textnormal{supp}(\vx_{i})\cup\textnormal{supp}(\vx_{i+1}),\, \forall\, i\in [1, m].$$
	\end{enumerate}
\end{lemma}

\begin{proof}
	(a) 
The relation in~\eqref{eq:supp-relation} immediately follows from the observation
\begin{equation}\label{eq:ata-struc}
\newcommand{\AtA}{\left[ 
\begin{array}{cccc}
\vB & & & \\
& \vB & & \\
& & \ddots &\\
& & & \vB
\end{array}
\right]
}
\newcommand{\Bmat}{\left[ 
\begin{array}{rrrrr} 
\vI_{\bar d} & - \vI_{\bar d} & & & \\
- \vI_{\bar d} & 2 \vI_{\bar d} & - \vI_{\bar d} & \\
& \ddots & \ddots & \ddots & \\
&  &	- \vI_{\bar d} & 2 \vI_{\bar d} & - \vI_{\bar d}  \\
& & & - \vI_{\bar d} & \vI_{\bar d}
		\end{array} \right]
}
\vA^\top\vA =  \left. \AtA 
		\right\}\text{$3m_2$ blocks}, \text{ with } 
\vB = m^2 L_f^2 
\left.
		\,\smash[b]{\underbrace{\!\Bmat\!}_{\textstyle\text{$m_1$ blocks}}}\,
		\right\}\text{$m_1$ blocks.}
		\vphantom{\underbrace{\Bmat}_{\textstyle\text{$m$ blocks}}}	
\end{equation}
	
	(b) Given any $x_1$ and $x_2$ in $\mathbb{R}$ and any $c>0$, consider the following optimization problem in $\mathbb{R}^2$:
	\begin{eqnarray*}
		(\widetilde x_1, \widetilde x_2)=\argmin_{z_1,z_2\in\mathbb{R}}\frac{1}{2}(z_1-x_1)^2+\frac{1}{2}(z_2-x_2)^2+c|z_1-z_2|.
	\end{eqnarray*}
	The optimal solution of this problem is 
	\begin{eqnarray}
		\label{eq:2dproxg}
		(\widetilde x_1, \widetilde x_2)=\left\{
		\begin{array}{ll}
			( (x_1+x_2)/2,  (x_1+x_2)/2),&\text{ if } |x_1-x_2|\leq 2c\\[0.1cm]
			( x_1-c\cdot\text{sign}(x_1-x_2),  x_2+c\cdot\text{sign}(x_1-x_2)),&\text{ if } |x_1-x_2|> 2c.
		\end{array}
		\right.
	\end{eqnarray}
	Recall the definition of $g$ in \eqref{eq:g} and $\textbf{prox}_{\eta g}$, we obtain that 
	$$
	\prox_{\eta g}(\vx) = \argmin_{\vy} \eta\beta\sum_{i\in\mathcal{M}}  \|\vx_{i }-\vx_{i+1}\|_1 + \frac{1}{2}\|\vx-\vy\|^2.
	$$
	It then holds that 
	\begin{eqnarray*}
		\widetilde\vx_i = 
		\left\{
		\begin{array}{ll}
			 \argmin_{\vy_i} \frac{1}{2}\|\vx_i-\vy_i\|^2, & \text{ if }i-1,i\notin\mathcal{M},\\\null
			\argmin_{\vy_i} \eta\beta \|\vx_{i }-\vx_{i-1}\|_1 + \frac{1}{2}\|\vx_i-\vy_i\|^2, &\text{ if }i-1\in\mathcal{M},\\\null
			\argmin_{\vy_i} \eta\beta \|\vx_{i }-\vx_{i+1}\|_1 + \frac{1}{2}\|\vx_i-\vy_i\|^2 ,&\text{ if }i\in\mathcal{M}.
		\end{array}
		\right.
	\end{eqnarray*}
	By~\eqref{eq:2dproxg} and the separability property of $\|\cdot\|_1$ and $\|\cdot\|^2$, 
	we have that for any $j\in\{1,\dots,\bar{d}\}$, 
	\begin{eqnarray*}
		[\widetilde\vx_i]_j=
		\left\{
		\begin{array}{ll}
			[\vx_i]_j, &\text{ if }i-1,i\notin\mathcal{M},\\\null
			([\vx_{i-1}]_j+[\vx_i]_j)/2, &\text{ if }i-1\in\mathcal{M}\text{ and }\left|[\vx_{i-1}]_j-	[\vx_{i}]_j\right|\leq \eta\beta,\\\null
			([\vx_{i}]_j+[\vx_{i+1}]_j)/2, &\text{ if }i\in\mathcal{M}\text{ and }\left|[\vx_{i}]_j-	[\vx_{i+1}]_j\right|\leq \eta\beta,\\\null
			[\vx_i]_j+\eta\beta\cdot\text{sign}([\vx_{i-1}]_j-[\vx_i]_j), &\text{ if }i-1\in\mathcal{M}\text{ and }\left|[\vx_{i-1}]_j-	[\vx_{i}]_j\right|> \eta\beta,\\\null
			[\vx_i]_j-\eta\beta\cdot\text{sign}([\vx_{i-1}]_j-[\vx_i]_j), &\text{ if }i\in\mathcal{M}\text{ and }\left|[\vx_{i-1}]_j-	[\vx_{i}]_j\right|> \eta\beta,
		\end{array}
		\right.
	\end{eqnarray*}
	which implies 
	\begin{eqnarray*}
		\textnormal{supp}(\widetilde\vx_i)&\subset&
		\left\{
		\begin{array}{ll}
			\textnormal{supp}(\vx_{i}),&\text{ if }i-1,i\notin\mathcal{M},\\
			\textnormal{supp}(\vx_{i-1})\cup\textnormal{supp}(\vx_{i}),&\text{ if }i-1\in\mathcal{M},\\
			\textnormal{supp}(\vx_{i})\cup\textnormal{supp}(\vx_{i+1}),&\text{ if }i\in\mathcal{M}.
		\end{array}
		\right.
	\end{eqnarray*}
	The proof is then completed.
\end{proof}

Now we are ready to show the following result on how fast $\textnormal{supp}(\bar{\vx}^{(t)})$ can expand with $t$.

\begin{proposition}
	\label{thm:iterateguess}
Under Assumption~\ref{ass:linearspan}, suppose an algorithm is applied to instance~$\mathcal{P}$ from $\vx^{(0)}=\mathbf{0}$ and generates a sequence $\{\vx^{(t)}\}_{t\geq0}$. By notations in~\eqref{eq:t-th-iter}, 
it holds for any $\bar{j}\in\{2,3,\dots,\bar{d}\}$ that 
\begin{eqnarray}
	\label{eq:outerinduction}
	\textnormal{supp}(\vx_i^{(t)})\subset\{1,\dots,\bar{j}-1\} \text{ for }i=1,\dots,m\text{ and }t\leq 1+ m(\bar{j}-2)/6.
\end{eqnarray}
\end{proposition}
 
\begin{proof}
We prove the claim by induction on $\bar{j}$. 
%
Let $\vxi^{(t)}=(\vxi^{(t)\top}_1, \ldots,\vxi^{(t)\top}_m)\zz$ with $\vxi^{(t)}_i\in\RR^{{\bar{d}}}$ and $\vzeta^{(t)}=(\vzeta^{(t)\top}_1, \ldots,\vzeta^{(t)\top}_m)\zz$  with $\vzeta^{(t)}_i\in\RR^{{\bar{d}}}$ be the vectors defined in Assumption~\ref{ass:linearspan} for $t\geq1$. 
Since $\vx^{(0)}=\mathbf{0}$, we have $\textnormal{supp}(\nabla f_i(\vx^{(0)}_i))\subset\{1\}, \forall\, i$ from Lemma~\ref{lem:iterateguess}. Notice $\vb = \vzero$. Hence, 
$\textnormal{supp}(\vxi_i^{(1)})\subset\{1\}, \forall\, i$, which further indicates $\textnormal{supp}(\vzeta_i^{(1)})\subset\{1\}, \forall\, i$ by Lemma~\ref{lem:supp}(b), and thus $\textnormal{supp}(\vx_i^{(1)})\subset\{1\}, \forall\, i$. This proves the claim in~\eqref{eq:outerinduction} for $\bar{j}=2$.
%
Now suppose that the claim~\eqref{eq:outerinduction} holds for some $\bar{j}\geq 2$. We go to prove it for $\bar{j}+1$.
	
According to the hypothesis of the induction, we have 
\begin{equation}\label{eq:induction-result}
\textnormal{supp}(\vx_i^{(r)})\subset\{1,\dots,\bar{j}-1\},  \forall\, i\in [1, m] \text{ and } \forall\, r\leq \bar t:=1+ m(\bar{j}-2)/6.
\end{equation} 
Below we let $\widehat\vx^{(r)}=\vA\zz \vA\vx^{(r)}$ for any $r\geq0$ and consider two cases: $\bar{j}$ is even and $\bar{j}$ is odd. 
 	
\textbf{Case 1}: Suppose $\bar{j}$ is even.  We claim that, for $s=0,1,\dots,\frac{m}{6}$,
\begin{eqnarray}
	\label{eq:innerinduction1}
\textnormal{supp}(\vx_i^{(r)})\subset\left\{
\begin{array}{ll}
	\{1,\dots,\bar{j}\}, &\text{ if } i\in\left[1, \frac{m}{3}+2s\right],\\[0.1cm]
	\{1,\dots,\bar{j}-1\}, & \text{ if } i\in\left[\frac{m}{3}+2s+1, m\right],
\end{array}
\right.\ \forall\, r \le \bar t+s.
\end{eqnarray}
Notice~\eqref{eq:induction-result} implies~\eqref{eq:innerinduction1} for $s=0$. Suppose~\eqref{eq:innerinduction1} holds for some integer $s \in [0, \frac{m}{6}]$. Then by Lemma~\ref{lem:iterateguess} and $\frac{m}{3} + 2s \le \frac{2m}{3}$, it holds 
$$
\textnormal{supp}(\nabla f_i(\vx_i^{(r)}))\subset\left\{
\begin{array}{ll}
	\{1,\dots,\bar{j}\}, & \text{ if } i\in\left[1, \frac{m}{3}+2s\right],\\[0.1cm]
	\{1,\dots,\bar{j}-1\}, &\text{ if }  i\in\left[\frac{m}{3}+2s+1, m\right],
\end{array}
\right.\  \forall\, r \le \bar t+s.
$$
In addition, by  Lemma~\ref{lem:supp}(a), we have from~\eqref{eq:innerinduction1} that
$$
\textnormal{supp}(\widehat\vx_i^{(r)})\subset\left\{
\begin{array}{ll}
	\{1,\dots,\bar{j}\}, &\text{ if } i\in\left[1, \frac{m}{3}+2s+1\right],\\[0.1cm]
	\{1,\dots,\bar{j}-1\}, &\text{ if } i\in\left[\frac{m}{3}+2s+2, m\right],
\end{array}
\right.\ \forall\, r \le \bar t+s.
$$
 Hence, by Assumption~\ref{ass:linearspan}, we have
$$
\textnormal{supp}(\vxi_i^{(\bar t+s+1)})\subset\left\{
\begin{array}{ll}
	\{1,\dots,\bar{j}\}, & \text{ if } i\in\left[1, \frac{m}{3}+2s+1\right],\\[0.1cm]
	\{1,\dots,\bar{j}-1\},& \text{ if } i\in\left[\frac{m}{3}+2s+2, m\right],
\end{array}
\right.
$$
and thus it follows from Lemma~\ref{lem:supp}(b) that
$$
\textnormal{supp}(\vzeta_i^{(\bar t+s+1)})\subset\left\{
\begin{array}{ll}
	\{1,\dots,\bar{j}\},& \text{ if } i\in\left[1, \frac{m}{3}+2s+2\right],\\[0.1cm]
	\{1,\dots,\bar{j}-1\}, &\text{ if } i\in\left[\frac{m}{3}+2s+3, m\right].
\end{array}
\right.
$$
Now since $\vx_i^{(\bar t+s+1)} \in \mathbf{span}\left(\{\vxi_i^{(\bar t+s+1)}, \vzeta_i^{(\bar t+s+1)}\}\right)$ by Assumption~\ref{ass:linearspan}, we have
$$
\textnormal{supp}(\vx_i^{(\bar t+s+1)})\subset\left\{
\begin{array}{ll}
	\{1,\dots,\bar{j}\}, & \text{ if } i\in\left[1, \frac{m}{3}+2s+2\right],\\[0.1cm]
	\{1,\dots,\bar{j}-1\}, & \text{ if } i\in\left[\frac{m}{3}+2s+3, m\right],
\end{array}
\right.
$$
which means~\eqref{eq:innerinduction1} holds for $s+1$ as well. By induction,~\eqref{eq:innerinduction1} holds for $s=0,1,\dots,\frac{m}{6}$. Let $s=\frac{m}{6}$ in~\eqref{eq:innerinduction1}. We have $\textnormal{supp}(\vx_i^{(r)})\subset\{1,\dots,\bar{j}\}$ for any $i$ and $r\leq \bar t+\frac{m}{6} = 1+ m(\bar{j}-2)/6+\frac{m}{6}=1+m(\bar{j}-1)/6$. 

\textbf{Case 2}: Suppose $\bar{j}$ is odd.  We claim that, for $s=0,1,\dots,\frac{m}{6}$,
\begin{eqnarray}
	\label{eq:innerinduction2}
	\textnormal{supp}(\vx_i^{(r)})\subset\left\{
	\begin{array}{ll}
		\{1,\dots,\bar{j}-1\},& \text{ if } i\in\left[1, \frac{2m}{3}-2s\right],\\[0.1cm]
		\{1,\dots,\bar{j}\},& \text{ if } i\in\left[\frac{2m}{3}-2s+1, m\right],
	\end{array}
	\right. \ \forall\, r \le \bar t+s.
\end{eqnarray}
Again~\eqref{eq:induction-result} implies~\eqref{eq:innerinduction2} for $s=0$. Suppose it holds for an integer $s\in [0, \frac{m}{6}]$. 
Then by Lemma~\ref{lem:iterateguess}, 
$$
\textnormal{supp}(\nabla f_i(\vx_i^{(r)}))\subset\left\{
\begin{array}{ll}
	\{1,\dots,\bar{j}-1\}, & \text{ if }  i\in\left[1, \frac{2m}{3}-2s\right],\\[0.1cm]
	\{1,\dots,\bar{j}\}, & \text{ if }  i\in\left[\frac{2m}{3}-2s+1, m\right],
\end{array}
\right. \ \forall\, r \le \bar t+s.
$$
In addition, by  Lemma~\ref{lem:supp}(a) and~\eqref{eq:innerinduction2}, we have 
$$
\textnormal{supp}(\widehat\vx_i^{(r)})\subset\left\{
\begin{array}{ll}
	\{1,\dots,\bar{j}-1\},& \text{ if } i\in\left[1, \frac{2m}{3}-2s-1\right],\\[0.1cm]
	\{1,\dots,\bar{j}\},& \text{ if } i\in\left[\frac{2m}{3}-2s, m\right],
\end{array}
\right. \ \forall\, r \le \bar t+s.
$$
Hence, by Assumption~\ref{ass:linearspan}, we have
$$
\textnormal{supp}(\vxi_i^{(\bar t+s+1)})\subset\left\{
\begin{array}{ll}
	\{1,\dots,\bar{j}-1\},& \text{ if } i\in\left[1, \frac{2m}{3}-2s-1\right],\\[0.1cm]
	\{1,\dots,\bar{j}\},& \text{ if } i\in\left[\frac{2m}{3}-2s, m\right],
\end{array}
\right.
$$
and then it follows from Lemma~\ref{lem:supp}(b) that
$$
\textnormal{supp}(\vzeta_i^{(\bar t+s+1)})\subset\left\{
\begin{array}{ll}
	\{1,\dots,\bar{j}-1\},&\text{ if } i\in\left[1, \frac{2m}{3}-2s-2\right],\\[0.1cm]
	\{1,\dots,\bar{j}\},&\text{ if } i\in\left[\frac{2m}{3}-2s-1, m\right].
\end{array}
\right.
$$
Again since $\vx_i^{(\bar t+s+1)} \in \mathbf{span}\left(\{\vxi_i^{(\bar t+s+1)}, \vzeta_i^{(\bar t+s+1)}\}\right)$ by Assumption~\ref{ass:linearspan}, we have
$$
\textnormal{supp}(\vx_i^{(\bar t+s+1)})\subset\left\{
\begin{array}{ll}
	\{1,\dots,\bar{j}-1\},&\text{ if } i\in\left[1, \frac{2m}{3}-2s-2\right],\\[0.1cm]
	\{1,\dots,\bar{j}\},&\text{ if } i\in\left[\frac{2m}{3}-2s-1, m\right],
\end{array}
\right.
$$
which means~\eqref{eq:innerinduction2} holds for $s+1$ as well. By induction,~\eqref{eq:innerinduction2} holds for $s=0,1,\dots,\frac{m}{6}$. Let $s=\frac{m}{6}$ in~\eqref{eq:innerinduction2}. We have $\textnormal{supp}(\vx_i^{(r)})\subset\{1,\dots,\bar{j}\}$ for any $i$ and $r\leq \bar t+\frac{m}{6} = 1+ m(\bar{j}-2)/6+\frac{m}{6}=1+m(\bar{j}-1)/6$. 

Therefore, we have proved that~\eqref{eq:outerinduction} holds for $\bar{j}+1$, when $\bar j$ is either even or odd. By induction,~\eqref{eq:outerinduction} holds for any integer $\bar j \in [2, \bar d]$, and we complete the proof.
\end{proof}
					
					Finally, we are ready to give our main result about  the lower bound  of oracle complexity.
					
					\begin{theorem}
						\label{thm:lower}
Let $\epsilon>0$ and $L_f > 0$ be given.				Suppose an algorithm  is applied to problem~\eqref{eq:model} that satisfies Assumption~\ref{assume:problemsetup} 
					and generates a sequence $\{\vx^{(t)}\}_{t\geq0}$ that satisfies Assumption~\ref{ass:linearspan}. Then for any $\omega\in [0, \frac{150\pi \epsilon}{L_f})$, there exists an instance of problem~\eqref{eq:model}, i.e., instance~$\mathcal{P}$ in Definition~\ref{def:hardinstance}, such that the algorithm requires at least 
					$\left \lceil \frac{\kappa([\bar{\vA}; \vA]) L_f \Delta_{F_0}}{36000\pi^2} \epsilon^{-2} \right\rceil$ 
					oracles to obtain a point that is $\omega$-close to an $\epsilon$-stationary point of instance~$\mathcal{P}$, where $\Delta_{F_0}=F_0(\vx^{(0)})-\inf_\vx F_0(\vx)$. 
					\end{theorem}
					
					\begin{proof}
As we discussed below~\eqref{eq:t-th-iter}, 	we assume $\vx^{(0)}=\mathbf{0}$ without loss of generality. 
%
						Thus by notation in~\eqref{eq:t-th-iter}, Proposition~\ref{thm:iterateguess} indicates that 	$\textnormal{supp}(\vx_i^{(t)})\subset\{1,\dots,\bar{d}-1\}$ for any $i \in [1,m]$ and any $t\leq 1+ m(\bar{d}-2)/6$, which means $[\bar{\vx}^{(t)}]_{\bar{d}}=0$ if  $t\leq 1+ m(\bar{d}-2)/6$, where $\bar{\vx}^{(t)}= \frac{1}{m} \sum_{i=1}^m \vx_i^{(t)}$. 
						
On the other hand, suppose $\vx^*$ with the structure as in~\eqref{eq:t-th-iter} is an $\epsilon$-stationary point of instance~$\mathcal{P}$. Then by 	Lemma~\ref{cor:kktvio2}, it must also be an $\epsilon$-stationary point of~\eqref{eq:model3}. Hence, by  Lemmas~\ref{lem:nablaf} and~\ref{lem:kktvio}, we have $\left|[\bar\vx^*]_j\right| \ge \frac{150\pi\epsilon}{\sqrt{m} L_f}$ for all $j=1,\ldots,\bar d$, where $\bar{\vx}^*= \frac{1}{m} \sum_{i=1}^m \vx_i^*$. Therefore, by the convexity of the square function, it follows that
$$\|\vx^{(t)} - \vx^*\|^2 \ge \sum_{i=1}^m \left([\vx_i^{(t)}]_{\bar d} - [\vx_i^*]_{\bar d}\right)^2 \ge m \left([\bar\vx^{(t)}]_{\bar d} - [\bar\vx^*]_{\bar d}\right)^2 \ge m \left(\frac{150\pi\epsilon}{\sqrt{m} L_f}\right)^2 > \omega^2,$$
and thus $\vx^{(t)}$ is not $\omega$-close to $\vx^*$ if  $t\leq 1+ m(\bar{d}-2)/6$.
						
						Moreover, by 				
						Lemma~\ref{cor:boundf}(a) and the fact that $g(\vx^{(0)})=0$ and $g(\vx) \ge 0, \forall\, \vx$, it holds that
						$$ {\bar{d}}\geq \frac{L_f \left( F_0(\vx^{(0)})-\inf_{\vx} F_0\left(\vx\right)\right)}{3000\pi^2} \epsilon^{-2}=\frac{L_f\Delta_{F_0}}{3000\pi^2}\epsilon^{-2}.
						$$
						In other words, in order for $\vx^{(t)}$ to be $\omega$-close to an $\epsilon$-stationary point of instant~$\mathcal{P}$, the algorithm needs at least $t=2+m({\bar{d}}-2)/6$ oracles.
We complete the proof by noticing
						\begin{equation*}
							2+m({\bar{d}}-2)/6\geq m{\bar{d}}/12\geq\frac{m L_f \Delta_{F_0}}{36000\pi^2} \epsilon^{-2}>\frac{\kappa([\bar{\vA}; \vA]) L_f \Delta_{F_0}}{36000\pi^2\epsilon^2},
						\end{equation*}
						 where the first inequality is because ${\bar{d}}\geq 5$, and the last one is by Lemma~\ref{lem:condH}.
					\end{proof}
					
						

	\section{Lower Bound of Oracle Complexity for Problems~\eqref{eq:model} and~\eqref{eq:model-spli} under Assumption \ref{ass:linearspan3}}
				\label{sec:extension}
Under Assumption~\ref{ass:linearspan}, an algorithm is allowed to call the operator $\prox_{\eta g}(\cdot)$ that may not be easy when $g$ has the structure as that in Assumption~\ref{assume:problemsetup}. Calculating $\prox_{\eta g}(\cdot)$ to a high accuracy or exactly may require many (or even infinitely many) calls to $\bar{\vA}$ and $\bar{\vb}$. In contrast, $\bar g$ 
is simpler than $g$, making $\prox_{\eta \bar g}(\cdot)$ easier to compute such as when $\bar g(\cdot) = \lambda\|\cdot\|_1$ for some $\lambda >0$ as in instance $\cP$. 
	These observations motivate us to reformulate \eqref{eq:model} into~\eqref{eq:model-spli} and seek an $\epsilon$-stationary point of \eqref{eq:model-spli} under Assumption~\ref{ass:linearspan3} which only requires the computation of $\prox_{\eta \bar g}(\cdot)$.
Furthermore, the approach that is designed based on \eqref{eq:model-spli} exhibits the potential to find a near $\epsilon$-stationary point of problem~\eqref{eq:model}, as shown in the next section. 

Consequently, two intriguing questions arise: (i) whether finding an $\epsilon$-stationary point of \eqref{eq:model-spli} under Assumption~\ref{ass:linearspan3} is easier or more challenging compared to finding an $\epsilon$-stationary point of \eqref{eq:model} under Assumption~\ref{ass:linearspan}, and (ii) whether finding a near $\epsilon$-stationary point of \eqref{eq:model} under Assumption~\ref{ass:linearspan3} is easier or harder  compared to that  
under Assumption~\ref{ass:linearspan}. We provide an answer to the first question, by showing that the same-order lower bound of complexity in Theorem~\ref{thm:lower} holds for finding  an $\epsilon$-stationary point of~\eqref{eq:model-spli} 
under Assumption~\ref{ass:linearspan3}. Moreover, we provide an answer to the second question, by showing that the lower bound of oracle complexity for finding a near $\epsilon$-stationary point of \eqref{eq:model} is $O({\kappa([\bar{\vA}; \vA]) L_f \Delta_{F_0}} \epsilon^{-2})$ under either Assumption~\ref{ass:linearspan} or~\ref{ass:linearspan3}; see Theorem \ref{thm:lower} and Corollary \ref{cor:lower2}.

Before we give our main results in this section, we introduce an instance of~\eqref{eq:model-spli} that is a reformulation of instant~$\mathcal{P}$ in Definition~\ref{def:hardinstance}. 	We show that an $\epsilon$-stationary point of the reformulation~\eqref{eq:model-spli} of instant~$\mathcal{P}$ is a $2\epsilon$-stationary point of the auxiliary problem~\eqref{eq:model3} of instant~$\mathcal{P}$. This is formally stated in the lemma below. 


\begin{lemma}
	\label{cor:kktequiv3}
Let $\epsilon>0$ be given in Definition~\ref{def:hardinstance} for instance~$\mathcal{P}$, and let  $\widehat{\epsilon}\in[0,\epsilon]$. Suppose $(\vx^*,\vy^*)$ is an $\widehat{\epsilon}$-stationary point of the reformulation~\eqref{eq:model-spli} of instant~$\mathcal{P}$.
Then $\vx^*$ is a $2\widehat{\epsilon}$-stationary point of the auxiliary problem~\eqref{eq:model3} of instant~$\mathcal{P}$. 
\end{lemma}
\begin{proof}
By Definition~\ref{def:eps-pt-P}, there exist $\vz_1\in \RR^{\bar{n}}$ and $\vz_2\in\RR^n$ such that the conditions in \eqref{eq:kktviofgesub} hold.
 Hence, for some $\vxi\in \partial \bar{g}(\vy^*)$, we have  $\|\vxi - \vz_1\| \le \widehat{\epsilon}$, and
\begin{equation}
	\label{eq:dderivegeq0_changenew}
\|\nabla f_0(\vx^*) +\bar{\vA}\zz\vxi + \vA\zz\vz_2\|\leq\|\nabla f_0(\vx^*) + \bar{\vA}\zz \vz_1 + \vA\zz \vz_2\|+ \|\bar\vA\|\widehat{\epsilon}\leq \widehat{\epsilon}+ \|\bar\vA\|\widehat{\epsilon}.
\end{equation}
Moreover, \eqref{eq:kktviofgesub} implies that, for any $\vv\in\mathbf{Null}(\vA)$,  we have
\begin{equation}
	\label{eq:dderivegeq0epsilon_changenew}
	\begin{aligned}
		F^{\prime}(\vx^*,\vy^*;\vv, \bar{\vA}\vv)=& \vv\zz\nabla  f_0(\vx^*)+ \bar g^{\prime}(\vy^*;\bar\vA\vv) \\
		\geq&
		 \vv\zz\nabla  f_0(\vx^*) +\vv\zz \bar\vA\zz\vxi =\vv\zz\left(\nabla f_0(\vx^*) +\bar{\vA}\zz\vxi + \vA\zz\vz_2\right)\geq -\widehat{\epsilon}(1+\|\bar\vA\|)\|\vv\|,
	\end{aligned}
\end{equation}
where the first inequality follows from \eqref{eq:result-dir-der} with $g$ replaced by $\bar g$, 
and the second one is by Cauchy-Schwarz inequality and~\eqref{eq:dderivegeq0_changenew}.

Below we prove $\vy^*=\mathbf{0}$. We write it into the block-structured form
$$\vy^*=(\vy_1^{*\top}, \ldots,\vy_{3m_2-1}^{*\top})\zz \text{ with }\vy^*_{i}\in\RR^{{\bar{d}}}, \forall\, i=1,2,\ldots,3m_2-1.$$ 
If  $\vy^*\neq \mathbf{0}$, then $\vy^*_{\bar i}\neq \vzero$ for some $\bar i\in \{1,2,\ldots,3m_2-1\}$. Let  
$\vv^*=(\vv_1^{*\top},\vv_2^{*\top},\dots,\vv_m^{*\top})^\top$ where 
$
\vv_i^*= \vy^*_{\bar i}/(mL_f) \text{ for }i\leq \bar{i}m_1,  \text{ and } \vv_i^*= \vzero \text{ otherwise.}
$
We then have $\vv^*_i=\vv^*_{i+1}$ for any $i\neq \bar{i}m_1$, so $\vA\vv^*=\mathbf{0}$ by Proposition~\ref{prop:consensus}(c). Moreover, let $\vu^*=\bar\vA\vv^*=(\vu_1^{*\top}, \ldots,\vu_{3m_2-1}^{*\top})\zz$ with $\vu^*_{i}\in\RR^{{\bar{d}}}$ for $i=1,2,\ldots,3m_2-1$. We must have 
$\vu_i^*= -\vy^*_{\bar i}$ for $i= \bar{i}$ and $\vu_i^*= \vzero$ for $i\neq\bar{i}$. 
Therefore by \eqref{eq:gbar}, for any $s\in (0,1)$, 
\begin{align}\label{eq:gtvchangey}
	 \bar g\left(\vy^*+s \vu^*\right) 
	= & \frac{\beta}{mL_f}  \sum_{i<\bar{i}}\left\|\vy^*_i\right\|_1+\frac{\beta}{mL_f} \|\vy^*_{\bar{i}}-s\vy^*_{\bar i}\|_1
	+\frac{\beta}{mL_f}\sum_{i>\bar{i}}\left\|\vy^*_i\right\|_1 
	=  \bar g\left(\vy^*\right)-\frac{s\beta}{mL_f}\|\vy^*_{\bar i}\|_1. 
\end{align}

Now by \eqref{eq:fi_infinitynorm}, the choice of $\vv^*_i$, and the mean value theorem, we have for any $i=1,\dots,m$ that
$$
f_i(\vx^*_i+s \vv^*_i)- f_i(\vx^*_i)=s\nabla f_i(\vx^*_i+s' \vv^*_i)^\top\vv^*_i\leq s\|\nabla f_i(\vx^*_i+s' \vv^*_i)\|_\infty\|\vv^*_i\|_1\leq \frac{50s\pi{\epsilon}}{\sqrt{m}}\cdot\frac{\|\vy^*_{\bar i}\|_1}{mL_f},
$$
where $s'\in(0,s)$. The inequality above, together with \eqref{eq:gtvchangey} and the definition of $f_0$ in~\eqref{eq:f0}, implies
\begin{equation*}
	\begin{aligned}
		&\frac{1}{s}\bigg(F(\vx^*+s \vv^*,\vy^*+s \vu^*)-F(\vx^*,\vy^*)\bigg)=\frac{1}{s}\bigg( f_0(\vx^*+s \vv^*)- f_0(\vx^*)+\bar g(\vy^*+s \vu^*)-\bar g(\vy^*)\bigg)\\
		=&\frac{1}{s}\bigg( \sum_{i=1}^m\big(f_i(\vx_i^*+s \vv_i^*)- f_i(\vx_i^*)\big)+\bar g(\vy^*+s \vu^*)-\bar g(\vy^*)\bigg)\\
		\leq& \frac{1}{smL_f}\left(50\pi s{\epsilon}\sqrt{m}\|\vy^*_{\bar i}\|_1-\beta s\|\vy^*_{\bar i}\|_1\right)= \frac{\left(50\pi {\epsilon}\sqrt{m}-\beta \right)\|\vy^*_{\bar i}\|_1}{mL_f}.
	\end{aligned}
\end{equation*}
Letting $s\downarrow 0$ in the inequality above gives $\frac{\left(50\pi {\epsilon}\sqrt{m}-\beta \right)\|\vy^*_{\bar i}\|_1}{mL_f}\geq F^{\prime}(\vx^*,\vy^*;\vv^*,\vu^*)$. Recall $\vA\vv^*=\mathbf{0}$ and $\vu^*=\bar\vA\vv^*$. Hence, we have from \eqref{eq:dderivegeq0epsilon_changenew} and the choice of $\vv^*$ that 
$$
\frac{\left(50\pi {\epsilon}\sqrt{m}-\beta \right)\|\vy^*_{\bar i}\|_1}{mL_f}\geq -\widehat{\epsilon}(1+\|\bar\vA\|)\|\vv^*\|\geq -\widehat{\epsilon}(1+\|\bar\vA\|)\frac{\sqrt{\bar{i}}\|\vy^*_{\bar i}\|}{mL_f}
\geq -\widehat{\epsilon}(1+\|\bar\vA\|)\frac{\sqrt{m}\|\vy^*_{\bar i}\|_1}{mL_f}.
$$
Since $\beta > (50\pi+1+\|\vA\|)\sqrt{m}{\epsilon}$, $\epsilon\ge \widehat{\epsilon}$ and $\|\vA\| \ge \|\bar \vA\|$, the inequalities above can hold only when $\vy^*_{\bar i} = \vzero$. This contradicts to the hypothesis $\vy^*_{\bar i}\neq \vzero$. Hence, $\vy^* = \vzero$, 
which together with \eqref{eq:kktviofgesub} gives $\|\bar\vA\vx^*\|\leq \widehat{\epsilon}$. Furthermore, $\|\vA\vx^*\|\leq \widehat{\epsilon}$ from~\eqref{eq:kktviofgesub}. Thus 
$\|\vH\vx^*\|\leq \|\bar\vA\vx^*\|+\|\vA\vx^*\|\leq 2\widehat{\epsilon}$, which, together with  $\|\nabla f_0(\vx^*) + \bar{\vA}\zz \vz_1 + \vA\zz \vz_2\| \le \widehat{\epsilon}$ from~\eqref{eq:kktviofgesub}, indicates that $\vx^*$ a $2\widehat{\epsilon}$-stationary point of~\eqref{eq:model3}.
\end{proof}

With Lemma~\ref{cor:kktequiv3}, we can establish the lower-bound complexity for solving problem~\eqref{eq:model-spli} in a similar way to show Theorem~\ref{thm:lower}. In particular, Lemma~\ref{cor:kktequiv3} indicates that a solution $(\vx^{(t)},\vy^{(t)})$ cannot be an $\epsilon/2$-stationary point of 
the reformulation~\eqref{eq:model-spli} of instant~$\mathcal{P}$, 
if $\vx^{(t)}$ is not an $\epsilon$-stationary point of~\eqref{eq:model3}. By Lemmas~\ref{lem:nablaf} and~\ref{lem:kktvio}, if there exists $\bar{j}\in\{1,2,\ldots,{\bar{d}}\}$ such that $[\bar{\vx}^{(t)}]_{\bar{j}}=0$, where $\bar{\vx}^{(t)}= \frac{1}{m} \sum_{i=1}^m \vx_i^{(t)}$, then $\vx^{(t)}$ cannot be an $\epsilon$-stationary point of problem~\eqref{eq:model3} and thus cannot be an $\epsilon/2$-stationary point of the reformulation~\eqref{eq:model-spli} of instant~$\mathcal{P}$.
Finally, similar to Proposition~\ref{thm:iterateguess}, we can show that, for any algorithm that is applied to 
the reformulation~\eqref{eq:model-spli} of instance~$\mathcal{P}$, 
if it starts from $(\vx^{(0)},\vy^{(0)})=(\mathbf{0},\mathbf{0})$ and generates a sequence $\{({\vx}^{(t)}, {\vy}^{(t)} )\}_{t\ge 0}$ satisfying Assumption~\ref{ass:linearspan3}, then $\textnormal{supp}({\vx}^{(t)})\subset\{1,\dots,\bar{j}-1\}$ for some $\bar{j}\in\{1,2,\ldots,{\bar{d}}\}$ if $t$ is not large enough. Hence, ${\vx}^{(t)}$ cannot be $\omega$-close to any $\epsilon$-stationary point of instance~$\mathcal{P}$ for any $\omega\in [0, \frac{150\pi \epsilon}{L_f})$, according to the proof of Theorem \ref{thm:lower}. This way, we can obtain a lower bound of oracle complexity to produce an $\epsilon$-stationary point of \eqref{eq:model-spli} and also a lower bound to obtain a near $\epsilon$-stationary point of \eqref{eq:model}. 
Since the aforementioned arguments are similar to those for proving Theorem~\ref{thm:lower}, we simply present the lower complexity bounds in the theorem and the corollary below and put the proofs in Appendix~\ref{sec:appen1}.  


\begin{theorem}
						\label{thm:lbcomposite}
Let $\epsilon>0$ and $L_f > 0$ be given.								Suppose an algorithm  is applied to problem~\eqref{eq:model-spli} that satisfies Assumption~\ref{assume:problemsetup}\textnormal{(a, c)} and $\inf_{\vx, \vy} F(\vx, \vy) > -\infty$.
						Suppose the generated sequence $\{(\vx^{(t)},\vy^{(t)})\}_{t\geq0}$ satisfies Assumption~\ref{ass:linearspan3}. Then there exists an instance of problem~\eqref{eq:model-spli}, i.e., the reformulation~\eqref{eq:model-spli} of instance~$\mathcal{P}$ given in Definition~\ref{def:hardinstance}, such that the algorithm requires at least $\left\lceil \frac{\kappa([\bar{\vA}; \vA]) L_f \Delta_F}{72000\pi^2} \epsilon^{-2}\right\rceil$ oracles to obtain an $\epsilon$-stationary point of that instance, where $\Delta_F =F(\vx^{(0)}, \vy^{(0)})-\inf_{\vx, \vy} F(\vx, \vy)$.
\end{theorem}
			

\begin{corollary}
	\label{cor:lower2}
	Let $\epsilon>0$ and $L_f > 0$ be given. Suppose an algorithm is applied to problem~\eqref{eq:model} that satisfies Assumption~\ref{assume:problemsetup} 
	and generates a sequence $\{(\vx^{(t)},\vy^{(t)})\}_{t\geq0}$ that satisfies Assumption~\ref{ass:linearspan3}. Then for any $\omega\in [0, \frac{150\pi \epsilon}{L_f})$, there exists an instance of problem~\eqref{eq:model}, i.e., instance~$\mathcal{P}$ in Definition~\ref{def:hardinstance}, such that the algorithm requires at least 
	$\left \lceil \frac{\kappa([\bar{\vA}; \vA]) L_f \Delta_{F_0}}{18000\pi^2} \epsilon^{-2} \right\rceil$ 
	oracles to obtain a point $\vx^{(t)}$ that is $\omega$-close to an $\epsilon$-stationary point of instance~$\mathcal{P}$, where $\Delta_{F_0}=F_0(\vx^{(0)})-\inf_\vx F_0(\vx)$. 
\end{corollary}

\section{Tightness of Lower Bounds of Oracle Complexity under Assumption \ref{ass:linearspan3}}\label{sec:ub}
					As we discussed in Section~\ref{sec:relatedwork}, various existing FOMs can be applied to \eqref{eq:model} and~\eqref{eq:model-spli} under a non-smooth non-convex setting. However, the best-known upper bound of oracle complexity  is $O\left(\kappa^2([\vA;\bar{\vA}]) L_f^2\Delta \epsilon^{-2} \right)$ (see \cite{goncalves2017convergence}) with $\Delta= \Delta_{F_0}$ or $\Delta_F$, which does not match our lower bounds. 
					To close the gap between the upper and lower complexity bounds, we present a new \emph{inexact proximal gradient} (IPG) method in this section that falls in the class of algorithms under Assumption~\ref{ass:linearspan3}. The oracle complexity of the IPG matches the lower bounds in Theorem~\ref{thm:lbcomposite} and Corollary~\ref{cor:lower2}, up to logarithmic factors, under a few more assumptions. More precisely, we will need Assumption~\ref{ass:222} and either (but not necessary both) of Assumptions~\ref{ass:dual2} and~\ref{ass:polyhedralg} below\footnote{To claim (near) tightness of our lower bounds, we also need $\vb=\vzero$ and $\bar\vb=\vzero$ under Assumptions~\ref{ass:222} and \ref{ass:polyhedralg}.}. 
					\begin{assumption}
						\label{ass:222}
						$\inf_{\vx,\vy}F(\vx, \vy) > -\infty$; $f_0$ is $l_f$-Lipschitz continuous; $\bar{g}$ is $l_g$-Lipschitz continuous; $\vA$ has  a full-row rank. There exists a feasible point $(\bar\vx, \bar\vy)$ of~\eqref{eq:model-spli} such that $\bar\vy$ is in the relative interior of $\dom(\bar{g})$. 
					\end{assumption}
					\begin{assumption}
						\label{ass:dual2}
						$[\bar{\vA}; \vA]$ has a full-row rank.
					\end{assumption}
					\begin{assumption}
						\label{ass:polyhedralg}
						$\bar{g}(\vy)=\max\{\vu^{\top}\vy : \vC \vu\leq\vd, \vu\in\RR^{\bar{n}}\}$ for some $\vC$ and $\vd$. 
					\end{assumption}
We point out that instance~$\mathcal{P}$ given in Definition~\ref{def:hardinstance} satisfies Assumptions~\ref{assume:problemsetup},~\ref{ass:222},~\ref{ass:dual2} and~\ref{ass:polyhedralg}. By Lemma~\ref{cor:boundf}, $f_0$ is ${50\pi\epsilon\sqrt{m{\bar{d}}}}$-Lipschitz continuous. Since ${\bar{d}}=\Theta(\epsilon^{-2})$ as mentioned in Theorem \ref{thm:lower}, the Lipschitz constant of $f_0$ is independent of  $\epsilon$. Also, the  Lipschitz constant of $\bar{g}$ is $l_g=\sqrt{(3m_2-1)\bar{d}}\frac{\beta}{mL_f}=\Theta(\sqrt{3m_2\bar{d}}\epsilon/\sqrt{m})$, which is not dependent on $\epsilon$ either. In addition, $\dom(\bar g)$ is the whole space. 
These observations, together with Lemmas~\ref{lem:functions} and~\ref{cor:boundf}, imply that Assumption~\ref{ass:222} is satisfied by instance~$\mathcal{P}$. Moreover, it can be easily checked that Assumptions~\ref{ass:dual2} and \ref{ass:polyhedralg} are also satisfied by instance~$\mathcal{P}$. 
%
Hence, the lower bound $O({\kappa([\bar{\vA};\vA]) L_f \Delta} \epsilon^{-2})$ with $\Delta=\Delta_F$ or $\Delta_{F_0}$ remains valid for problems~\eqref{eq:model} and~\eqref{eq:model-spli} even with these additional assumptions, indicating that the oracle complexity of the IPG method (resp. the established lower bound) under these assumptions is optimal (resp. tight). 
		
	\subsection{A New Inexact Proximal Gradient Method}


The IPG method we propose 
generates a sequence $\{(\vx^{(k)},\vy^{(k)})\}$ by 
		\begin{equation}
			\label{eq:subpga-2}
			\begin{aligned}
				(\vx^{(k+1)},\vy^{(k+1)})\approx\argmin_{\vx, \vy} \,\,&\underbrace{\left\langle \nabla f_0(\vx^{(k)}), \vx-\vx^{(k)}\right\rangle+\frac{\tau}{2}\left\|\vx-{\vx}^{(k)}\right\|^2}_{:=\bar{f}(\vx)}+\bar{g}(\vy)
				\\
				\st \,\,&\vy=
				\bar{\vA}\vx +
				\bar{\vb}, \,\, {\vA}\vx +
				{\vb}=\mathbf{0},
			\end{aligned}
		\end{equation}		
for each $k\ge0$, where $\tau \ge L_f$. Solving the minimization problem in~\eqref{eq:subpga-2} exactly can be difficult due to the coexistence of affine constraints and the regularization term. 
	%
To obtain the inexact solution $(\vx^{(k+1)},\vy^{(k+1)})$ efficiently to a desired accuracy, we consider the following Lagrangian function of the problem in~\eqref{eq:subpga-2} 
		\begin{equation}\label{eq:lag-func}
		\mathcal{L}_k(\vx, \vy, \vz)=\bar{f}(\vx)+\bar{g}(\vy)-\vz_1^{\top}\left(\vy-\left(
		\bar{\vA} \vx+
		\bar{\vb} \right)\right)+ \vz_2^{\top}\left(
		{\vA} \vx+
		{\vb} \right),
		\end{equation}
where $\vz=(\vz_1\zz,\vz_2\zz)\zz$, $\vz_1\in \RR^{\bar{n}}$ and $\vz_2\in\RR^n$ are dual variables. Let $\mathcal{D}_k$ be the negative Lagrangian dual function
\footnote{For ease of discussion, we formulate the Lagrangian dual problem into a minimization one by negating the dual function.}, i.e., 
$$
\mathcal{D}_k(\vz):= - \min_{\vx, \vy}\cL_k(\vx, \vy, \vz) = {\frac{1}{2\tau}\|\bar{\vA}\zz\vz_1 +\vA\zz\vz_2 +\nabla f_0(\vx^{(k)})-\tau \vx^{(k)}\|^2} + \bar{g}^{\star} (\vz_1)- \vz_1^{\top}\bar{\vb}- \vz_2^{\top}\vb, \forall\,\vz,
$$
where $ \bar{g}^{\star}$ is the convex conjugate function of $\bar{g}$, i.e., $\bar{g}^{\star}(\vz_1)=\max_{\vy}\{\vy^\top\vz_1-\bar{g}(\vy)\}$. 
We then define
		\begin{equation}
			\label{eq:sublb-2}
			\Omega^{(k+1)}:=\Argmin_{\vz} \mathcal{D}_k(\vz)\quad \text{ and }\quad 	\mathcal{D}_k^*:=\min_{\vz} \mathcal{D}_k(\vz).
		\end{equation}
		
		Note that $\mathbf{prox}_{\eta\bar{g}^{\star}}(\vz_1)= \vz_1-\mathbf{prox}_{\eta\bar{g}(\cdot/\eta)}(\vz_1)= \vz_1-\eta\mathbf{prox}_{\eta^{-1}\bar{g}(\cdot)}(\vz_1/\eta)$~\cite{rockafellar1970convex}. Thus 
		the proximal operator $\mathbf{prox}_{\eta\bar{g}^{\star}}(\vz)$ can be calculated easily if $\mathbf{prox}_{\eta\bar{g}(\cdot/\eta)}(\vz_1)$ can be. 
		Moreover, when Assumption~\ref{ass:dual2} holds, the objective function in~\eqref{eq:sublb-2} is strongly convex composite, so an \emph{accelerated proximal gradient} (APG) method, e.g., the one in~\cite{nesterov2013gradient}, can approach $\Omega^{(k+1)}$ in~\eqref{eq:sublb-2} at a linear rate. When Assumption~\ref{ass:polyhedralg} holds, it is shown in~\cite[Theorem 10]{necoara2019linear} that  the objective function in~\eqref{eq:sublb-2} has a quadratic growth, so a restarted APG method can approach an optimal solution at a linear rate. This motivates us to apply a (restarted) APG method to find a nearly optimal solution of problem in~\eqref{eq:sublb-2}, which is then used to obtain $(\vx^{(k+1)},\vy^{(k+1)})$. 
		
		More specifically, we find a near-optimal point $\vz^{(k+1)}=((\vz_1^{(k+1)})\zz, (\vz_2^{(k+1)})\zz)\zz$ of $\min_{\vz} \mathcal{D}_k(\vz)$ such that  
		\begin{equation}
			\label{eq:boundz-2}
			\vz_1^{(k+1)}\in \dom(\bar{g}^{\star})\quad\text{ and }\quad
			\dist(\vz^{(k+1)}, \Omega^{(k+1)})\leq \delta,
		\end{equation} 
		where $\delta$ is a small number (to be specified). Then, we obtain a primal solution from $\vz^{(k+1)}$ by 
		\begin{eqnarray}
			\label{eq:xupdate-2}
			{\vx}^{(k+1)}&=& \vx^{(k)}-\frac{1}{\tau} \left(\bar{\vA}\zz\vz_1^{(k+1)} +\vA\zz\vz_2^{(k+1)} +\nabla f_0(\vx^{(k)})\right),\\
			\label{eq:yupdate-2}
			{\vy}^{(k+1)}&=& \prox_{\sigma^{-1} \bar{g}}\left(\sigma^{-1} \vz_1^{(k+1)} +\bar{\vA}{\vx}^{(k+1)} +
			\bar{\vb}  \right).
		\end{eqnarray}
		This procedure is presented in Algorithm~\ref{alg:ipga-2}. 
		
		\begin{algorithm}[H]
			\caption{An inexact proximal gradient (IPG) method for problem~\eqref{eq:model-spli}}\label{alg:ipga-2}
			\begin{algorithmic}[1]
				\State \textbf{Input:} a feasible initial point $({\vx}^{(0)}, \vy^{(0)})$, $\sigma>0$, $\epsilon>0$, and $\tau> L_f$. 
				\State Choose $\delta>0$ and let $k\leftarrow 0$.
				\While {an $\epsilon$-stationary  point of problem~\eqref{eq:model-spli} is not obtained}
				
				\State Calculate  a near-optimal point $\vz^{(k+1)}$ of problem~$\min_{\vz} \mathcal{D}_k(\vz)$ that satisfies the conditions in~\eqref{eq:boundz-2}. 
				
				\State Set $\vx^{(k+1)}$ by~\eqref{eq:xupdate-2} and $\vy^{(k+1)}$ by~\eqref{eq:yupdate-2}.
				
				\State Let $k\leftarrow k+1$.
				
				\EndWhile
				\State \textbf{Output:} $(\vx^{(k)},\vy^{(k)})$.
			\end{algorithmic}
		\end{algorithm}

	\subsection{Number of Outer Iterations for Finding an $\epsilon$-stationary Point}
	To characterize the convergence property of Algorithm~\ref{alg:ipga-2}, we need the following two lemmas. 
		\begin{lemma}
			\label{lem:barxyz-2}
Suppose that Assumption~\ref{ass:222} holds.	For any $\sigma>0$ and any $\bar{\vz}^{(k+1)}\in\Omega^{(k+1)}$, let  $(\bar{\vx}^{(k+1)}, \bar{\vy}^{(k+1)})$ be the optimal solution of the strongly convex problem 
			\begin{equation}
				\label{eq:subxy-2}
				\min_{\vx, \vy }\left\{\mathcal{L}_k(\vx, \vy, \bar{\vz}^{(k+1)})+\frac{\sigma}{2}\|{\vy}-
				(\bar{\vA}{\vx} +
				\bar{\vb})\|^2+ \frac{\sigma}{2}\|
				{\vA}{\vx} +
				{\vb}\|^2\right\},
			\end{equation}
where $\cL_k$ and $\Omega^{(k+1)}$ are defined in~\eqref{eq:lag-func} and~\eqref{eq:sublb-2}, respectively.			Then it holds that for any $k\ge0$,
			\begin{equation}
				\label{eq:feasiblexy-2}
				\bar{\vy}^{(k+1)}-
				(\bar{\vA}\bar{\vx}^{(k+1)} +
				\bar{\vb})=\mathbf{0},\quad {\vA}\bar{\vx}^{(k+1)} +
				{\vb}=\mathbf{0},
			\end{equation}
			\begin{equation}
				\label{eq:barx-2}
				\begin{aligned}
					\bar{\vx}^{(k+1)}= \vx^{(k)}-\frac{1}{\tau} \left(\bar{\vA}\zz\bar{\vz}_1^{(k+1)}+\vA\zz\bar{\vz}_2^{(k+1)}+\nabla f_0(\vx^{(k)})\right),
				\end{aligned}
			\end{equation}
			and
			\begin{equation}
				\label{eq:bary-2}
				\,\,\,\bar{\vy}^{(k+1)}= \prox_{\sigma^{-1} \bar{g}}\left(\sigma^{-1} \bar{\vz}_1^{(k+1)} +\bar{\vA}\bar{\vx}^{(k+1)} +
				\bar{\vb}  \right).
			\end{equation}
		\end{lemma}
		\begin{proof}
		Let $(\widehat{\vx}^{(k+1)}, \widehat{\vy}^{(k+1)})$ be an optimal solution of the problem in~\eqref{eq:subpga-2}. 
		Under Assumption~\ref{ass:222}, the strong duality holds~\cite[Section 5.2.3]{boyd2004convex}.  Then the optimal objective value of the minimization problem in~\eqref{eq:subpga-2} is 
		$\max_{\vz} \mathcal{L}_k(\widehat{\vx}^{(k+1)}, \widehat{\vy}^{(k+1)}, {\vz})$. By the definition of $\bar{\vz}^{(k+1)}$ and the strong duality,  it holds that
		\begin{equation*}
			\mathcal{L}_k(\widehat{\vx}^{(k+1)}, \widehat{\vy}^{(k+1)}, \bar{\vz}^{(k+1)})\geq \min_{\vx, \vy} \mathcal{L}_k(\vx, \vy, \bar{\vz}^{(k+1)})=\max_{\vz} \mathcal{L}_k(\widehat{\vx}^{(k+1)}, \widehat{\vy}^{(k+1)}, {\vz})\geq  \mathcal{L}_k(\widehat{\vx}^{(k+1)}, \widehat{\vy}^{(k+1)}, \bar{\vz}^{(k+1)}).
		\end{equation*}
	Hence, the inequalities above must hold with equalities. 
	Thus
		$
		(\widehat{\vx}^{(k+1)}, \widehat{\vy}^{(k+1)})\in\Argmin_{\vx, \vy} \mathcal{L}_k(\vx, \vy, \bar{\vz}^{(k+1)}).
		$
By this fact and also that $(\widehat{\vx}^{(k+1)}, \widehat{\vy}^{(k+1)})$ solves the problem in~\eqref{eq:subpga-2}, we obtain
		\begin{eqnarray*}
			\mathcal{L}_k(\widehat{\vx}^{(k+1)}, \widehat{\vy}^{(k+1)}, \bar{\vz}^{(k+1)}) 
			&=&\min_{\vx, \vy} \mathcal{L}_k(\vx, \vy, \bar{\vz}^{(k+1)})
			\leq  \mathcal{L}_k(\bar{\vx}^{(k+1)}, \bar{\vy}^{(k+1)}, \bar{\vz}^{(k+1)})\\
			&\leq&\mathcal{L}_k(\bar{\vx}^{(k+1)}, \bar{\vy}^{(k+1)}, \bar{\vz}^{(k+1)}) +\frac{\sigma}{2}\|\bar{\vy}^{(k+1)}-(\bar{\vA}{\bar{\vx}^{(k+1)}} +\bar{\vb})\|^2+ \frac{\sigma}{2}\|{\vA}{\bar{\vx}^{(k+1)}} +{\vb}\|^2\\
			&=& \min_{\vx, \vy} \left\{\mathcal{L}_k(\vx, \vy, \bar{\vz}^{(k+1)})+\frac{\sigma}{2}\|{\vy}-
			(\bar{\vA}{\vx} +\bar{\vb})\|^2+ \frac{\sigma}{2}\|
			{\vA}{\vx} +{\vb}\|^2\right\}\\		
			&\leq  &\mathcal{L}_k(\widehat{\vx}^{(k+1)}, \widehat{\vy}^{(k+1)}, \bar{\vz}^{(k+1)}) +\frac{\sigma}{2}\|\widehat{\vy}^{(k+1)}-(\bar{\vA}{\widehat{\vx}^{(k+1)}} +\bar{\vb})\|^2+ \frac{\sigma}{2}\|{\vA}{\widehat{\vx}^{(k+1)}} +{\vb}\|^2\\
			&=&\mathcal{L}_k(\widehat{\vx}^{(k+1)}, \widehat{\vy}^{(k+1)}, \bar{\vz}^{(k+1)}),
		\end{eqnarray*}
	where the second equality is by the definition of $(\bar{\vx}^{(k+1)}, \bar{\vy}^{(k+1)})$. Hence all the inequalities above must hold with   equalities. Thus~\eqref{eq:feasiblexy-2} follows and
	$
(\bar{\vx}^{(k+1)}, \bar{\vy}^{(k+1)})\in\Argmin_{\vx, \vy} \mathcal{L}_k(\vx, \vy, \bar{\vz}^{(k+1)}),
$
the optimality condition of which gives 
\eqref{eq:barx-2} and~\eqref{eq:bary-2}. This completes the proof.
%
		\end{proof}
\begin{remark}
	\label{rem:ipga2}
The proof of Lemma~\ref{lem:barxyz-2} also implies that  $(\bar{\vx}^{(k+1)}, \bar{\vy}^{(k+1)})$ is an optimal solution of the problem in~\eqref{eq:subpga-2}. The lemma below bounds the inexactness of $(\vx^{(k+1)},\vy^{(k+1)})$ generated in Algorithm~\ref{alg:ipga-2}.
\end{remark}
		
		\begin{lemma}
			\label{lem:boundxyz-2}
Suppose that Assumption~\ref{ass:222} holds. Let $\{(\vx^{(k+1)},\vy^{(k+1)})\}_{k\ge 0}$ be generated from Algorithm~\ref{alg:ipga-2}. Denote the vector used to produce $(\vx^{(k+1)},\vy^{(k+1)})$ in~\eqref{eq:boundz-2}-\eqref{eq:yupdate-2} by $\vz^{(k+1)}=(({\vz}_1^{(k+1)})\zz,({\vz}_2^{(k+1)})\zz)\zz$. Let  $(\bar{\vx}^{(k+1)}, \bar{\vy}^{(k+1)})$ be defined as in Lemma~\ref{lem:barxyz-2} with $\bar{\vz}^{(k+1)} = \mathbf{proj}_{\Omega^{(k+1)}}(\vz^{(k+1)})$. Then  the following inequalities hold for all $k\ge0$:
\begin{align}
&\label{eq:lemmaineq1}
					\|\bar{\vx}^{(k+1)}-{\vx}^{(k+1)}\|\leq \frac{1}{\tau}\left\|[\bar{\vA}; \vA]\right\|\delta,\quad  \|\bar{\vy}^{(k+1)}-{\vy}^{(k+1)}\|\leq \frac{1}{\tau}\left\|\bar{\vA}\right\|\left\|[\bar{\vA}; \vA]\right\|\delta +\sigma^{-1} \delta, \\
&\label{eq:lemmaineq2}
			\|{\vy}^{(k+1)}-
			(\bar{\vA}{\vx}^{(k+1)} +	
			\bar{\vb})\|+ \|{\vA}{\vx}^{(k+1)} +
			{\vb}\| 
			\leq B_1 \delta,	\\
&\label{eq:lemmaineq3}
\|\vz_1^{(k+1)}\|\leq l_g, \quad \|\vz_2^{(k+1)}\|\leq B_2 + B_3\delta,	\\	
&\label{eq:lemmaineq4}
	\|\vx^{(k+1)}-\vx^{(k)}\| \leq B_4 +  \frac{1}{\tau} \left\|[\bar{\vA}; \vA]\right\|B_3 \delta,					
\end{align}
where $B_1,B_2,B_3,B_4$ are some constants defined by
\begin{align*}
& \textstyle B_1:=\frac{1}{\tau}\|\bar{\vA}\|\left\|[\bar{\vA}; \vA]\right\| +\sigma^{-1}  + 	\frac{1}{\tau}(\|\bar{\vA}\|+\|{\vA}\|)\left\|[\bar{\vA}; \vA]\right\|, \ B_2:= \|( \vA \vA\zz)^{-1}\vA\| (l_f+ \|\bar{\vA}\|l_g), \\
& \textstyle B_3:= (1+\|( \vA \vA\zz)^{-1}\vA\| \left\|[\bar{\vA}; \vA]\right\|), \ B_4:={\frac{1}{\tau} \left(l_f+\left\|[\bar{\vA}; \vA]\right\| \left( l_g+ B_2\right)\right)}.
\end{align*}
		\end{lemma}
		
		\begin{proof}
			By \eqref{eq:boundz-2}, we have $\|\vz^{(k+1)}-\bar{\vz}^{(k+1)}\|\leq\delta$. 
			Then we obtain from \eqref{eq:xupdate-2} and~\eqref{eq:barx-2} that
			$$
			\|\bar{\vx}^{(k+1)}-{\vx}^{(k+1)}\|\leq\frac{1}{\tau}\left\|[\bar{\vA}; \vA]\right\|\|\vz^{(k+1)}-\bar{\vz}^{(k+1)}\|\leq  \frac{1}{\tau}\left\|[\bar{\vA}; \vA]\right\|\delta,
			$$
			and from \eqref{eq:boundz-2},~\eqref{eq:yupdate-2} and~\eqref{eq:bary-2} that
			$$
			\|\bar{\vy}^{(k+1)}-{\vy}^{(k+1)}\|\leq \|\bar{\vA}\|\|\bar{\vx}^{(k+1)}-{\vx}^{(k+1)}\|+\sigma^{-1}\|\vz^{(k+1)}-\bar{\vz}^{(k+1)}\|\leq \frac{1}{\tau}\left\|\bar{\vA}\right\|\left\|[\bar{\vA}; \vA]\right\|\delta +\sigma^{-1} \delta.
			$$
			Hence, the two inequalities in~\eqref{eq:lemmaineq1} hold. 
				
			In addition, by~\eqref{eq:feasiblexy-2}, we have 
			\begin{equation*}	
				\begin{aligned}
					&\|{\vy}^{(k+1)}-
					(\bar{\vA}{\vx}^{(k+1)} +	
					\bar{\vb})\|+ \|{\vA}{\vx}^{(k+1)} +
					{\vb}\|  \\
					= & \left\|{\vy}^{(k+1)}-
					(\bar{\vA}{\vx}^{(k+1)} +
					\bar{\vb})-\bar{\vy}^{(k+1)}+
					(\bar{\vA}\bar{\vx}^{(k+1)} +
					\bar{\vb})\right\| +  \left\|({\vA}\bar{\vx}^{(k+1)} +
					{\vb})-({\vA}{\vx}^{(k+1)} +{\vb})\right\|\\
					\leq  &
					\left\|{\vy}^{(k+1)}-\bar{\vy}^{(k+1)}\right\| + \left\| \bar{\vA}{\vx}^{(k+1)}-   \bar{\vA}\bar{\vx}^{(k+1)} \right\| +\left\| {\vA}{\vx}^{(k+1)}-   {\vA}\bar{\vx}^{(k+1)} \right\| \\
					\leq  &
					\left\|{\vy}^{(k+1)}-\bar{\vy}^{(k+1)}\right\| + \left(\|\bar{\vA}\|+\|\vA\|\right)\left\| {\vx}^{(k+1)} -  \bar{\vx}^{(k+1)} \right\|
					\leq
					\delta { B_1},
				\end{aligned}
			\end{equation*}
		where the last inequality is from~\eqref{eq:lemmaineq1} and the definition of $B_1$. Hence, the claim in \eqref{eq:lemmaineq2} holds.
		
		Moreover, recall $\vz_1^{(k+1)}\in \dom(\bar{g}^{\star})$ in~\eqref{eq:boundz-2}. Since $\bar{g}$ is $l_g$-Lipschitz continuous by Assumption~\ref{ass:222}, we must have $\|\vz_1^{(k+1)}\|\leq l_g$~\cite{rockafellar1970convex}. 
		Write $\bar{\vz}^{(k+1)}=((\bar{\vz}_1^{(k+1)})\zz,(\bar{\vz}_2^{(k+1)})\zz)\zz$. Then $\bar\vz_1^{(k+1)}\in \dom(\bar{g}^{\star})$. Thus for the same reason, 
		we have $\|\bar{\vz}_1^{(k+1)}\|\leq l_g$. By  Assumption~\ref{ass:222}, $\vA \vA\zz$ is non-singular, so the optimality condition of~\eqref{eq:sublb-2}  implies 
			\begin{equation}\label{eq:formula-z-2-k+1}
			\bar{\vz}_2^{(k+1)}=-\left( \vA \vA\zz\right)^{-1} \left(\vA\bar{\vA}\zz\bar{\vz}_1^{(k+1)}+\vA\nabla f_0(\vx^{(k)})-\tau\vA\vx^{(k)} -\tau \vb\right).
			\end{equation}
		For $k\geq1$, we obtain from the second inequality in~\eqref{eq:feasiblexy-2} and the first inequality in~\eqref{eq:lemmaineq1} that 
		\begin{align*}
		\|\left( \vA \vA\zz\right)^{-1} (\vA\vx^{(k)} +\vb)\|=&\|\left( \vA \vA\zz\right)^{-1}(\vA\vx^{(k)} + \vb-\vA\bar\vx^{(k)} -\vb)\|\\
		\leq& \|( \vA \vA\zz)^{-1}\vA\|\|\vx^{(k)}-\bar\vx^{(k)}\|\leq \frac{1}{\tau}\|( \vA \vA\zz)^{-1}\vA\|\left\|[\bar{\vA}; \vA]\right\|\delta.
		\end{align*}
		Since $\vA\vx^{(0)} +\vb=\mathbf{0}$, the inequality above also holds for $k=0$. 
		Hence, 
			\begin{align*}
		\|\vz_2^{(k+1)}\|\leq&\|\vz_2^{(k+1)}-\bar{\vz}_2^{(k+1)}\|+\|\bar{\vz}_2^{(k+1)}\| \leq\delta+\left\|\left( \vA \vA\zz\right)^{-1} \left(\vA\bar{\vA}\zz\bar{\vz}_1^{(k+1)}+\vA\nabla f_0(\vx^{(k)})-\tau\vA\vx^{(k)} -\tau \vb\right)\right\|\\
		\leq &\delta+\|( \vA \vA\zz)^{-1}\vA\| (l_f+ \|\bar{\vA}\|l_g)+\|( \vA \vA\zz)^{-1}\vA\| \left\|[\bar{\vA}; \vA]\right\|\delta\leq B_2 +B_3\delta.
			\end{align*} 
Thus both inequalities in~\eqref{eq:lemmaineq3} hold.
			
		Finally, it follows from the updating rule in \eqref{eq:xupdate-2} and the inequalities in~\eqref{eq:lemmaineq3} that 
			\begin{align*}&\|\vx^{(k+1)}-\vx^{(k)}\| =\frac{1}{\tau}\left\|  [\bar{\vA}; \vA]\zz \vz^{(k+1)}+\nabla f_0(\vx^{(k)})\right\| \\
				\leq & \frac{1}{\tau} \left(\left\|[\bar{\vA}; \vA]\right\|\left\| \vz_1^{(k+1)}\right\|+\left\|[\bar{\vA}; \vA]\right\|\left\|\vz_2^{(k+1)}\right\|+\left\|\nabla f_0(\vx^{(k)})\right\|\right)\\
					\leq &\frac{1}{\tau}l_f+ \frac{1}{\tau} \left\|[\bar{\vA}; \vA]\right\|l_g+ \frac{1}{\tau} \left\|[\bar{\vA}; \vA]\right\| (B_2 +B_3\delta). 
			\end{align*}
Hence by the definition of $B_4$, \eqref{eq:lemmaineq4} is obtained, and we complete the proof.			
		\end{proof}	

With Lemmas~\ref{lem:barxyz-2} and \ref{lem:boundxyz-2}, we can characterize the number of outer iterations that Algorithm~\ref{alg:ipga-2} needs to find an $\epsilon$-stationary point of problem~\eqref{eq:model-spli}.
		
		\begin{theorem}
			\label{thm:allcomple}
			Suppose that Assumptions~\ref{assume:problemsetup} and~\ref{ass:222} hold. Given $\epsilon >0$, in Algorithm~\ref{alg:ipga-2}, let $\tau=2L_f$ and $\delta= \delta_\epsilon$ with 
			\begin{align}
				\label{eq:delta}
			\delta_\epsilon:= \min\left\{\frac{\epsilon}{B_1\sigma},\ \frac{\epsilon}{B_1},\ \frac{\epsilon^2}{48L_fB_1\left(B_2 +\sigma \|\bar{\vA}\|B_4+l_g\right)},\ \sqrt{\frac{\epsilon^2}{48L_fB_1B_3\left(1+\frac{1}{\tau} \sigma \|\bar{\vA}\|\left\|[\bar{\vA}; \vA]\right\|\right)}} \right\},
			\end{align}
			where $B_1,B_2,B_3$ and $B_4$ are given in Lemma~\ref{lem:boundxyz-2}.  Let 
			\begin{equation}\label{eq:set-K-eps}
				K_\epsilon:= \left \lceil 12L_f \Delta_F\epsilon^{-2} \right\rceil. 
			\end{equation}			
			For each $k\ge0$, let $(\vx^{(k)},\vy^{(k)})$ be generated from  Algorithm~\ref{alg:ipga-2}. Let $k'=\argmin_{k=0,\dots,K_\epsilon-1}\left\|\vx^{(k+1)}-{\vx}^{(k)}\right\|$. Then $(\vx^{(k'+1)},\vy^{(k'+1)})$ is an  $\epsilon$-stationary point  of problem~\eqref{eq:model-spli}.
		\end{theorem}
		
		\begin{proof}
			Since  $\delta\leq \frac{\epsilon}{B_1}$ from~\eqref{eq:delta}, we obtain from inequality~\eqref{eq:lemmaineq2} that, for any $k\ge 0$,  
			\begin{equation}
				\label{eq:feasibility}
				\|{\vy}^{(k+1)}-
				(\bar{\vA}{\vx}^{(k+1)} +	
				\bar{\vb})\|+ \|{\vA}{\vx}^{(k+1)} +
				{\vb}\| 
				\leq\delta B_1\leq \epsilon.
			\end{equation}
	In addition, by 
	\eqref{eq:yupdate-2}, 
	it holds for any $k\ge 0$ that
			\begin{equation}
				\label{eq:kktvy-2}
				\mathbf{0}\in \partial \bar{g}(\vy^{(k+1)}) - \vz_1^{(k+1)} +\sigma \left({\vy}^{(k+1)}-
				(\bar{\vA}{{\vx}^{(k+1)}} +
				\bar{\vb})\right),
			\end{equation}	
		which implies that there exists $\vxi^{(k+1)}\in\partial \bar{g}(\vy^{(k+1)})$ such that 
			\begin{equation}
				\label{eq:kktvy-2vio}
			\| \vxi^{(k+1)} - \vz_1^{(k+1)}\| = \sigma \|{\vy}^{(k+1)}-
			(\bar{\vA}{{\vx}^{(k+1)}} +
			\bar{\vb})\| \overset{\eqref{eq:lemmaineq2}} \leq  \delta B_1\sigma \overset{\eqref{eq:delta}}\leq \epsilon.
			\end{equation}
%
			Moreover, it holds by~\eqref{eq:xupdate-2} that for any $k\ge 0$, 
			\begin{equation}
				\label{eq:kktvxvio-2}
				\mathbf{0}= \nabla f_0(\vx^{(k)}) + [
				\bar{\vA} ;
				\mathbf{A}]\zz \vz^{(k+1)} +\tau (\vx^{(k+1)}-\vx^{(k)}).
			\end{equation}
		Hence, by the $L_f$-Lipschitz continuity of $\nabla f_0$ and the fact that $\tau=2L_f$, we have for any $k\ge0$,
		\begin{align}\label{eq:xerror-2}
		\left\|\nabla f_0(\vx^{(k+1)}) + [
		\bar{\vA} ;
		\mathbf{A}]\zz \vz^{(k+1)}\right\|
		\leq &\left\|\nabla f_0(\vx^{(k)}) + [
		\bar{\vA} ;
		\mathbf{A}]\zz \vz^{(k+1)}\right\|+L_f\|\vx^{(k+1)}-\vx^{(k)} \| \nonumber\\
		=& (\tau+L_f)\|\vx^{(k+1)}-\vx^{(k)} \|
		=3L_f\|\vx^{(k+1)}-\vx^{(k)} \|.
		\end{align}
			Below we bound the average of the square of the right-hand side of~\eqref{eq:xerror-2} over $K$ terms.
			First, by \eqref{eq:lemmaineq2} and the feasibility of $({\vx}^{(0)}, {\vy}^{(0)})$, i.e, ${\vA}{\vx}^{(0)} +{\vb} = \vzero$ and ${\vy}^{(0)} = \bar{\vA}{\vx}^{(0)} +\bar{\vb}$, we have for any $k\ge0$,
			\begin{align}
				\label{eq:lemmaineq21}
					\left\|{\vy}^{(k+1)}-{\vy}^{(k)}-
				\bar{\vA}({\vx}^{(k+1)}-{\vx}^{(k)})\right\|
				=\left\|{\vy}^{(k+1)}-
				(\bar{\vA}{\vx}^{(k+1)} +
				\bar{\vb})-{\vy}^{(k)}+
				(\bar{\vA}{\vx}^{(k)} +
				\bar{\vb})\right\|\leq &2\delta B_1,\\	\label{eq:lemmaineq22}
				\left\|\vA({\vx}^{(k+1)}-{\vx}^{(k)})\right\| = \left\|({\vA}{\vx}^{(k)} +
				{\vb})-({\vA}{\vx}^{(k+1)} +{\vb})\right\|\leq &2\delta B_1	.		
			\end{align}
	Second, it follows from the $L_f$-Lipschitz continuity of $\nabla f_0$ that 
			\begin{align}
			\nonumber
			&f_0(\vx^{(k+1)})-f_0(\vx^{(k)})
			\leq \left\langle \nabla f_0(\vx^{(k)}), \vx^{(k+1)}-\vx^{(k)}\right\rangle+\frac{L_f}{2}\left\|\vx^{(k+1)}-{\vx}^{(k)}\right\|^2
			\\	\nonumber
			\overset{\eqref{eq:kktvxvio-2}}= &\left\langle -{\tau}\left(\vx^{(k+1)}-{\vx}^{(k)}\right) -\bar{\vA}\zz \vz_1^{(k+1)}-\vA\zz \vz_2^{(k+1)}, \vx^{(k+1)}-\vx^{(k)}\right\rangle +\frac{L_f}{2}\left\|\vx^{(k+1)}-{\vx}^{(k)}\right\|^2
			\\	\nonumber
			= &\frac{L_f-2\tau}{2}\left\|\vx^{(k+1)}-{\vx}^{(k)}\right\|^2-\left\langle \bar{\vA}\zz \vxi^{(k+1)}, \vx^{(k+1)}-\vx^{(k)}\right\rangle-\left\langle \vz_2^{(k+1)}, \vA(\vx^{(k+1)}-\vx^{(k)})		
			\right\rangle\\	\nonumber
			&~-\left\langle \bar{\vA}\zz (\vz_1^{(k+1)}-\vxi^{(k+1)}), \vx^{(k+1)}-\vx^{(k)}\right\rangle\\	\nonumber
			\leq &-\frac{3L_f}{2}\left\|\vx^{(k+1)}-{\vx}^{(k)}\right\|^2-\left\langle \bar{\vA}\zz \vxi^{(k+1)}, \vx^{(k+1)}-\vx^{(k)}\right\rangle+2\delta B_1\| \vz_2^{(k+1)}\|+2\delta B_1\sigma \|\bar{\vA}\|\|\vx^{(k+1)}-{\vx}^{(k)}\|\\	\nonumber
			\leq &-\frac{3L_f}{2}\left\|\vx^{(k+1)}-{\vx}^{(k)}\right\|^2-\left\langle \bar{\vA}\zz \vxi^{(k+1)}, \vx^{(k+1)}-\vx^{(k)}\right\rangle+2\delta B_1\left(B_2 +B_3\delta\right)\\\label{eq:proofstationarity1}
			&~+2\delta B_1\sigma \|\bar{\vA}\|\left(B_4 +  \frac{1}{\tau} \left\|[\bar{\vA}; \vA]\right\|B_3\delta\right),
		\end{align}
	where 
	the second inequality holds from $\tau=2L_f$,~\eqref{eq:kktvy-2vio} and~\eqref{eq:lemmaineq22}, and the last inequality follows from~\eqref{eq:lemmaineq3} and~\eqref{eq:lemmaineq4}.
Third, by the convexity of $\bar g$ and the fact that $\vxi^{(k+1)}\in\partial \bar{g}(\vy^{(k+1)})$, we have 
	\begin{align}
		\nonumber
		& \bar g(\vy^{(k+1)})- \bar g(\vy^{(k)})
		\leq \left\langle \vxi^{(k+1)}, \vy^{(k+1)}-\vy^{(k)} \right\rangle 
		\\	\nonumber
		= &	 \left\langle \vxi^{(k+1)}, 	\bar{\vA}({\vx}^{(k+1)}-{\vx}^{(k)})  \right\rangle + \left\langle \vxi^{(k+1)}, {\vy}^{(k+1)}-{\vy}^{(k)}-
		\bar{\vA}({\vx}^{(k+1)}-{\vx}^{(k)}) \right\rangle 
		\\	\nonumber
		\leq & \left\langle \vxi^{(k+1)}, 	\bar{\vA}({\vx}^{(k+1)}-{\vx}^{(k)})  \right\rangle +2\delta B_1\|\vxi^{(k+1)}\|\\\label{eq:proofstationarity2}
		\leq &\left\langle \vxi^{(k+1)}, 	\bar{\vA}({\vx}^{(k+1)}-{\vx}^{(k)})  \right\rangle +2\delta B_1l_g,
	\end{align}
where the second inequality is from~\eqref{eq:lemmaineq21}, and the last one holds by $\|\vxi^{(k+1)}\|\le l_g$ from the $l_g$-Lipschitz continuity of $\bar g$. 
Adding~\eqref{eq:proofstationarity1} and~\eqref{eq:proofstationarity2} and combining terms give
	\begin{align*}
	\nonumber
	&f_0(\vx^{(k+1)}) +\bar g(\vy^{(k+1)})-f_0(\vx^{(k)})-\bar g(\vy^{(k)})\\
	\leq &-\frac{3L_f}{2}\left\|\vx^{(k+1)}-{\vx}^{(k)}\right\|^2+2\delta B_1\left(B_2 +\sigma \|\bar{\vA}\|B_4+l_g\right)
	+2\delta^2 B_1B_3\left(1+\frac{1}{\tau} \sigma \|\bar{\vA}\|\left\|[\bar{\vA}; \vA]\right\|\right).
	\end{align*}

			Multiplying $3L_f$ to the inequality above and summing it over $k=0,1,\ldots,K-1$, we  obtain 
			\begin{equation}\label{eq:stationarybound}
			\begin{aligned}
				\frac{(3L_f)^2}{K}&\sum_{k=0}^{K-1} \left\|\vx^{(k+1)}-{\vx}^{(k)}\right\|^2 \leq   \frac{6L_f\left(F(\vx^{(0)}, \vy^{(0)})-\inf_{\vx,\vy}F(\vx, \vy)\right)}{K }\\
				& + 12L_f\delta B_1\left(B_2 +\sigma \|\bar{\vA}\|B_4+l_g\right)	
				+12L_f\delta^2 B_1B_3\left(1+\frac{1}{\tau} \sigma \|\bar{\vA}\|\left\|[\bar{\vA}; \vA]\right\|\right), \forall\, K\ge1.
			\end{aligned}	
			\end{equation}
			Let $K=K_\epsilon$ in~\eqref{eq:stationarybound}. Since  $K_\epsilon\geq 12L_f\left(F(\vx^{(0)}, \vy^{(0)})-\inf_{\vx,\vy}F(\vx, \vy)\right)\epsilon^{-2}$, it holds that 
			\begin{equation}
				\label{eq:Kep}
				\frac{6L_f\left(F(\vx^{(0)}, \vy^{(0)})-\inf_{\vx,\vy}F(\vx, \vy)\right)}{K_\epsilon}\leq \frac{1}{2}\epsilon^2.
			\end{equation}
		Moreover, the choice of $\delta=\delta_\epsilon$ in~\eqref{eq:delta} ensures			
				\begin{equation}
				\label{eq:deltaep}
	12L_f\delta B_1\left(B_2 +\sigma \|\bar{\vA}\|B_4+l_g\right)\leq \frac{1}{4}\epsilon^2\quad\text{and}\quad 12L_f\delta^2 B_1B_3\left(1+\frac{1}{\tau} \sigma \|\bar{\vA}\|\left\|[\bar{\vA}; \vA]\right\|\right)\leq \frac{1}{4}\epsilon^2.
		\end{equation}			
Applying the bounds in~\eqref{eq:Kep} and~\eqref{eq:deltaep} to~\eqref{eq:stationarybound} with $K=K_\epsilon$ gives 
\begin{align}
\label{eq:Kepsilon1}
\frac{(3L_f)^2}{K_\epsilon}\sum_{k=0}^{K_\epsilon-1} \left\|\vx^{(k+1)}-{\vx}^{(k)}\right\|^2 \leq \epsilon^2.
\end{align}
Hence, it follows that ${3L_f}\min_{k=0,\dots,K_\epsilon-1}\left\|\vx^{(k+1)}-{\vx}^{(k)}\right\| \leq \epsilon$. Then it results from the definition of $k'$ and \eqref{eq:xerror-2} that 
\begin{equation}
\label{eq:kktviosq-1}
\begin{aligned}
\left\|\nabla f_0(\vx^{(k'+1)}) + [
\bar{\vA} ;
\mathbf{A}]\zz \vz^{(k'+1)}\right\| 
\leq  {3L_f}\left\|\vx^{(k'+1)}-{\vx}^{(k')}\right\| = {3L_f}\min_{k=0,\dots,K_\epsilon-1}\left\|\vx^{(k+1)}-{\vx}^{(k)}\right\| \leq   
\epsilon,
\end{aligned}
\end{equation}
which together with~\eqref{eq:feasibility} and~\eqref{eq:kktvy-2vio} 
shows that $(\vx^{(k'+1)},\vy^{(k'+1)})$ is an  $\epsilon$-stationary point  of problem~\eqref{eq:model-spli} by Definition~\ref{def:eps-pt-P}. This completes the proof.
		\end{proof}
	
Theorem~\ref{thm:allcomple} above only shows that  Algorithm~\ref{alg:ipga-2} can produce an $\epsilon$-stationary point  of problem~\eqref{eq:model-spli}. In fact, with a $\delta$ that is slightly smaller  than $\delta_\epsilon$ in Algorithm~\ref{alg:ipga-2}, we can also show that Algorithm~\ref{alg:ipga-2} can produce a near $\epsilon$-stationary point  of~\eqref{eq:model} with an outer-iteration number similar to that in Theorem~\ref{thm:allcomple}. This result is presented in the following theorem. 

\begin{theorem}
\label{thm:allcompleforp}
Suppose that Assumptions~\ref{assume:problemsetup} and~\ref{ass:222} hold.  Given $\epsilon >0$, in Algorithm~\ref{alg:ipga-2}, let $\tau=2L_f$ and $\delta = \bar{\delta}_\epsilon:=\min\left\{ \frac{\epsilon}{6\left\|[\bar{\vA}; \vA]\right\|} , \frac{\Delta_{F_0}}{B_1l_g} ,\delta_\epsilon \right\}$, where $B_1$ is given in Lemma~\ref{lem:boundxyz-2} and $\delta_\epsilon$ is given in~\eqref{eq:delta}. Let $$\bar{K}_\epsilon:=\lceil 192 L_f\Delta_{F_0}\epsilon^{-2} \rceil.$$
For each $k\ge0$, let $(\vx^{(k)},\vy^{(k)})$ and $\vz^{(k)}$ be generated from  Algorithm~\ref{alg:ipga-2},  $(\bar{\vx}^{(k+1)}, \bar{\vy}^{(k+1)})$ be defined as in Lemma~\ref{lem:barxyz-2} with $\bar{\vz}^{(k+1)} = \mathbf{proj}_{\Omega^{(k+1)}}(\vz^{(k+1)})$. Let $k'=\underset{{k=0,\dots,\bar{K}_\epsilon-1}}\argmin\left\|\vx^{(k+1)}-{\vx}^{(k)}\right\|$. Then $(\vx^{(k'+1)},\vy^{(k'+1)})$ is an  $\epsilon$-stationary point  of problem~\eqref{eq:model-spli},  $\bar{\vx}^{(k'+1)}$ is an $\epsilon$-stationary point  of problem~\eqref{eq:model}, and $\vx^{(k'+1)}$ is a near $\epsilon$-stationary point  of problem~\eqref{eq:model}. More specifically, $\|\vx^{(k'+1)}-\bar\vx^{(k'+1)}\|\leq \frac{\epsilon}{12L_f}\in [0, \frac{150\pi \epsilon}{L_f})$.
\end{theorem}
\begin{proof}
Notice $\mathbf{0}= \nabla f_0(\vx^{(k)}) + [
				\bar{\vA} ;
				\mathbf{A}]\zz \bar\vz^{(k+1)} +\tau (\bar\vx^{(k+1)}-\vx^{(k)})$ from~\eqref{eq:barx-2}. Hence through the same arguments to obtain~\eqref{eq:xerror-2}, we arrive at  
	\begin{equation}
		\label{eq:fobar}
	\left\|\nabla f_0(\bar{\vx}^{(k+1)}) + [
	\bar{\vA} ;
	\mathbf{A}]\zz \bar{\vz}^{(k+1)}\right\| \leq {3L_f}\left\|\bar{\vx}^{(k+1)}-{\vx}^{(k)}\right\|, \forall\, k\ge0.
	\end{equation}
By the triangle inequality, it holds $\left\|\bar{\vx}^{(k+1)}-{\vx}^{(k)}\right\| \le \left\|\bar{\vx}^{(k+1)}-{\vx}^{(k+1)}\right\|+ \left\|{\vx}^{(k+1)}-{\vx}^{(k)}\right\|$. Thus from~\eqref{eq:lemmaineq1}, we have	
$\left\|\bar{\vx}^{(k+1)}-{\vx}^{(k)}\right\| \le \left\|{\vx}^{(k+1)}-{\vx}^{(k)}\right\| + \frac{1}{\tau}\left\|[\bar{\vA}; \vA]\right\|{\delta},$ which together with~\eqref{eq:fobar} and $\tau=2L_f$ gives
\begin{equation}
		\label{eq:fobar-2-1}
	\textstyle \left\|\nabla f_0(\bar{\vx}^{(k+1)}) + [
	\bar{\vA} ;
	\mathbf{A}]\zz \bar{\vz}^{(k+1)}\right\| \leq {3L_f}\left\|{\vx}^{(k+1)}-{\vx}^{(k)}\right\| + \frac{3}{2} \left\|[\bar{\vA}; \vA]\right\|{\delta}, \forall\,\ k\ge0.
	\end{equation}
By using the same method to obtain argument~\eqref{eq:stationarybound} and the definition of $\bar{\delta}_{\epsilon}$, we have
\begin{equation}\label{eq:stationarybound-15}
	\begin{aligned}
		\frac{(3L_f)^2}{K}&\sum_{k=0}^{K-1} \left\|\vx^{(k+1)}-{\vx}^{(k)}\right\|^2 \leq   \frac{6L_f\left(F(\vx^{(0)}, \vy^{(0)})-[f_0(\vx^{(K+1)}) +\bar g(\vy^{(K+1)})]\right)}{K }+\frac{\epsilon^2}{2}, \forall\, K\ge1.
	\end{aligned}	
\end{equation}

Recall that $(\vx^{(0)}, \vy^{(0)})$ is a feasible point of problem \eqref{eq:model-spli}. We have
\begin{align*}
	&F(\vx^{(0)}, \vy^{(0)})-\left[f_0(\vx^{(K+1)}) +\bar g(\vy^{(K+1)})\right]\\=& \left(f_0(\vx^{(0)})+\bar{g}(\bar{\vA}\vx^{(0)}+\bar\vb)-\left[f_0(\vx^{(K+1)}) +\bar g(\bar{\vA}{\vx}^{(K+1)} +	
	\bar{\vb})\right]\right) + \left[\bar g(\bar{\vA}{\vx}^{(K+1)} +	
	\bar{\vb})-\bar g(\vy^{(K+1)})\right] \\
	\leq&  \Delta_{F_0} + l_g\|\bar{\vA}{\vx}^{(K+1)} +	
	\bar{\vb}-\vy^{(K+1)}\|\\
	\leq& \Delta_{F_0} +\Delta_{F_0} = 2 \Delta_{F_0}, \quad\forall\, K\ge1,
\end{align*}
where the first inequality is because  $\bar{g}$ is $l_g$-Lipschitz continuous by Assumption~\ref{ass:222} and the second inequality is because 
 $ \|{\vy}^{(K+1)}-
(\bar{\vA}{\vx}^{(K+1)} +	
\bar{\vb})\|
\leq\bar{\delta}_{\epsilon} B_1\leq \frac{\Delta_{F_0}}{l_g}$ according to \eqref{eq:feasibility} and the definition of $\bar{\delta}_{\epsilon}$.
Now substituting the above inequality into \eqref{eq:stationarybound-15} with $K=\bar K_\epsilon=\lceil 192 L_f\Delta_{F_0}\epsilon^{-2} \rceil$, 
we obtain 	
$$\frac{(3L_f)^2}{K}\sum_{k=0}^{K-1} \left\|\vx^{(k+1)}-{\vx}^{(k)}\right\|^2 \le \frac{12 L_f \Delta_{F_0}}{192 L_f\Delta_{F_0}\epsilon^{-2}} +\frac{\epsilon^2}{2} = \frac{\epsilon^2}{16}+\frac{\epsilon^2}{2}=\frac{9\epsilon^2}{16}.$$ 
Since $k'=\argmin_{k=0,\dots,\bar{K}_\epsilon-1}\left\|\vx^{(k+1)}-{\vx}^{(k)}\right\|$, we have
${3L_f}\left\|\vx^{(k'+1)}-{\vx}^{(k')}\right\| \leq \frac{3\epsilon}{4}$. Then, it follows from the same arguments in the proof of Theorem \ref{thm:allcomple} that $(\vx^{(k'+1)},\vy^{(k'+1)})$ is an  $\epsilon$-stationary point  of problem~\eqref{eq:model-spli}. 

Also,
we obtain from~\eqref{eq:fobar-2-1} and the choice of $\delta$ that 
\begin{equation}\label{eq:fobar-2-2}
\left\|\nabla f_0(\bar{\vx}^{(k'+1)}) + [
	\bar{\vA} ;
	\mathbf{A}]\zz \bar{\vz}^{(k'+1)}\right\| \leq \epsilon.
\end{equation}		
On the other hand, we have from \eqref{eq:feasiblexy-2} and~\eqref{eq:bary-2} that	
\begin{equation}
		\label{eq:stabarx}		
			\bar{\vy}^{(k'+1)}= 
			(\bar{\vA}\bar{\vx}^{(k'+1)} +	
			\bar{\vb}),\quad {\vA}\bar{\vx}^{(k'+1)} +
			{\vb} =\vzero ,\quad
			\bar{\vz}_1^{(k'+1)}\in \partial \bar{g}(\bar{\vy}^{(k'+1)}).	
	\end{equation}
Recall $g(\vx) = \bar g(\bar\vA\vx + \bar\vb)$. 	By~\cite{clarke1990optimization}, 
it holds that
	$
	\partial g(\bar{\vx}^{(k'+1)})=\overline{\mathbf{co}}\left(\left\{\bar{\vA}\zz \vxi: \vxi \in \partial\bar{g}\big(\bar{\vA}\bar{\vx}^{(k'+1)} +	
			\bar{\vb}\big)\right\} \right),
	$
and thus from~\eqref{eq:stabarx}, it follows	 $\bar{\vA}\zz\bar{\vz}_1^{(k'+1)} \in \partial g(\bar{\vx}^{(k'+1)}).$ Then by~\eqref{eq:fobar-2-2}, we have 
	$$
	\textstyle \dist\left(\vzero, \nabla {f}_0\big(\bar{\vx}^{(k'+1)}\big)+\partial g\big(\bar{\vx}^{(k'+1)}\big)+{\vA}\zz \bar{\vz}^{(k'+1)}_2\right) \le \left\|\nabla {f}_0\big(\bar{\vx}^{(k'+1)}\big)+\bar{\vA}\zz\bar{\vz}^{(k'+1)}_1+{\vA}\zz \bar{\vz}^{(k'+1)}_2\right\| \le \epsilon,$$
which, together with 	${\vA}\bar{\vx}^{(k'+1)} + {\vb} =\vzero$ from~\eqref{eq:stabarx}, indicates that
	$\bar{\vx}^{(k'+1)}$ is an $\epsilon$-stationary point of~\eqref{eq:model}. 
	
		By \eqref{eq:lemmaineq1}, 	$\|\bar{\vx}^{(k'+1)}-{\vx}^{(k'+1)}\|\leq \frac{1}{\tau}\left\|[\bar{\vA}; \vA]\right\|\delta_\epsilon\leq \frac{\epsilon}{12L_f}$, where the second inequality is from the definition of $\bar\delta_\epsilon$ which ensures $\bar\delta_\epsilon \leq  \frac{\epsilon}{6\|[\bar{\vA}; \vA]\|}$.
%
\end{proof}

\subsection{Number of Inner Iterations for Finding $\vz^{(k+1)}$ Satisfying \eqref{eq:boundz-2}}

		To obtain the total number of inner iterations of Algorithm~\ref{alg:ipga-2} to find an $\epsilon$-stationary point of~\eqref{eq:model-spli}, we still need to evaluate the  number of iterations for computing $\vz^{(k+1)}$ that satisfies the criterion in \eqref{eq:boundz-2} for each $k\ge0$. 
The problem in~\eqref{eq:sublb-2} has the convex composite structure. Additionally, the gradient of the smooth function $\mathcal{D}_k(\vz)-\bar{g}^{\star} (\vz_1)$ in the objective function $\mathcal{D}_k$ in~\eqref{eq:sublb-2} is $L_\mathcal{D}$-Lipschitz continuous with 
$$
L_\mathcal{D}:=\lambda_{\max}([\bar{\vA}; \vA][\bar{\vA}; \vA]\zz)/\tau.
$$
Hence, we apply the APG method in~\cite{beck2009fast} with a standard restarting technique
to~\eqref{eq:sublb-2} to find $\vz^{(k+1)}$.  The restarted APG algorithm instantiated on~\eqref{eq:sublb-2} is presented in Algorithm~\ref{alg:restartedAPG}. The algorithm has double loops and, in addition to the main iterate $\vz^{(j,k)}$, it also generates an auxiliary iterate $\widehat\vz^{(j,k)}$. In the algorithm, we use $\widehat{\mathcal{G}}_1^j$ and $\widehat{\mathcal{G}}_2^j$ to represent the gradients of  $\mathcal{D}_k(\vz)-\bar{g}^{\star} (\vz_1)$ with respect to $\vz_1$ and $\vz_2$, respectively, at $\vz=\widehat\vz^{(j,k)}$. The algorithm is restarted after every $j_k$ APG steps. We call each APG step, i.e., Line~5-9 in Algorithm~\ref{alg:restartedAPG}, as one inner iteration of Algorithm~\ref{alg:ipga-2}. 

		\begin{algorithm}[H]
	\caption{Restarted accelerated proximal gradient method for~\eqref{eq:sublb-2}}\label{alg:restartedAPG}
	\begin{algorithmic}[1]
		\State \textbf{Input:} an initial point $\vz^{\text{ini}}=((\vy^{(0)})\zz, (\vA\vx^{(0)})\zz)\zz$ where $(\vx^{(0)},\vy^{(0)})$ is the same as in Algorithm~\ref{alg:ipga-2}.
		\State $\vz^{(0,k)}\leftarrow \vz^{\text{ini}}$, $\widehat\vz^{(0,k)}\leftarrow \vz^{(0,k)}$, $\alpha_0\leftarrow 1$.
		\For  {$i=0,\dots,i_k-1$}
		\For {$j=0,\dots,j_k-1$}
		\State $\widehat{\mathcal{G}}_1^j\leftarrow\frac{1}{\tau}\bar{\vA}\left(\bar{\vA}\zz\widehat\vz_1^{(j,k)} +\vA\zz\widehat\vz_2^{(j,k)} +\nabla f_0(\vx^{(k)})-\tau \vx^{(k)}\right)-\bar\vb.$
		
		\State
		$\widehat{\mathcal{G}}_2^j\leftarrow\frac{1}{\tau}\vA\left(\bar{\vA}\zz\widehat\vz_1^{(j,k)} +\vA\zz\widehat\vz_2^{(j,k)} +\nabla f_0(\vx^{(k)})-\tau \vx^{(k)}\right)-\vb$.
		
		\State $\vz_1^{(j+1,k)}\leftarrow \mathbf{prox}_{L_\mathcal{D}^{-1}\bar{g}^*}\left(\widehat\vz_1^{(j,k)}-\frac{1}{L_\mathcal{D}}\widehat{\mathcal{G}}_1^j\right)$,~
		$\vz_2^{(j+1,k)}\leftarrow \widehat\vz_2^{(j,k)}-\frac{1}{L_\mathcal{D}}\widehat{\mathcal{G}}_2^j$.
		
		\State $\alpha_{j+1}\leftarrow \frac{1+\sqrt{1+4\alpha_j^2}}{2}$.
			
		\State $\widehat\vz^{(j+1,k)}\leftarrow \vz^{(j+1,k)}+\left(\frac{\alpha_{j}-1}{\alpha_{j+1}}\right)\left(\vz^{(j+1,k)}-\vz^{(j,k)}\right)$.
		
		\EndFor
\State $\vz^{(0,k)}\leftarrow \vz^{(j_k,k)}$, $\widehat\vz^{(0,k)}\leftarrow \vz^{(0,k)}$, $\alpha_0\leftarrow 1$.
		\EndFor
		\State \textbf{Output:} $\vz^{(k+1)}=\vz^{(0,k)}$.
	\end{algorithmic}
\end{algorithm}

The number of APG steps performed in Algorithm~\ref{alg:restartedAPG} to find a point satisfying \eqref{eq:boundz-2} is known when Assumption~\ref{ass:222} and either one of Assumptions~\ref{ass:dual2} and~\ref{ass:polyhedralg} hold. We first present the result when Assumptions	\ref{ass:222} and~\ref{ass:dual2} hold.	
		\begin{theorem}\label{thm:inner-iter-dual2}
			Suppose Assumptions~\ref{assume:problemsetup},~\ref{ass:222} and~\ref{ass:dual2} hold.
			Given any $\delta>0$, if $j_k=\lceil 2\sqrt{2}\kappa ([\bar{\vA}; \vA])\rceil$ and $i_k=\left\lceil \log_2(\frac{2(\mathcal{D}_k(\vz^{\text{ini}})-\mathcal{D}_k^*)}{\mu_\mathcal{D}\delta^2})\right\rceil$ in Algorithm~\ref{alg:restartedAPG}, the output $\vz^{(k+1)}$ satisfies the criteria in \eqref{eq:boundz-2}. Here, $\mu_\mathcal{D}=\lambda_{\min}([\bar{\vA}; \vA][\bar{\vA}; \vA]\zz)/\tau$.
		\end{theorem}
	
	
	\begin{proof}
		With Assumption~\ref{ass:dual2}, the problem in~\eqref{eq:sublb-2} is $\mu_\mathcal{D}$-strongly convex. Consider the end of the first inner loop (i.e., $i=0$) of Algorithm~\ref{alg:restartedAPG}.  According to~\cite[Theorem 4.4.]{beck2009fast}, after $j_k$ APG steps, the APG method generates an output $\vz^{(j_k,k)}$ satisfying 
		$$
		\mathcal{D}_k(\vz^{(j_k,k)})-\mathcal{D}_k^*\leq\frac{2L_\mathcal{D}\dist^2(\vz^{\text{ini}}, \Omega^{(k+1)})}{j_k^2}\leq\frac{4L_\mathcal{D} [\mathcal{D}_k(\vz^{\text{ini}})-\mathcal{D}_k^*]}{\mu_\mathcal{D}j_k^2}=\frac{4 \kappa^2([\bar{\vA}; \vA]) [\mathcal{D}_k(\vz^{\text{ini}})-\mathcal{D}_k^*]}{j_k^2},
		$$
		where the second inequality is by the $\mu_\mathcal{D}$-strong convexity of $\mathcal{D}_k$. Since $j_k=\lceil 2\sqrt{2}\kappa ([\bar{\vA}; \vA])\rceil$, we have $\mathcal{D}_k(\vz^{(j_k,k)})-\mathcal{D}_k^* \le \frac{1}{2}\left(\mathcal{D}_k(\vz^{\text{ini}})-\mathcal{D}_k^*\right)$, meaning that we can reduce the objective gap by a half with each inner loop of  Algorithm~\ref{alg:restartedAPG}. Since the next inner loop is started at $\vz^{(j_k,k)}$ according to Line 11 of Algorithm~\ref{alg:restartedAPG}, we repeat this argument $i_k=\left\lceil \log_2(\frac{2(\mathcal{D}_k(\vz^{\text{ini}})-\mathcal{D}_k^*)}{\mu_\mathcal{D}\delta^2})\right\rceil$ times and show that the final output $\vz^{(k+1)}$ satisfies $\mathcal{D}_k(\vz^{(k+1)})-\mathcal{D}_k^*\leq \frac{\mu_\mathcal{D}\delta^2}{2}$ and thus
		$$
		\dist(\vz^{(k+1)}, \Omega^{(k+1)})\leq  \sqrt{\frac{2(\mathcal{D}_k(\vz^{(k+1)})-\mathcal{D}_k^*)}{\mu_\mathcal{D}}} \leq \delta,
		$$
		which means $\vz^{(k+1)}$ satisfies  \eqref{eq:boundz-2}. 
	\end{proof}

	Below we derive the number of APG steps performed in Algorithm~\ref{alg:restartedAPG}  to find a point satisfying \eqref{eq:boundz-2} under Assumptions~\ref{ass:222} and~\ref{ass:polyhedralg}. In this case, 
	$\bar{g}^{\star}(\vz_1)=\iota_{\vC \vz_1\leq\vd}(\vz_1)$ because the conjugate of $\iota_{\vC \vz_1\leq\vd}(\vz_1)$ is $\max\{\vu\zz\vz_1: \vC \vu\leq\vd\}=\bar g(\vz_1)$ and the conjugate of the conjugate of a closed convex function is just the function itself. In this case, the minimization problem in~\eqref{eq:sublb-2} 
		may not be strongly convex, but it is a quadratic program for which a restarted APG method can still have a linear convergence rate as shown in~\cite{necoara2019linear}.  Recall that $\Omega^{(k+1)}$ is the solution set in~\eqref{eq:sublb-2}. Since $\|\cdot\|^2$ is strongly convex, by the same proof of Eqn.~(40) in~\cite{necoara2019linear}, there exists $\vnu^{(k)}\in\mathbb{R}^d$ such that 
	\begin{eqnarray}
		\label{eq:tk}
	[\bar{\vA}; \vA]\zz\bar\vz^{(k+1)} +\nabla f_0(\vx^{(k)})-\tau \vx^{(k)}=\vnu^{(k)},~\forall\,  \bar\vz^{(k+1)} \in \Omega^{(k+1)}.
		\end{eqnarray}
	As a result, we have 
	\begin{eqnarray}
	\label{eq:sk}
	- (\bar\vz^{(k+1)}_1)^{\top}\bar{\vb}- (\bar\vz^{(k+1)}_2)^{\top}\vb =	\mathcal{D}_k^*-\frac{1}{2\tau}\|\vnu^{(k)}\|^2,~\forall\,  \bar\vz^{(k+1)} \in \Omega^{(k+1)}.
\end{eqnarray}
Therefore, $\Omega^{(k+1)}$ can be characterized  as the solution set of the linear system as follows
\begin{align}
	\label{eq:linearsystem}
\Omega^{(k+1)}=\left\{\vz=(\vz_1\zz,\vz_2\zz)\zz\,\Bigg| \,
\begin{array}{c}
[\bar{\vA}; \vA]\zz\vz=\vnu^{(k)}-\nabla f_0(\vx^{(k)})+\tau \vx^{(k)},\\
 -\vz_1^{\top}\bar{\vb}- \vz_2^{\top}\vb=\mathcal{D}_k^*-\frac{1}{2\tau}\|\vnu^{(k)}\|^2,\\
  \vC \vz_1\leq\vd
\end{array}  
  \right\}.
\end{align}
The Hoffman constant is a constant $\theta(\bar{\vA},\vA,\bar\vb,\vb,\vC)$ such that 
\begin{align}
	\label{eq:defHoffman}
\dist(\vz, \Omega^{(k+1)})\leq \theta(\bar{\vA},\vA,\bar\vb,\vb,\vC)\left\|\left[
\begin{array}{c}
	[\bar{\vA}; \vA]\zz\vz+\nabla f_0(\vx^{(k)})-\tau \vx^{(k)}-\vnu^{(k)}\\
	-\vz_1^{\top}\bar{\vb}- \vz_2^{\top}\vb-\mathcal{D}_k^*+\frac{1}{2\tau}\|\vnu^{(k)}\|^2\\
	 \left(\vC \vz_1-\vd \right)_+
\end{array}
\right]
\right\|, \forall\, \vz.
\end{align}
Note that $\theta(\bar{\vA},\vA,\bar\vb,\vb,\vC)$ does not depend on the right-hand sides of the linear system in~\eqref{eq:linearsystem}. As shown in~\cite{necoara2019linear}, the linear convergence rate of the APG depends on $\theta(\bar{\vA},\vA,\bar\vb,\vb,\vC)$. In a special case where $\bar{\vb}=\mathbf{0}$ and $\vb=\mathbf{0}$, we drop the equality $-\vz_1^{\top}\bar{\vb}- \vz_2^{\top}\vb=\mathcal{D}_k^*-\frac{1}{2\tau}\|\vnu^{(k)}\|^2$ from~\eqref{eq:linearsystem} and denote the corresponding Hoffman constant by $\theta(\bar{\vA},\vA,\vC)$ which satisfies
\begin{align}
	\label{eq:defHoffman_small}
	\dist(\vz, \Omega^{(k+1)})\leq \theta(\bar{\vA},\vA,\vC)\left\|\left[
	\begin{array}{c}
		[\bar{\vA}; \vA]\zz\vz+\nabla f_0(\vx^{(k)})-\tau \vx^{(k)}-\vnu^{(k)}\\
		\left(\vC \vz_1-\vd \right)_+
	\end{array}
	\right]
	\right\| , \forall\, \vz.
\end{align}
With these notations, the number of APG steps to obtain $\vz^{k+1}$ satisfying \eqref{eq:boundz-2} is characterized as follows. 


		\begin{theorem}
			\label{thm:complesub}
				Suppose Assumptions~\ref{assume:problemsetup},~\ref{ass:222} and~\ref{ass:polyhedralg} hold.
			Given any $\delta>0$, if $j_k=\left\lceil 2\sqrt{2 L_\mathcal{D}/\rho_k}\right\rceil$ and $i_k=\left\lceil \log_2(\frac{2(\mathcal{D}_k(\vz^{\text{ini}})-\mathcal{D}_k^*)}{\rho_k\delta^2})\right\rceil$ in Algorithm~\ref{alg:restartedAPG}, the output $\vz^{(k+1)}$ satisfies  the criteria in \eqref{eq:boundz-2}. Here, 
			\begin{align}
				\label{eq:muk}
				\rho_k= \left\{
				\begin{array}{ll}
				\left[\theta^2(\bar{\vA},\vA,\bar\vb,\vb,\vC)(\tau+ \mathcal{D}_k(\vz^{\text{ini}})-\mathcal{D}_k^*+2\tau \|\vnu^{(k)}\|^2)\right]^{-1}&\text{ if }\bar{\vb}\neq\vzero\text{ or } \vb\neq\vzero,\\
		\left[\theta^2(\bar{\vA},\vA,\vC)\tau\right]^{-1} &\text{ if }\bar{\vb}=\vzero\text{ and } \vb=\vzero,
				\end{array}
							\right.
			\end{align}
			with $\vnu^{(k)}$ defined in \eqref{eq:tk} for $k\geq0$.
		\end{theorem}
		
		\begin{proof}
			According to~\cite[Theorem~10]{necoara2019linear}, $\mathcal{D}_k(\vz)$  has a quadratic growth on the level set 
		$$
		\mathcal{Z}_k:=\left\{\vz\big |\vC \vz_1\leq\vd,~\mathcal{D}_k(\vz)\leq \mathcal{D}_k(\vz^{\text{ini}}) \right\}.
		$$
More specifically, it holds that
		\begin{align}
			\label{eq:quadraticgrowth}
		\frac{\rho_k}{2}\dist^2(\vz, \Omega^{(k+1)})\leq \mathcal{D}_k(\vz)-\mathcal{D}_k^*,\quad\forall\, \vz\in \mathcal{Z}_k,
		\end{align}
where $\rho_k$ is defined in \eqref{eq:muk}.

%

Consider the end of the first inner loop (i.e., $i=0$) of Algorithm~\ref{alg:restartedAPG}.  According to~\cite[Theorem 4.4]{beck2009fast}, after $j_k$ APG steps, the APG method generates an output $\vz^{(j_k,k)}$ satisfying 
		$$
		\mathcal{D}_k(\vz^{(j_k,k)})-\mathcal{D}_k^*\leq\frac{2L_\mathcal{D}\dist^2(\vz^{\text{ini}}, \Omega^{(k+1)})}{j_k^2} \overset{\eqref{eq:quadraticgrowth}}\leq\frac{4 L_\mathcal{D} [\mathcal{D}_k(\vz^{\text{ini}})-\mathcal{D}_k^*]}{\rho_k j_k^2}.
		$$
Since $j_k=\left\lceil 2\sqrt{2 L_\mathcal{D}/\rho_k}\right\rceil$, we have $\mathcal{D}_k(\vz^{(j_k,k)})-\mathcal{D}_k^* \le \frac{1}{2}\left(\mathcal{D}_k(\vz^{\text{ini}})-\mathcal{D}_k^*\right)$, meaning that we can reduce the objective gap by a half with each inner loop of  Algorithm~\ref{alg:restartedAPG}. Since the next inner loop is started at $\vz^{(j_k,k)}$ according to Line 11 of Algorithm~\ref{alg:restartedAPG}, we repeat this argument $i_k=\left\lceil\log_2(\frac{2(\mathcal{D}_k(\vz^{\text{ini}})-\mathcal{D}_k^*)}{\rho_k\delta^2})\right\rceil$ times and show that the final output $\vz^{(k+1)}$ satisfies $\mathcal{D}_k(\vz^{(k+1)})-\mathcal{D}_k^*\leq \frac{\rho_k\delta^2}{2}$ and thus
$$
\dist(\vz^{(k+1)}, \Omega^{(k+1)})\leq  \sqrt{\frac{2(\mathcal{D}_k(\vz^{(k+1)})-\mathcal{D}_k^*)}{\rho_k}} \leq \delta,
$$
which means $\vz^{(k+1)}$ satisfies \eqref{eq:boundz-2}.  The proof is then completed. 
%
\end{proof}

\subsection{Total Number of Inner Iterations for Finding an $\epsilon$-stationary Point}

Combining Theorems~\ref{thm:allcomple},~\ref{thm:inner-iter-dual2} and~\ref{thm:complesub}, we derive the total number of inner iterations for computing an $\epsilon$-stationary point of~\eqref{eq:model-spli} as follows. 

\begin{corollary}
	\label{cor:totalcom}
Suppose Assumptions~\ref{assume:problemsetup} and~\ref{ass:222} hold and Algorithm~\ref{alg:ipga-2} uses Algorithm~\ref{alg:restartedAPG} at Step~4. Let $\tau=2L_f$ and $\delta= \delta_\epsilon$ in Algorithm~\ref{alg:ipga-2} with $\delta_\epsilon$ defined in~\eqref{eq:delta}. Then the following statements hold.
\begin{enumerate}	
\item[\textnormal{(a)}] If Assumption~\ref{ass:dual2} holds,  Algorithm~\ref{alg:ipga-2} finds an $\epsilon$-stationary point of problem~\eqref{eq:model-spli} 
	with total number 
	\begin{eqnarray}
	\label{eq:complexity_stronglyconvex}
O\left({\kappa([\bar{\vA}; \vA])} \log\left({\textstyle\frac{\Delta_F}{\epsilon}}\right)\frac{L_f\Delta_F}{\epsilon^2}\right)
	\end{eqnarray}	
of inner iterations (i.e, APG steps),
where $\Delta_F=F(\vx^{(0)}, \vy^{(0)})-\inf_{\vx,\vy}F(\vx, \vy)$.	
\item[\textnormal{(b)}] If Assumption~\ref{ass:polyhedralg} holds and, in addition, $\bar{\vb}=\mathbf{0}, \vb=\mathbf{0}$, 
the same conclusion in statement \textnormal{(a)} holds except that the total number of inner iterations becomes
	\begin{eqnarray}
	\label{eq:complexity_polyhedral_nob}
	O\left(\sqrt{\lambda_{\max}([\bar{\vA}; \vA] [\bar{\vA}; \vA]\zz)\theta^2(\bar{\vA},\vA,\vC)}\log\left({\textstyle\frac{\Delta_F}{\epsilon}}\right)\frac{L_f\Delta_F}{\epsilon^2}\right),
		\end{eqnarray}	
where $\theta(\bar{\vA},\vA,\vC)$ is the Hoffman constant given in~\eqref{eq:defHoffman_small}.		
\item[\textnormal{(c)}]
	If Assumption~\ref{ass:polyhedralg} holds and, in addition, the sequence $\{\vx^{(k)}\}_{k\geq0}$ is bounded, the same conclusion in statement \textnormal{(a)} holds except that the total number of inner iterations becomes
	\begin{eqnarray}
		\label{eq:complexity_polyhedral}
		O\bigg(\sqrt{\lambda_{\max}([\bar{\vA}; \vA] [\bar{\vA}; \vA]\zz)\theta^2(\bar{\vA},\vA,\bar\vb,\vb,\vC)(1+B_6/\tau+2 B_5)}\log\left({\textstyle \frac{B_6}{\epsilon}}\right)\frac{L_f\Delta_F}{\epsilon^2} \bigg),
	\end{eqnarray}
where $ \theta(\bar{\vA},\vA,\bar\vb,\vb,\vC)$ is the Hoffman constant given in~\eqref{eq:defHoffman}, $B_5$ and $B_6$ are constants\footnote{Such $B_5$ and $B_6$ exist because of the boundedness of $\{\vx^{(k)}\}_{k\geq0}$.} independent of $\epsilon$ such that 
\begin{align}
	\label{eq:boundgAx}
	\|\vnu^{(k)}\|^2\leq B_5,\quad \mathcal{D}_k(\vz^{\text{ini}})-\mathcal{D}_k^*\leq B_6,~~~\forall\, k\ge 0, 
\end{align}
with $\vnu^{(k)}$ defined in~\eqref{eq:tk}.
\end{enumerate}	
\end{corollary}


		\begin{proof}
According to Theorem~\ref{thm:allcomple}, Algorithm~\ref{alg:ipga-2} needs $K_\epsilon$ outer iterations to find an $\epsilon$-stationary point of problem~\eqref{eq:model-spli}, where $K_\epsilon$ is given in~\eqref{eq:set-K-eps}. Moreover, in each outer iteration of  Algorithm~\ref{alg:ipga-2}, Algorithm~\ref{alg:restartedAPG} needs $i_kj_k$ APG steps to find a point $\vz^{(k+1)}$  satisfying \eqref{eq:boundz-2} with $i_k$ and $j_k$ specified in  Theorems~\ref{thm:inner-iter-dual2} and~\ref{thm:complesub} for different assumptions. 
Hence, the total number of inner iterations by Algorithm~\ref{alg:ipga-2} is $\sum_{k=0}^{K_\epsilon-1}i_kj_k.$

Below we prove statement (c) first. In this case,  $i_k$ and $j_k$ are given in  Theorem~\ref{thm:complesub} and they depend on $\rho_k$ in \eqref{eq:muk} which further depends on $\|\vnu^{(k)}\|$ and $\mathcal{D}_k(\vz^{\text{ini}})-\mathcal{D}_k^*$. Using \eqref{eq:boundgAx} and the fact that $\log\frac{1}{\delta} = O(\log\frac{1}{\epsilon})$. We can bound $i_k$ and $j_k$ from above and obtain \eqref{eq:complexity_polyhedral}.

To  prove statements (a) and (b), we only need to derive an upper bound for $\mathcal{D}_k(\vz^{\text{ini}})-\mathcal{D}_k^*$ that appears in the formula of $i_k$ in  Theorems~\ref{thm:inner-iter-dual2} and~\ref{thm:complesub}. First, we observe that, for any $k\leq K_\epsilon $, it holds that
\begin{align}
\nonumber
\|\vx^{(k)}\| \le &\|\vx^{(0)}\| + \sum_{s=0}^{k-1}\|\vx^{(s)}-\vx^{(s+1)}\|\leq \|\vx^{(0)}\| + \sum_{s=0}^{K_\epsilon-1}\|\vx^{(s)}-\vx^{(s+1)}\|\\\label{eq:xkx0}
\leq &\|\vx^{(0)}\| + \sqrt{K_\epsilon}\cdot \sqrt{\sum_{s=0}^{K_\epsilon-1}\|\vx^{(s)}-\vx^{(s+1)}\|^2} \leq \|\vx^{(0)}\| + \sqrt{K_\epsilon}\cdot \sqrt{\frac{K_\epsilon\cdot \epsilon^2}{9L_f^2}} \leq \|\vx^{(0)}\| + \frac{4\Delta_F}{\epsilon},
\end{align}
where the third inequality is by the Cauchy–Schwarz inequality, the fourth one is by~\eqref{eq:Kepsilon1}, and the last one follows from \eqref{eq:set-K-eps}. Recall that $\|\bar{\vz}_1^{(k+1)}\|\leq l_g$ and thus from \eqref{eq:formula-z-2-k+1} and~\eqref{eq:xkx0}, we have
\begin{equation}\label{eq:bd-bar-z2-k+1}
\|\bar{\vz}_2^{(k+1)}\| \le l_g + \left(l_f + \tau \|\vx^{(0)}\| + \frac{4\tau\Delta_F}{\epsilon} \right)\left\|\left( \vA \vA\zz\right)^{-1} \vA\right\| + \tau \left\|\left( \vA \vA\zz\right)^{-1} \vb\right\|, 
\end{equation}
where we have used the $l_f$-Lipschitz continuity of $f_0$ from Assumption~\ref{ass:222}.

Moreover,  for any $\vz$, we have
\begin{align}
\|\nabla (\mathcal{D}_k - \bar g^\star)(\vz)\| = &~\left\|\frac{1}{\tau}\left([\bar \vA; \vA] \big([\bar \vA; \vA] ^\top \vz +\nabla f_0(\vx^{(k)})-\tau \vx^{(k)}\big)\right)- [\bar\vb; \vb]\right\| \cr
\le &~\frac{1}{\tau} \|[\bar \vA; \vA]\|^2 \cdot \|\vz\| + \frac{1}{\tau}\|[\bar \vA; \vA]\|\left(l_f + \tau \|\vx^{(0)}\| + \frac{4\tau\Delta_F}{\epsilon}\right) + \|[\bar\vb; \vb]\| =: \widehat{\mathcal{D}}(\vz).
\end{align}
Hence, it follows from the convexity of $\mathcal{D}_k$ that for any $\vxi \in \partial \bar g^\star(\vz^{\text{ini}})$,
$$
\mathcal{D}_k(\vz^{\text{ini}})-\mathcal{D}_k^*\leq \left\langle \nabla (\mathcal{D}_k - \bar g^\star)(\vz^{\text{ini}}) + \vxi, \vz^{\text{ini}}-\bar\vz^{(k+1)}\right\rangle 
\leq \big(\widehat{\mathcal{D}}(\vz^{\text{ini}}) + \|\vxi\|\big) \left(\|\vz^{\text{ini}}\| + l_g + \|\bar{\vz}_2^{(k+1)}\|\right).
$$
Keeping only the key quantities $\epsilon$ and $\Delta_F$, we then have $\mathcal{D}_k(\vz^{\text{ini}})-\mathcal{D}_k^*= O\left( \frac{\Delta_F}{\epsilon}\right)$.  Applying this bound of $\mathcal{D}_k(\vz^{\text{ini}})-\mathcal{D}_k^*$ to the $i_k$'s in Theorems~\ref{thm:inner-iter-dual2} and~\ref{thm:complesub}, we obtain~\eqref{eq:complexity_stronglyconvex} and~\eqref{eq:complexity_polyhedral_nob}.
\end{proof}

In the following corollary, we also derive from Theorems~\ref{thm:allcompleforp},~\ref{thm:inner-iter-dual2} and~\ref{thm:complesub} the required total number of inner iterations for obtaining a near $\epsilon$-stationary point of~\eqref{eq:model}. We skip its proof because it is essentially the same as that of Corollary~\ref{cor:totalcom}.
					
\begin{corollary}
	\label{cor:totalcom2}
		Suppose Assumptions~\ref{assume:problemsetup} and~\ref{ass:222} hold and Algorithm~\ref{alg:ipga-2} uses Algorithm~\ref{alg:restartedAPG} at Step~4.  Let $\tau=2L_f$ and $\delta= \bar{\delta}_\epsilon$ in Algorithm~\ref{alg:ipga-2} with $\bar{\delta}_\epsilon$ defined in Theorem~\ref{thm:allcompleforp}.  Then the following statements hold.
		\begin{enumerate}	
			\item[\textnormal{(a)}] If Assumption~\ref{ass:dual2} holds, Algorithm~\ref{alg:ipga-2} finds a point that is $\frac{\epsilon}{12L_f}$-close to an $\epsilon$-stationary point of problem~\eqref{eq:model} with total number  
			\begin{eqnarray}
				\label{eq:complexity_stronglyconvex2}
				O\left({\kappa([\bar{\vA}; \vA])} \log\left({\textstyle\frac{\Delta_{F_0}}{\epsilon}}\right)\frac{L_f\Delta_{F_0}}{\epsilon^2}\right)
			\end{eqnarray}	
			of inner iterations (i.e, APG steps), where $\Delta_{F_0}=F_0(\vx^{(0)})-\inf_{\vx}F_0(\vx)$.	
			\item[\textnormal{(b)}] If Assumption~\ref{ass:polyhedralg} holds and, in addition, $\bar{\vb}=\mathbf{0}, \vb=\mathbf{0}$, 
			 the same conclusion in statement \textnormal{(a)} holds except that the total number of inner iterations  becomes
			\begin{eqnarray}
				\label{eq:complexity_polyhedral_nob2}
				O\left(\sqrt{\lambda_{\max}([\bar{\vA}; \vA] [\bar{\vA}; \vA]\zz)\theta^2(\bar{\vA},\vA,\vC)}\log\left({\textstyle\frac{\Delta_{F_0}}{\epsilon}}\right)\frac{L_f\Delta_{F_0}}{\epsilon^2}\right),
			\end{eqnarray}	
			where $\theta(\bar{\vA},\vA,\vC)$ is the Hoffman constant given in~\eqref{eq:defHoffman_small}.		
			\item[\textnormal{(c)}]
			If Assumption~\ref{ass:polyhedralg} holds and, in addition, the sequence $\{\vx^{(k)}\}_{k\geq0}$ is bounded, the same conclusion in statement \textnormal{(a)} holds except that the total number of inner iterations  becomes
			\begin{eqnarray}
				\label{eq:complexity_polyhedral2}
				O\bigg(\sqrt{\lambda_{\max}([\bar{\vA}; \vA] [\bar{\vA}; \vA]\zz)\theta^2(\bar{\vA},\vA,\bar\vb,\vb,\vC)(1+B_6/\tau+2 B_5)}\log\left({\textstyle \frac{B_6}{\epsilon}}\right)\frac{L_f\Delta_{F_0}}{\epsilon^2} \bigg),
			\end{eqnarray}
			where $ \theta(\bar{\vA},\vA,\bar\vb,\vb,\vC)$ is the Hoffman constant given in~\eqref{eq:defHoffman}, $B_5$ and $B_6$ are given in Corollary \ref{cor:totalcom}.
		\end{enumerate}	
\end{corollary}
					
\subsection{Oracle Complexity Matching the Lower Bounds}
					
Suppose Algorithm~\ref{alg:ipga-2} uses Algorithm~\ref{alg:restartedAPG} at Step~4 to find a point $\vz^{(k+1)}$  satisfying \eqref{eq:boundz-2}. We can show that Algorithm~\ref{alg:ipga-2} satisfies Assumption~\ref{ass:linearspan3}, i.e., the iterate points of Algorithm~\ref{alg:ipga-2} can be generated by only querying the oracles and performing the operations given in Assumption~\ref{ass:linearspan3}. Then, we show that the oracle  complexity of Algorithm~\ref{alg:ipga-2} matches the lower complexity bound presented in Section \ref{sec:extension}. To do so, we first present the following observations. 




\begin{lemma}
	\label{lem:spanz}
	Let $\vz^{(j,k)}$ be the solution generated in the $j$-th inner iteration of the $k$-th outer iteration of Algorithm~\ref{alg:restartedAPG}. Let $\vx^{(j,k)}=\vA\zz\vz_2^{(j,k)}$. 
	It holds that, for all $j\geq1$ and $k\ge0$,
	\begin{subequations}	\label{eq:spanz2}
		\begin{align}
			\label{eq:spanz2-Az2}
			&\vx^{(j,k)}\in  \vA^\top\vA\cdot\mathbf{span} \left( \left\{\vx^{(k)},\nabla f_0(\vx^{(k)}),\vx^{(0)},\bar{\vA}^\top\bar{\vb}\right\}\bigcup\cup_{i=0}^{j-1}\left\{\vx^{(i,k)}, \bar{\vA}\zz\vz_1^{(i,k)} \right\} \right)\\
			\label{eq:spanz2-z1}
			&\vz_1^{(j,k)}\in \mathbf{span} \left( \big\{ \vzeta^{(j,k)}, \vxi^{(j,k)} \big\}\right),  
		\end{align}			
	\end{subequations}		
	where $\vzeta^{(j,k)}\in \left\{\prox_{\eta \bar{g}}(\vxi^{(j,k)})~|~\eta>0\right\}$ and
	\begin{equation}\label{eq:spanz2-xi}
		\vxi^{(j,k)}\in \mathbf{span} \left( \left\{\bar{\vA}\vx^{(k)},\bar{\vA}\nabla f_0(\vx^{(k)}),\bar{\vb}\right\}\bigcup\cup_{i=0}^{j-1}\left\{\vz_1^{(i,k)}, \bar{\vA}\bar{\vA}\zz\vz_1^{(i,k)}, \bar{\vA} \vx^{(i,k)}\right\}  \right).
	\end{equation}	
\end{lemma}

\begin{proof}
	Consider any $j\geq0$ and $k\geq0$. By Steps 2 and 9 in  Algorithm~\ref{alg:restartedAPG}, we have $\widehat\vz^{(0,k)}=\vz^{(0,k)}$ and $\widehat\vz^{(j+1,k)}=\vz^{(j+1,k)}+\left(\frac{\alpha_{j}-1}{\alpha_{j+1}}\right)\left(\vz^{(j+1,k)}-\vz^{(j,k)}\right)$ for $j\geq0$. This further implies 
	$\vA\zz\widehat\vz_2^{(0,k)}=\vx^{(0,k)}$ and $\vA\zz\widehat\vz_2^{(j+1,k)}=\vx^{(j+1,k)}+\left(\frac{\alpha_{j}-1}{\alpha_{j+1}}\right)\left(\vx^{(j+1,k)}-\vx^{(j,k)}\right)$ for $j\geq0$.
	
	Hence, by Steps 5 and 7 in  Algorithm~\ref{alg:restartedAPG},
	$$
	\widehat\vz_1^{(j-1,k)}-\frac{1}{L_\mathcal{D}}\widehat{\mathcal{G}}_1^{j-1}\in \mathbf{span} \left( \left\{\bar{\vA}\vx^{(k)},\bar{\vA}\nabla f_0(\vx^{(k)}),\bar{\vb}\right\}\bigcup\cup_{i=0}^{j-1}\left\{\vz_1^{(i,k)}, \bar{\vA}\bar{\vA}\zz\vz_1^{(i,k)}, \bar{\vA} \vx^{(i,k)}\right\}  \right)
	$$
	and 
	$$
	\vz_1^{(j,k)}=\mathbf{prox}_{L_\mathcal{D}^{-1}\bar{g}^*}\left(\widehat\vz_1^{(j-1,k)}-L_\mathcal{D}^{-1}\widehat{\mathcal{G}}_1^{j-1}\right)
	=
	\widehat\vz_1^{(j-1,k)}-L_\mathcal{D}^{-1}\widehat{\mathcal{G}}_1^{j-1}-L_\mathcal{D}^{-1}\mathbf{prox}_{L_\mathcal{D}\bar{g}}\left(	L_\mathcal{D}\widehat\vz_1^{(j-1,k)}-\widehat{\mathcal{G}}_1^{j-1}\right)
	$$
	which means \eqref{eq:spanz2-z1} holds. Recall that $\vb=-\vA\vx^{(0)}$. By Steps 6 and 7 in  Algorithm~\ref{alg:restartedAPG},
	$$
	\vA\zz\widehat{\mathcal{G}}_2^{j-1}\in \vA^\top\vA\cdot\mathbf{span} \left( \left\{\vx^{(k)},\nabla f_0(\vx^{(k)}),\vx^{(0)},\bar{\vA}^\top\bar{\vb}\right\}\bigcup\cup_{i=0}^{j-1}\left\{\vx^{(i,k)}, \bar{\vA}\zz\vz_1^{(i,k)} \right\} \right)
	$$
	and 
	\begin{align*}
		\vx_2^{(j,k)}=&\vA\zz\vz_2^{(j,k)}=\vA\zz\widehat\vz_2^{(j-1,k)}-L_\mathcal{D}^{-1}\vA\zz\widehat{\mathcal{G}}_2^{j-1}\\
		\in&
		\vA^\top\vA\cdot\mathbf{span} \left( \left\{\vx^{(k)},\nabla f_0(\vx^{(k)}),\vx^{(0)},\bar{\vA}^\top\bar{\vb}\right\}\bigcup\cup_{i=0}^{j-1}\left\{\vx^{(i,k)}, \bar{\vA}\zz\vz_1^{(i,k)} \right\} \right)
	\end{align*}
	which means \eqref{eq:spanz2-Az2} holds. 
%
\end{proof}

With these observations, we can formally show that Algorithm~\ref{alg:ipga-2} follows Assumption~\ref{ass:linearspan3}.
\begin{proposition}
	\label{lem:spanz_subseq}
The updating schemes in Algorithm~\ref{alg:ipga-2} using Algorithm~\ref{alg:restartedAPG} at Step~4 satisfy Assumption~\ref{ass:linearspan3}.
\end{proposition}
\begin{proof}
	We first claim that Algorithms~\ref{alg:ipga-2} and~\ref{alg:restartedAPG} can be both implemented by computing and updating $\vz_1^{(j,k)}$ and $\vx^{(j,k)}=\vA\zz\vz_2^{(j,k)}$ without explicitly computing $\vz_2^{(j,k)}$. For Algorithm~\ref{alg:restartedAPG}, this claim can be verified by its updating steps. For Algorithms~\ref{alg:ipga-2}, updating steps \eqref{eq:xupdate-2} and \eqref{eq:yupdate-2} can be also implemented directly based on $\vz_1^{(j,k)}$ and $\vx^{(j_k,k)}$ without knowing $\vz_2^{(k+1)}$ because $\vA\zz\vz_2^{(k+1)}=\vA\zz\vz_2^{(j_k,k)}=\vx^{(j_k,k)}$. 
	
	Comparing  \eqref{eq:spanz2} and \eqref{eq:spanz2-xi} with the updating schemes allowed by Assumption~\ref{ass:linearspan3}, we can prove that 	Algorithm~\ref{alg:ipga-2} satisfies Assumption~\ref{ass:linearspan3} by proving that  $\vx^{(0,k)}=\vA\zz\vz_2^{(0,k)}$ and $\vx^{(k)}$ for any $k$ are generated only  by the updating schemes allowed by Assumption~\ref{ass:linearspan3}. We will prove this statement by induction on $k$.
	
	Suppose  $k=0$. Consider $\vx^{(0,k)}=\vA\zz\vz_2^{(0,k)}$ in the first inner loop (i.e., $i=0$) of Algorithm~\ref{alg:restartedAPG}. Since $\vx^{(0,k)}=\vA\zz\vz_2^{(0,k)}=\vA\zz\vA\vx^{(0)}$, $\vx^{(0,k)}$ is generated by the schemes allowed by Assumption~\ref{ass:linearspan3}. According to  \eqref{eq:spanz2} and \eqref{eq:spanz2-xi} and the fact that $k=0$, we conclude that all iterates during the first inner loop of  Algorithm~\ref{alg:restartedAPG}  are generated only by the updating schemes allowed by Assumption~\ref{ass:linearspan3}. Consider the second inner loop (i.e., $i=1$) of Algorithm~\ref{alg:restartedAPG}. Recall that the initial solution of this inner loop is created as $\vx^{(0,k)}=\vA\zz\vz_2^{(0,k)}=\vA\zz\vz_2^{(j_k,k)}=\vx^{(j_k,k)}$, where $\vx^{(j_k,k)}$ is the outputs of the first inner loop of Algorithm~\ref{alg:restartedAPG}. We then also conclude that all iterates during the second inner loop of  Algorithm~\ref{alg:restartedAPG} are generated only by the schemes allowed by Assumption~\ref{ass:linearspan3}. Specifically, all iterates during the second inner loop of  Algorithm~\ref{alg:restartedAPG} can be generated by accessing two oracles given in Assumption~\ref{ass:linearspan3}.
	Repeating this argument for all inner loops  of  Algorithm~\ref{alg:restartedAPG}, we prove the statement for $k=0$. 
	
	Suppose the statement is true for $k\geq0$. Then $\vz_1^{(k+1)}$ and $\vA\zz\vz_2^{(k+1)}$ are generated by the updating schemes allowed by Assumption~\ref{ass:linearspan3}, so are $\vx^{(k+1)}$ and $\vy^{(k+1)}$ according to \eqref{eq:xupdate-2} and \eqref{eq:yupdate-2}. By the same argument as when $k=0$, we can prove that the statement is true for $k+1$. The proof is completed by induction.  
\end{proof}

By Proposition \ref{lem:spanz_subseq} and its proof, we have that the oracle complexity of Algorithm \ref{alg:ipga-2} equals $O(1)$ times of the total number of inner iterations (i.e., APG steps) of Algorithm \ref{alg:ipga-2}. 
We then finish the main body of this paper with several remarks to point out that the lower bound of oracle complexity matches, up to a difference of logarithmic factors, the oracle complexity in Corollaries~\ref{cor:totalcom} and~\ref{cor:totalcom2} under two cases: (i)  Assumptions~\ref{ass:222} and~\ref{ass:dual2} hold; (ii) Assumptions~\ref{ass:222} and~\ref{ass:polyhedralg} hold and in addition, $\bar{\vb}=\mathbf{0}$ and $\vb=\mathbf{0}$.
					
\begin{remark} \label{rem:matchsp}
 A few remarks follow on the oracle complexity of FOMs for solving problem~\eqref{eq:model-spli}. 
\begin{enumerate}
\item[\textnormal{(i)}] 
By Proposition~\ref{lem:spanz_subseq}, 	the oracle complexity of Algorithm~\ref{alg:ipga-2} using Algorithm~\ref{alg:restartedAPG} at Step~4 is subject to the  lower bound established in Theorem~\ref{thm:lbcomposite} for finding an $\epsilon$-stationary point of~\eqref{eq:model-spli}. We then assert that this lower bound and the oracle complexity\footnote{The dependency of oracle complexity in~\eqref{eq:complexity_stronglyconvex} on $l_f$ and $l_g$ is only logarithmic and has been suppressed in $O(\cdot)$.} in~\eqref{eq:complexity_stronglyconvex} are both not improvable (up to a logarithmic factor) under Assumptions~\ref{ass:222} and~\ref{ass:dual2} by comparing their dependency on $\epsilon$, $L_f$ and $\Delta_F$ and $\kappa([\bar{\vA}; \vA])$. 
	
\item[\textnormal{(ii)}] Suppose Algorithm~\ref{alg:ipga-2} is applied to the reformulation~\eqref{eq:model-spli} of instance~$\mathcal{P}$ given in Definition~\ref{def:hardinstance}. By Lemma~\ref{cor:boundf}, instance~$\mathcal{P}$ satisfies Assumptions~\ref{ass:222} and~\ref{ass:polyhedralg}. In particular, for instance~$\mathcal{P}$,  it holds that $\bar g(\vy)=\frac{\beta}{mL_f}\|\vy\|_1=\max_{\|\vz_1\|_\infty\leq\frac{\beta}{mL_f}}\vz_1^\top\vy$ and  $\bar g^\star(\vz_1)=\iota_{\|\vz_1\|_\infty\leq \frac{\beta}{mL_f}}$. Moreover, since $\bar{\vb}=\mathbf{0}$ and $\vb=\mathbf{0}$ in  instance~$\mathcal{P}$,  Algorithm~\ref{alg:ipga-2} finds an $\epsilon$-stationary point of~\eqref{eq:model-spli} with oracle complexity given in \eqref{eq:complexity_polyhedral_nob}. In this case,  $\theta(\bar{\vA},\vA,\vC)$ becomes the Hoffman constant of the linear system
$$
\left\{\vz=(\vz_1\zz,\vz_2\zz)\zz\Bigg| 
\begin{array}{c}
	[\bar{\vA}; \vA]\zz\vz=\vnu^{(k)}-\nabla f_0(\vx^{(k)})+\tau \vx^{(k)},\\
 -\frac{\beta}{mL_f}\leq[\vz_{1}]_i\leq \frac{\beta}{mL_f}, \forall i
\end{array}  
\right\},
$$
for some $\vnu^{(k)}$. According to Proposition 6 and the discussion on page 12 of~\cite{pena2021new}, we can define a reference polyhedron 
$$
\mathcal{R} = \left\{\vz=[\vz_1;\vz_2]:  -\frac{\beta}{mL_f}\leq[\vz_{1}]_i\leq \frac{\beta}{mL_f}, \forall i\right\}
$$
and characterize $\theta(\bar{\vA},\vA,\vC)$ as
$$
\frac{1}{\theta(\bar{\vA},\vA,\vC)}=\min_{\mathcal{K}\in\mathcal{S}}
	\min\limits_{\vu,\vv}
\left\{ \|\vu\|~\Big|~\|\vH^\top \vv\|= 1,~\vv\in \mathcal{K},~\vH\vH^\top \vv-\vu\in \mathcal{K}^*
\right\},
$$
where $\vH$ is defined in~\eqref{eq:matrixAstar}, $\mathcal{K}^*$ is the dual cone of $\mathcal{K}$, and
$$
\mathcal{S}=\{\mathcal{K}: \mathcal{K} \text{ is a tangent cone of }\mathcal{R}\text{ at some point of }\mathcal{R}\text{, and }  \vH^\top \mathcal{K} \text{  is a linear space}\}.
$$

Next we show $\mathcal{S}=\{\mathbb{R}^{d}\}$. By the definition of $\mathcal{R}$, any tangent cone of $\mathcal{R}$ must look like 
$$
\mathcal{K}=\left\{\vz=[\vz_1;\vz_2]~\big|~[\vz_{1}]_i\geq 0, \forall i\in \mathcal{J}_1, [\vz_{1}]_j\leq 0, \forall j\in \mathcal{J}_2\right\}
$$
for some {disjoint} index sets $\mathcal{J}_1$ and $\mathcal{J}_2$. Suppose  $\mathcal{K}\in\cS$. 
Then, for any $\vv\in \mathcal{K}$, we must have $-\vH^\top\vv\in \vH^\top \mathcal{K}$, namely, there exists $\vv'\in\mathcal{K}$ such that $-\vH^\top\vv=\vH^\top\vv'$. Since $\vH^\top$ has a full-column rank, $\vv+\vv'=\vzero$. This means for any $\vv\in \mathcal{K}$,  we must have $-\vv\in \mathcal{K}$. {Since $\mathcal{J}_1\cap\mathcal{J}_2=\emptyset$, this happens only if} $\mathcal{J}_1=\mathcal{J}_2=\emptyset$. Thus $\mathcal{K}=\mathbb{R}^{ d}$ and $\mathcal{K}^*=\{\vzero\}$. 

With this fact, we have
$$
\theta(\bar{\vA},\vA,\vC)=\frac{1}{\min\limits_{\vv: \|\vH^\top \vv\|= 1}\|\vH\vH^\top \vv\|}=\frac{1}{\sqrt{\lambda_{\min}(\vH\vH^\top)}}.
$$ 
Hence by Corollary~\ref{cor:totalcom}(b), the oracle complexity of Algorithm~\ref{alg:ipga-2} becomes
$$
O\left(\sqrt{\frac{\lambda_{\max}(\vH \vH\zz)}{\lambda_{\min}(\vH \vH\zz)}}\log({\textstyle \frac{1}{\epsilon}})\frac{L_f\Delta_F}{\epsilon^2}\right)=
O\left(\kappa([\bar{\vA};\vA])\log({\textstyle \frac{1}{\epsilon} })\frac{L_f\Delta_F}{\epsilon^2}\right).
$$
Again, we conclude that the lower bound given in Theorem~\ref{thm:lbcomposite} and the oracle complexity above are neither improvable (up to a logarithmic factor) under  Assumptions~\ref{ass:222} and~\ref{ass:polyhedralg} when $\bar{\vb}=\mathbf{0}$ and $\vb=\mathbf{0}$.
\end{enumerate}
\end{remark}			

\begin{remark} A few remarks follow on the oracle complexity of FOMs for solving problem~\eqref{eq:model}. 
	\begin{enumerate}
		\item[\textnormal{(i)}] 
		By Proposition~\ref{lem:spanz_subseq} again, the oracle  complexity of Algorithm~\ref{alg:ipga-2} using Algorithm~\ref{alg:restartedAPG} at Step~4  is subject to the  lower bound established in Corollary~\ref{cor:lower2} for finding a point $\vx^{(k)}$ that is $\omega$-close to  an $\epsilon$-stationary point of problem~\eqref{eq:model} for some $\omega\in[0,\frac{150\pi\epsilon}{L_f})$. We then assert that this lower bound and the oracle complexity in~\eqref{eq:complexity_stronglyconvex2} (with $\omega=\frac{\epsilon}{12L_f}$) are both not improvable, up to a logarithmic factor, under Assumptions~\ref{ass:222} and~\ref{ass:dual2} by comparing their dependency on $\epsilon$, $L_f$ and $\Delta_{F_0}$ and $\kappa([\bar{\vA}; \vA])$. 
		
		\item[\textnormal{(ii)}] Suppose Algorithm~\ref{alg:ipga-2} is applied to instance~$\mathcal{P}$ given in Definition~\ref{def:hardinstance}. We obtain from Remark \ref{rem:matchsp}(ii) that if $\bar{\vb}=\mathbf{0}$ and $\vb=\mathbf{0}$, then Algorithm~\ref{alg:ipga-2} finds a near $\epsilon$-stationary point of~\eqref{eq:model} with oracle complexity~\eqref{eq:complexity_polyhedral_nob2}, which equals $O\left(\kappa([\bar{\vA};\vA])\log({\textstyle \frac{1}{\epsilon} })\frac{L_f\Delta_{F_0}}{\epsilon^2}\right)$ under  Assumptions~\ref{ass:222} and~\ref{ass:polyhedralg}. Hence, the lower bound given in Corollary~\ref{cor:lower2} and the oracle complexity above are neither improvable (up to a logarithmic factor) under  Assumptions~\ref{ass:222} and~\ref{ass:polyhedralg} when $\bar{\vb}=\mathbf{0}$ and $\vb=\mathbf{0}$.
	\end{enumerate}
\end{remark}	


\section{Conclusion and Open Questions}\label{sec:conclusion}
We present a lower bound for the oracle complexity of first-order methods for finding a (near) $\epsilon$-stationary point of a non-convex composite non-smooth optimization problem with affine equality constraints. We also show that the same-order lower bound holds for first-order methods applied to a reformulation of the original problem. 
In addition, we design an inexact proximal gradient method for the reformulation. To find an $\epsilon$-stationary point (that is also a near $\epsilon$-stationary point of the original problem), our designed method has an oracle complexity that matches the established lower bounds under two scenarios, up to a difference of a logarithmic factor. This shows that both the lower bound and the oracle complexity of the  inexact proximal gradient method are nearly not improvable. 


Though we make the first attempt on establishing lower complexity bounds of first-order methods for solving affine-constrained non-convex composite non-smooth problems, there are still a few open questions that are worth further exploration. 
First, under Assumption~\ref{ass:linearspan}, can the lower bound $O({\kappa([\bar{\vA};\vA]) L_f \Delta_{F_0}} \epsilon^{-2})$ for problem~\eqref{eq:model} be achieved? The second question is whether the lower bound $O({\kappa([\bar{\vA};\vA]) L_f \Delta_{F}} \epsilon^{-2})$ for problem~\eqref{eq:model-spli} can be achieved by an algorithm satisfying Assumption~\ref{ass:linearspan3} without Assumptions~\ref{ass:222},~\ref{ass:dual2} or~\ref{ass:polyhedralg}. Third, what will the lower bound look like if there are convex nonlinear inequality constraints? 


\section*{Acknowledgements}

This work is partly supported by NSF awards DMS-2053493 and IIS-2147253 and the ONR award N00014-22-1-2573. 


				
			\bibliographystyle{plain}
			\bibliography{optim}

\appendix

\section{Proofs of Theorem~\ref{thm:lbcomposite} and Corollary \ref{cor:lower2}}		
\label{sec:appen1}
In this section, we give a complete proof of Theorem~\ref{thm:lbcomposite} and Corollary \ref{cor:lower2}.
We first show a lemma and a proposition.			
			According to the structure of $\vA$ and $\bar{\vA}$ given in~\eqref{eq:AandAbar},  $\textnormal{supp}((\vA^{\top}\vA\vx)_i)$ and $\textnormal{supp}((\bar{\vA}^{\top}\bar{\vA}\vx)_i)$ are determined by $\textnormal{supp}(\vx_{i-1})$, $\textnormal{supp}(\vx_{i})$ and $\textnormal{supp}(\vx_{i+1})$. Also, $\textnormal{supp}(\prox_{\eta \bar{g}}(\vy))$, $\bar{\vA}\zz\vy$ and $\vA\vx$ have a similar property. They are stated in the following lemma.

\begin{lemma}
	\label{lem:supp2}
	Let $\vx$ be the structured vector given in~\eqref{eq:xblock}, $\vA$ in~\eqref{eq:AandAbar}, and $\bar{g}$ be given in~\eqref{eq:gbar}. Define $\vx_0=\vx_{m+1}=\mathbf{0}\in\mathbb{R}^{\bar d}$. The following statements hold:
	\begin{enumerate}
		\item[\textnormal{(a)}] Let  $\widehat\vx=(\widehat\vx_1\zz, \ldots,\widehat\vx_m\zz)\zz\in\textnormal{\textbf{span}}\{\vA^{\top}\vA\vx, \bar{\vA}^{\top}\bar{\vA}\vx\}$ with $\widehat\vx_i\in\RR^{{\bar{d}}}$.  Then 
		\begin{equation}\label{eq:supp-relation2}
			\textnormal{supp}(\widehat\vx_i)\subset \textnormal{supp}(\vx_{i-1})\cup\textnormal{supp}(\vx_{i})\cup\textnormal{supp}(\vx_{i+1}),\, \forall\, i\in [1, m].
		\end{equation}		
		
		\item[\textnormal{(b)}] Let  $\vy=(\vy^{\top}_1, \ldots,\vy^{\top}_{3m_2-1})\zz$ with $\vy_j\in\RR^{{\bar{d}}}$ and $\widetilde \vx=\bar\vA^\top \vy$. Then 
		\begin{eqnarray}
			\label{eq:suppAy}
			\textnormal{supp}(\widetilde\vx_i)&\subset&
			\left\{
			\begin{array}{ll}
				\emptyset&\text{ if }i-1,i\notin\mathcal{M},\\
				\textnormal{supp}(\vy_{j})&\text{ if }i-1=jm_1\in\mathcal{M},\\
				\textnormal{supp}(\vy_{j})&\text{ if }i=jm_1\in\mathcal{M},
			\end{array}
			\right.\,\,\forall i\in[1,m].
		\end{eqnarray} 
		
		\item[\textnormal{(c)}]	Let $\widehat\vy=\bar\vA\vx$ and $\widehat\vy=(\widehat\vy_1\zz, \ldots,\widehat\vy_{3m_2-1}\zz)\zz$ with $\widehat\vy_j\in\RR^{{\bar{d}}}$. Then 
		\begin{eqnarray}
			\label{eq:suppAbar3}
			\textnormal{supp}(\widehat\vy_j)
			&\subset&\textnormal{supp}(\vx_{jm_1})\cup\textnormal{supp}(\vx_{jm_1+1}),\,\,\forall j\in[1,3m_2-1].
		\end{eqnarray} 
		
		\item[\textnormal{(d)}] It holds that 
		\begin{eqnarray}
			\label{eq:suppAbarAbary}
			\textnormal{supp}(\vy)&=&\textnormal{supp}(\bar\vA\bar\vA\zz\vy),\,\,\forall \vy\in\mathbb{R}^{\bar n}.
		\end{eqnarray}
		
		\item[\textnormal{(e)}] For any given $\eta>0$, let $\widetilde\vy=\prox_{\eta \bar g}(\vy)=(\widetilde\vy_1\zz, \ldots,\widetilde\vy_{3m_2-1}\zz)\zz$ with $\widetilde\vy_j\in\RR^{{\bar{d}}}$.   Then 
		\begin{eqnarray}
			\label{eq:suppg3}
			\textnormal{supp}(\widetilde\vy_j)&\subset&\textnormal{supp}(\vy_j),\,\,\forall j\in[1,3m_2-1].
		\end{eqnarray}
		
	\end{enumerate}
\end{lemma}

\begin{proof}
	(a) 
	Recall that $\vH$ are split into $\bar{\vA}$ and $\vA$ in rows. The relation in~\eqref{eq:supp-relation2} immediately follows from \eqref{eq:ata-struc} the observation
	\begin{equation}\label{eq:ata-struc2}
		\newcommand{\Hmat}{\left[ 
			\begin{array}{rrrrr} 
				\vI_{\bar d} & - \vI_{\bar d} & & & \\
				- \vI_{\bar d} & 2 \vI_{\bar d} & - \vI_{\bar d} & \\
				& \ddots & \ddots & \ddots & \\
				&  &	- \vI_{\bar d} & 2 \vI_{\bar d} & - \vI_{\bar d}  \\
				& & & - \vI_{\bar d} & \vI_{\bar d}
			\end{array} \right]
		}
		\vH\zz\vH = m^2 L_f^2 
		\left.
		\,\smash[b]{\underbrace{\!\Hmat\!}_{\textstyle\text{$m-1$ blocks}}}\,
		\right\}\text{$m-1$ blocks}
		\vphantom{\underbrace{\Hmat}_{\textstyle\text{$m$ blocks}}}
			~\text{  and  }~\bar{\vA}^{\top}\bar{\vA}=	\vH\zz\vH- 	\vA^{\top}\vA.
	\end{equation}
	
	(b)  The relation in~\eqref{eq:suppAy} immediately follows from the definitions of $\bar\vA$ and $\mathcal{M}$ in~\eqref{eq:indexsetM}.
	
	(c) The relation in~\eqref{eq:suppAbar3} immediately follows from the definition of $\bar\vA$ and $\mathcal{M}$ in~\eqref{eq:indexsetM}.
	
	(d) The relation in~\eqref{eq:suppAbarAbary} immediately follows from the fact that $\bar\vA\bar\vA\zz=2m^2L_f^2\vI_{\bar{n}}$ by
	the definition of $\bar\vA$ and $\mathcal{M}$ in~\eqref{eq:indexsetM}.
	
	(e) Given any $y\in\mathbb{R}$ and any $c>0$, consider the following optimization problem in $\mathbb{R}$:
	\begin{eqnarray}
		\label{eq:2dproxg3}
		\widetilde y=\argmin_{z\in\mathbb{R}}\frac{1}{2}(z-y)^2+c|z|=\text{sign}(y)\cdot (|y|-c)_+.
	\end{eqnarray}
	Recall the definition of $\bar{g}$ in \eqref{eq:gbar}, we obtain that 
	$$
	\widetilde\vy=\prox_{\eta \bar g}(\vy) = \argmin_{\vy'} \frac{\beta}{mL_f} \|\vy'\|_1 + \frac{1}{2}\|\vy'-\vy\|^2.
	$$	
Applying \eqref{eq:2dproxg3} to each coordinate of $\widetilde\vy$ above, we have~\eqref{eq:suppg3}
	for $j=1,\dots,3m_2-1$. 
	The proof is then completed.
\end{proof}

Now we are ready to show the following result on how fast $\textnormal{supp}(\bar{\vx}^{(t)})$ and $\textnormal{supp}(\bar{\vy}^{(t)})$ can expand with $t$.

\begin{proposition}
	\label{thm:iterateguess3}
	Suppose an algorithm is applied to the reformulation \eqref{eq:model-spli} of instance~$\mathcal{P}$
	started from an initial solution $\vx^{(0)}=\mathbf{0}$ and $\vy^{(0)}=\mathbf{0}$, and generates a sequence $\{(\vx^{(t)},\vy^{(t)})\}_{t\geq0}$ that satisfies Assumption~\ref{ass:linearspan3}. By notations in~\eqref{eq:t-th-iter} and 
	$\vy^{(t)}=(\vy^{(t)\top}_1, \ldots,\vy^{(t)\top}_{3m_2-1})\zz$ with $\vy^{(t)}_j\in\RR^{{\bar{d}}}$. It holds, for any $\bar{j}\in\{2,3,\dots,\bar{d}\}$,  that 
	\begin{align}
		\label{eq:outerinduction3}
		\textnormal{supp}(\vx_i^{(t)})\subset\{1,\dots,\bar{j}-1\} \text{ and }	\textnormal{supp}(\vy_j^{(t)})\subset\{1,\dots,\bar{j}-1\}
	\end{align}
	$\text{ for }i=1,\dots,m,~j=1,\dots,3m_2-1 \text{ and }t\leq 1+ m(\bar{j}-2)/3.$
\end{proposition}

\begin{proof}
	We prove this claim by induction on $\bar{j}$. 
	Let $\vxi^{(t)}=(\vxi^{(t)\top}_1, \ldots,\vxi^{(t)\top}_{3m_2-1})\zz$ with $\vxi^{(t)}_j\in\RR^{{\bar{d}}}$ and $\vzeta^{(t)}=(\vzeta^{(t)\top}_1, \ldots,\vzeta^{(t)\top}_{3m_2-1})\zz$  with $\vzeta^{(t)}_j\in\RR^{{\bar{d}}}$ defined as in Assumption~\ref{ass:linearspan3} for $t\geq1$.
	Since the algorithm is initialized with $\vx^{(0)}=\mathbf{0}$ and $\vy^{(0)}=\mathbf{0}$, we have $\textnormal{supp}(\nabla f_i(\vx^{(0)}_i))\subset\{1\}$ for any $i$ according to Lemma~\ref{lem:iterateguess}. Notice $\vb=\textbf{0}$ and $\bar{\vb}=\textbf{0}$. By Assumption~\ref{ass:linearspan3} and \eqref{eq:suppg3},  we have $\textnormal{supp}(\vx_i^{(1)})\subset\{1\}$ for any $i$. Meanwhile, we have $\vxi^{(1)}_j=\vzero$ and $\vzeta^{(1)}_j=\vzero$ for any $j$. This implies $\vy_j^{(1)}=\mathbf{0}$.  Thus the claim in \eqref{eq:outerinduction3} holds for $\bar{j}=2$. Suppose that we have proved the claim in \eqref{eq:outerinduction3}  for all $\bar{j}\geq 2$. We next prove it for $\bar{j}+1$. 
	According to the hypothesis of the induction, we have
	\begin{equation}
		\begin{aligned}
			\label{eq:outerinduction3-inner}
			&\textnormal{supp}(\vx_i^{(t)})\subset\{1,\dots,\bar{j}-1\} \text{ and }	\textnormal{supp}(\vy_j^{(t)})\subset\{1,\dots,\bar{j}-1\},
			\\&\forall\, i\in[1,m],\,\,\forall j\in[1,3m_2-1] \text{ and }r\leq \bar{t}:= 1+ m(\bar{j}-2)/3.
		\end{aligned}
	\end{equation}
	
	Below we let $\widehat\vx^{(s)}$ be any vector in $\mathbf{span}\left\{\vA^{\top}\vA\vx^{(s)},\bar\vA^{\top}\bar\vA\vx^{(s)}\right\}$, 
	 $\widetilde \vx^{(s)}=\bar\vA^\top \vy^{(s)}$ and $\widehat\vy^{(s)}=\bar\vA\vx^{(s)}$ for any $s\geq0$, 
and	we  consider two cases: $\bar{j}$ is even and $\bar{j}$ is odd. 
	
	\textbf{Case 1}: Suppose $\bar{j}$ is even.  We claim that, for $s=0,1,\dots,\frac{m}{3}$,
	\begin{eqnarray}
		\label{eq:innerinduction13}
		&\textnormal{supp}(\vx_i^{(r)})\subset\left\{
		\begin{array}{ll}
			\{1,\dots,\bar{j}\},& \text{ if }i\in\left[1, \frac{m}{3}+s\right],\\
			\{1,\dots,\bar{j}-1\},&\text{ if } i\in\left[\frac{m}{3}+s+1, m\right],
		\end{array}
		\right.\,\, \forall r\in[\bar{t}+s]
		\text{ and } \\ \label{eq:innerinduction23}
		&\textnormal{supp}( \vy_j^{(r)})\subset\left\{
		\begin{array}{ll}
			\{1,\dots,\bar{j}\},& \text{ if }j\in\left[1, m_2+\lfloor \frac{s}{m_1}\rfloor \right],\\
			\{1,\dots,\bar{j}-1\},& \text{ if }j\in\left[m_2+\lfloor \frac{s}{m_1}\rfloor+1, 3m_2-1\right],
		\end{array}
		\right.\,\, \forall r\in[\bar{t}+s].
	\end{eqnarray} 
	Notice~\eqref{eq:outerinduction3-inner} implies~\eqref{eq:innerinduction13} and~\eqref{eq:innerinduction23} for $s=0$. 
	Suppose~\eqref{eq:innerinduction13} and~\eqref{eq:innerinduction23} hold for an integer  $s$ satisfying $0\leq s\leq\frac{m}{3}$. According to Lemma~\ref{lem:iterateguess} and $\frac{m}{3}+s \leq \frac{2m}{3}$, 
	$$
	\textnormal{supp}(\nabla f_i(\vx_i^{(r)}))\subset\left\{
	\begin{array}{ll}
		\{1,\dots,\bar{j}\},& \text{ if }i\in\left[1, \frac{m}{3}+s\right],\\
		\{1,\dots,\bar{j}-1\},&\text{ if } i\in\left[\frac{m}{3}+s+1, m\right],
	\end{array}
	\right.\,\, \forall r\in[\bar{t}+s].
	$$
	In addition, by Lemma~\ref{lem:supp2}(a), we have from~\eqref{eq:suppAy} that
	$$
	\textnormal{supp}(\widehat\vx_i^{(r)}),~\textnormal{supp}(\widetilde\vx_i^{(r)}) \subset\left\{
	\begin{array}{ll}
		\{1,\dots,\bar{j}\},&\text{ if } i\in\left[1, \frac{m}{3}+s+1\right],\\
		\{1,\dots,\bar{j}-1\},&\text{ if } i\in\left[\frac{m}{3}+s+2, m\right],
	\end{array}
	\right.\,\, \forall r\in[\bar{t}+s].
	$$
	Hence, by Assumption~\ref{ass:linearspan3}, we have
	$$
	\textnormal{supp}(\vx_i^{(\bar{t}+s+1)})\subset\left\{
	\begin{array}{ll}
		\{1,\dots,\bar{j}\},&\text{ if } i\in\left[1, \frac{m}{3}+s+1\right],\\
		\{1,\dots,\bar{j}-1\},&\text{ if } i\in\left[\frac{m}{3}+s+2, m\right].
	\end{array}
	\right.
	$$
	This means the claim in \eqref{eq:innerinduction13} holds for $s+1$ as well. 
	
	In addition, by the relation in \eqref{eq:suppAbar3}, we have
	$$
	\textnormal{supp}(\widehat\vy_j^{(r)})\subset\left\{
	\begin{array}{ll}
		\{1,\dots,\bar{j}\},& \text{ if }j\in\left[1, m_2+\lfloor \frac{s}{m_1}\rfloor \right],\\
		\{1,\dots,\bar{j}-1\},&\text{ if } j\in\left[m_2+\lfloor \frac{s}{m_1}\rfloor+1, 3m_2-1\right],
	\end{array}
	\right.\,\, \forall r\in[\bar{t}+s].
	$$
	Together with \eqref{eq:innerinduction23} and \eqref{eq:suppAbarAbary}, the inclusion above implies that 
	$$
	\textnormal{supp}(\vxi_j^{(\bar{t}+s+1)})\subset\left\{
	\begin{array}{ll}
		\{1,\dots,\bar{j}\},&\text{ if } j\in\left[1, m_2+\lfloor \frac{s}{m_1}\rfloor \right],\\
		\{1,\dots,\bar{j}-1\},&\text{ if } j\in\left[m_2+\lfloor \frac{s}{m_1}\rfloor+1, 3m_2-1\right].
	\end{array}
	\right.
	$$
	It then follows from~\eqref{eq:suppg3} that
	$$
	\textnormal{supp}(\vzeta_j^{(\bar{t}+s+1)})\subset\left\{
	\begin{array}{ll}
		\{1,\dots,\bar{j}\},&\text{ if } j\in\left[1, m_2+\lfloor \frac{s}{m_1}\rfloor \right],\\
		\{1,\dots,\bar{j}-1\},&\text{ if } j\in\left[m_2+\lfloor \frac{s}{m_1}\rfloor+1, 3m_2-1\right].
	\end{array}
	\right.
	$$
	By Assumption~\ref{ass:linearspan3}, we have
	$$
	\textnormal{supp}(\vy_j^{(\bar{t}+s+1)})\subset\left\{
	\begin{array}{ll}
		\{1,\dots,\bar{j}\},&\text{ if } j\in\left[1, m_2+\lfloor \frac{s}{m_1}\rfloor \right],\\
		\{1,\dots,\bar{j}-1\},&\text{ if } j\in\left[m_2+\lfloor \frac{s}{m_1}\rfloor+1, 3m_2-1\right].
	\end{array}
	\right.
	$$
	This means the claim in \eqref{eq:innerinduction23} holds for $s+1$ as well. By induction, \eqref{eq:innerinduction13} and~\eqref{eq:innerinduction23} hold for $s=0,1,\dots,\frac{m}{3}$. Let $s=\frac{m}{3}$ in them. We have $\textnormal{supp}(\vx_i^{(r)})\subset\{1,\dots,\bar{j}\}$ and $\textnormal{supp}(\vy_j^{(r)})\subset\{1,\dots,\bar{j}\}$ for any $i$, $j$ and $r\leq \bar{t}+\frac{m}{3}= 1+ m(\bar{j}-2)/3+\frac{m}{3}=1+m(\bar{j}-1)/3$. 
	
	\textbf{Case 2}: Suppose $\bar{j}$ is odd.  We claim that, for $s=0,1,\dots,\frac{m}{3}$,
	\begin{eqnarray}
		\label{eq:innerinduction33}
		\textnormal{supp}(\vx_i^{(r)})\subset\left\{
		\begin{array}{ll}
			\{1,\dots,\bar{j}-1\},&\text{ if } i\in\left[1, \frac{2m}{3}-s\right],\\
			\{1,\dots,\bar{j}\},&\text{ if } i\in\left[\frac{2m}{3}-s+1, m\right],
		\end{array}
		\right.\,\, \forall r\in[\bar{t}+s],
		\text{ and }\\\label{eq:innerinduction43}
		\textnormal{supp}( \vy_j^{(r)})\subset\left\{
		\begin{array}{ll}
			\{1,\dots,\bar{j}-1\},&\text{ if } j\in\left[1, 2m_2-\lceil \frac{s}{m_1}\rceil \right],\\
			\{1,\dots,\bar{j}\},&\text{ if } j\in\left[2m_2-\lceil \frac{s}{m_1}\rceil+1, 3m_2-1\right],
		\end{array}
		\right.\,\, \forall r\in[\bar{t}+s].
	\end{eqnarray} 
	Again~\eqref{eq:outerinduction3-inner} implies~\eqref{eq:innerinduction33} and~\eqref{eq:innerinduction43} for $s=0$. 
	Suppose~\eqref{eq:innerinduction33} and~\eqref{eq:innerinduction43} hold for an integer $s$ satisfying $0\leq s\leq\frac{m}{3}$. According to Lemma~\ref{lem:iterateguess} and $\frac{2m}{3}-s\geq \frac{m}{3}$,
	$$
	\textnormal{supp}(\nabla f_i(\vx_i^{(r)}))\subset\left\{
	\begin{array}{ll}
		\{1,\dots,\bar{j}-1\},&\text{ if } i\in\left[1, \frac{2m}{3}-s\right],\\
		\{1,\dots,\bar{j}\},&\text{ if } i\in\left[\frac{2m}{3}-s+1, m\right],
	\end{array}
	\right.\,\, \forall r\in[\bar{t}+s].
	$$
	In addition, by Lemma~\ref{lem:supp2}(a) and~\eqref{eq:suppAy}, we have 
	$$
	\textnormal{supp}(\widehat\vx_i^{(r)}),~\textnormal{supp}(\widetilde\vx_i^{(r)}) \subset\left\{
	\begin{array}{ll}
		\{1,\dots,\bar{j}-1\},&\text{ if } i\in\left[1, \frac{2m}{3}-s-1\right],\\
		\{1,\dots,\bar{j}\},&\text{ if } i\in\left[\frac{2m}{3}-s, m\right],
	\end{array}
	\right.\,\, \forall r\in[\bar{t}+s].
	$$
	Hence, by Assumption~\ref{ass:linearspan3}, we have
	$$
	\textnormal{supp}(\vx_i^{(\bar{t}+s+1)})\subset\left\{
	\begin{array}{ll}
		\{1,\dots,\bar{j}-1\},&\text{ if } i\in\left[1, \frac{2m}{3}-s-1\right]\\
		\{1,\dots,\bar{j}\},&\text{ if } i\in\left[\frac{2m}{3}-s, m\right].
	\end{array}
	\right.
	$$
	This means claim~\eqref{eq:innerinduction33} holds for $s+1$ as well. 
	
	In addition, by \eqref{eq:suppAbar3}, we have
	$$
	\textnormal{supp}(\widehat\vy_j^{(\bar{t}+s+1)})\subset\left\{
	\begin{array}{ll}
		\{1,\dots,\bar{j}-1\},&\text{ if } j\in\left[1, 2m_2-\lceil \frac{s+1}{m_1}\rceil \right],\\
		\{1,\dots,\bar{j}\},&\text{ if } j\in\left[2m_2-\lceil \frac{s+1}{m_1}\rceil+1, 3m_2-1\right].
	\end{array}
	\right.
	$$
	Together with \eqref{eq:innerinduction43}, the inclusion above implies that 
	$$
	\textnormal{supp}(\vxi_j^{(\bar{t}+s+1)})\subset\left\{
	\begin{array}{ll}
		\{1,\dots,\bar{j}-1\},&\text{ if } j\in\left[1, 2m_2-\lceil \frac{s+1}{m_1}\rceil \right],\\
		\{1,\dots,\bar{j}\},&\text{ if } j\in\left[2m_2-\lceil \frac{s+1}{m_1}\rceil+1, 3m_2-1\right].
	\end{array}
	\right.
	$$
	It then follows from~\eqref{eq:suppg3} that
	$$
	\textnormal{supp}(\vzeta_j^{(\bar{t}+s+1)})\subset\left\{
	\begin{array}{ll}
		\{1,\dots,\bar{j}-1\},&\text{ if } j\in\left[1, 2m_2-\lceil \frac{s+1}{m_1}\rceil \right],\\
		\{1,\dots,\bar{j}\},&\text{ if } j\in\left[2m_2-\lceil \frac{s+1}{m_1}\rceil+1, 3m_2-1\right].
	\end{array}
	\right.
	$$
	By Assumption~\ref{ass:linearspan3}, we have
	$$
	\textnormal{supp}(\vy_j^{(\bar{t}+s+1)})\subset\left\{
	\begin{array}{ll}
		\{1,\dots,\bar{j}-1\},&\text{ if } j\in\left[1, 2m_2-\lceil \frac{s+1}{m_1}\rceil \right],\\
		\{1,\dots,\bar{j}\},&\text{ if } j\in\left[2m_2-\lceil \frac{s+1}{m_1}\rceil+1, 3m_2-1\right],
	\end{array}
	\right.
	$$
	which means \eqref{eq:innerinduction43} holds for $s+1$ as well. By induction, \eqref{eq:innerinduction33} and~\eqref{eq:innerinduction43} holds for $s=0,1,\dots,\frac{m}{3}$. Let $s=\frac{m}{3}$ in~\eqref{eq:innerinduction43}. We have $\textnormal{supp}(\vx_i^{(r)})\subset\{1,\dots,\bar{j}\}$ and $\textnormal{supp}(\vy_j^{(r)})\subset\{1,\dots,\bar{j}\}$ for any $i$, $j$ and $r\leq \bar{t}+\frac{m}{3}= 1+ m(\bar{j}-2)/3+\frac{m}{3}=1+m(\bar{j}-1)/3$. 
	
	Therefore, we have proved that~\eqref{eq:outerinduction3-inner} holds for $\bar{j}+1$, when $\bar{j}$ is either even or odd. By induction,~\eqref{eq:outerinduction3-inner} holds for any integer $\bar{j}\in[2,\bar{d}]$, and we complete the proof.	
\end{proof}

Now, we are ready to prove Theorem \ref{thm:lbcomposite}.


\begin{proof}[of Theorem \ref{thm:lbcomposite}]
	As we discussed below~\eqref{eq:t-th-iter}, we assume $\vx^{(0)}=\mathbf{0}$ and $\vy^{(0)}=\mathbf{0}$ without loss of generality.  Thus by notation in~\eqref{eq:t-th-iter}, Proposition~\ref{thm:iterateguess3} indicates that 	$\textnormal{supp}(\vx_i^{(t)})\subset\{1,\dots,\bar{d}-1\}$ and $\textnormal{supp}(\vy_j^{(t)})\subset\{1,\dots,\bar{d}-1\}$ for $i=1,\dots,m$ and $j=1,2,\ldots,3m_2-1$ for all $t\leq 1+ m(\bar{d}-2)/3$, which means $[\bar{\vx}^{(t)}]_{\bar{d}}=0$ if  $t\leq 1+ m(\bar{d}-2)/3$. 
	Hence, by  Lemmas~\ref{lem:nablaf} and~\ref{lem:kktvio}, we have
	$$
	\max\left\{\left\|\vH\vx^{(t)}\right\|, \min _{\vgamma}\left\|\nabla  f_0(\vx^{(t)})+\vH^{\top} \vgamma\right\|\right\} \geq \frac{\sqrt{m}}{2}\left\|\frac{1}{m}\sum_{i=1}^m\nabla f_i(\bar{\vx}^{(t)})\right\|>\epsilon,\,\,\forall t\leq 1+ m(\bar{d}-2)/3.
	$$ 						
	Hence, $\vx^{(t)}$ is not an $\epsilon$-stationary point of problem~\eqref{eq:model3} if $t\leq m({\bar{d}}-1)/3$. Thus, $(\vx^{(t)},\vy^{(t)})$
	cannot be an $\epsilon/2$-stationary point of the reformulation~\eqref{eq:model-spli} of instance $\cP$ according to Corollary~\ref{cor:kktequiv3}. Moreover, by 				
	Lemma~\ref{cor:boundf}(a) and the facts that $\inf_\vy\bar g(\vy)=\bar g(\vy^{(0)})=0$, $\inf_{\vx,\vy} [f_0\left(\vx\right)+\bar g\left(\vy\right)]\geq \inf_{\vx} f_0\left(\vx\right)$, it holds that
	$$ {\bar{d}}\geq \frac{L_f \left( f_0(\vx^{(0)})-\inf_{\vx} f_0\left(\vx\right)\right)}{3000\pi^2} \epsilon^{-2}
	\geq \frac{L_f \left( f_0(\vx^{(0)})+\bar g(\vy^{(0)})-\inf_{\vx,\vy} \left[f_0\left(\vx\right)+\bar g\left(\vy\right)\right]\right)}{3000\pi^2} \epsilon^{-2}
	=\frac{L_f\Delta_F}{3000\pi^2}\epsilon^{-2}.
	$$
	In other words, in order for $(\vx^{(t)},\vy^{(t)})$ to be an $\epsilon/2$-stationary point of~\eqref{eq:model-spli}, 
	the algorithm needs at least $t= 2+m({\bar{d}}-2)/3$ oracles. Notice
	\begin{equation}
		\label{eq:statement}
		2+m({\bar{d}}-2)/3\geq m{\bar{d}}/6\geq\frac{m L_f \Delta_F}{18000\pi^2} \epsilon^{-2}
		>\frac{\kappa([\bar{\vA}; \vA]) L_f \Delta_F}{18000\pi^2}\epsilon^{-2},
	\end{equation}
	where the second inequality is because ${\bar{d}}\geq 5$ and the last inequality is by Lemma~\ref{lem:condH}. The conclusion is then proved by replacing $\epsilon$ in \eqref{eq:statement} to $2\epsilon$.
\end{proof}

Finally, we give the proof to Corollary \ref{cor:lower2}.


					\begin{proof}[of Corollary \ref{cor:lower2}]
	As we discussed in Section \ref{sec:extension}, 	we assume $\vx^{(0)}=\mathbf{0}$ and $\vy^{(0)}=\mathbf{0}$ without loss of generality. 
	Thus by notation in~\eqref{eq:t-th-iter}, Proposition~\ref{thm:iterateguess3} indicates that 	$\textnormal{supp}(\vx_i^{(t)})\subset\{1,\dots,\bar{d}-1\}$ for any $i \in [1,m]$ and any $t\leq 1+ m(\bar{d}-2)/3$, which means $[\bar{\vx}^{(t)}]_{\bar{d}}=0$ if  $t\leq 1+ m(\bar{d}-2)/3$, where $\bar{\vx}^{(t)}= \frac{1}{m} \sum_{i=1}^m \vx_i^{(t)}$.  
	
	On the other hand, suppose $\vx^*$ with the structure as in~\eqref{eq:t-th-iter} is an $\epsilon$-stationary point of instance~$\mathcal{P}$. Then by 	Lemma~\ref{cor:kktvio2}, it must also be an $\epsilon$-stationary point of~\eqref{eq:model3}. Hence, by  Lemmas~\ref{lem:nablaf} and~\ref{lem:kktvio}, we have $\left|[\bar\vx^*]_j\right| \ge \frac{150\pi\epsilon}{\sqrt{m} L_f}$ for all $j=1,\ldots,\bar d$, where $\bar{\vx}^*= \frac{1}{m} \sum_{i=1}^m \vx_i^*$.  Therefore, by the convexity of the square function, it follows that
	$$\|\vx^{(t)} - \vx^*\|^2 \ge \sum_{i=1}^m \left([\vx_i^{(t)}]_{\bar d} - [\vx_i^*]_{\bar d}\right)^2 \ge m \left([\bar\vx^{(t)}]_{\bar d} - [\bar\vx^*]_{\bar d}\right)^2 \ge m \left(\frac{150\pi\epsilon}{\sqrt{m} L_f}\right)^2 > \omega^2,$$
	and thus $\vx^{(t)}$ is not $\omega$-close to $\vx^*$ if  $t\leq 1+ m(\bar{d}-2)/3$.

	Moreover, by 				
	Lemma~\ref{cor:boundf}(a) and the fact that $g(\vx^{(0)})=0$ and $g(\vx) \ge 0, \forall\, \vx$, it holds that
	$$ {\bar{d}}\geq \frac{L_f \left( F_0(\vx^{(0)})-\inf_{\vx} F_0\left(\vx\right)\right)}{3000\pi^2} \epsilon^{-2}=\frac{L_f\Delta_{F_0}}{3000\pi^2}\epsilon^{-2}.
	$$
	In other words, in order for $\vx^{(t)}$ to be $\omega$-close to an $\epsilon$-stationary point of instant~$\mathcal{P}$, the algorithm needs at least $t=2+m({\bar{d}}-2)/3$ oracles.
The proof is then completed, by observing	
					\begin{equation*}
		2+m({\bar{d}}-2)/3\geq m{\bar{d}}/6\geq\frac{m L_f \Delta_{F_0}}{18000\pi^2} \epsilon^{-2}>\frac{\kappa([\bar{\vA}; \vA]) L_f \Delta_{F_0}}{18000\pi^2},
	\end{equation*}
	where the first inequality is because ${\bar{d}}\geq 5$, and the last one is by Lemma~\ref{lem:condH}. 
\end{proof}

			\end{document}